\documentclass[11pt, oneside, a4paper]{amsart}
 \usepackage{geometry}
\usepackage{hyperref}
\usepackage[utf8]{inputenc}
\usepackage{enumerate}
\usepackage[english]{babel}
\usepackage[T1]{fontenc}
\usepackage{ tipa }
\usepackage[utf8]{inputenc}
\usepackage{amsmath, amssymb,amsthm}
\usepackage{array}
\usepackage{graphicx}
\usepackage{float}
\usepackage[normalem]{ulem} % 导入 ulem 宏包并保留原来的 \emph 命令
\usepackage{caption,subcaption}
\usepackage{tikz}
\usepackage{tikz-cd}
\usetikzlibrary{graphs}
\usepackage{ upgreek }
\usetikzlibrary{matrix,arrows,decorations.pathmorphing}
\usepackage{xfrac} 
\usepackage{mathrsfs}
\usepackage{physics}
\usepackage{faktor}

\makeatletter
\newcommand*{\rom}[1]{\expandafter\@slowromancap\romannumeral #1@}
\makeatother

\usepackage{graphicx}
\newtheorem{thm}{Theorem}[section]
\newtheorem{lemma}[thm]{Lemma}
\newtheorem{example}[thm]{Example}
\newtheorem{rmk}[thm]{Remark}
\newtheorem{defi}[thm]{Definition}
\newtheorem{cor}[thm]{Corollary}
\newtheorem{prop}[thm]{Proposition}

\newtheorem{ThmInt}{Theorem}

\newtheorem*{DefInt}{Definition}
\newtheorem*{ExInt}{Example}

\def\A{\mathbb{A}}

\def\C{\mathbb{C}}
\def\D{\mathbb{D}}

\def\H{\mathbb{H}}

\def\L{\mathbb{L}}

\def\N{\mathbb{N}}

\def\P{\mathbb{P}}

\def\R{\mathbb{R}}

\def\Z{\mathbb{Z}}

\def\la{\lambda}
\def\La{\Lambda}
\def\om{\omega}

\def\P{\mathbb{P}}

\def\cC{\mathcal{C}}

\def\cE{\mathcal{E}}

\def\cO{\mathcal{O}}

\def\cR{\mathcal{R}}

\def\sH{\mathscr{H}}

\def\eps{\varepsilon}

\def\na{\mathrm{na}}
\def\tk{\tilde{k}}
\def\xg{x_g}

\DeclareMathOperator{\FL}{FL}

\DeclareMathOperator{\eva}{ev}

\DeclareMathOperator{\car}{char}

\DeclareMathOperator{\an}{an}

\DeclareMathOperator{\Rat}{Rat}
\DeclareMathOperator{\ovRat}{\overline{Rat}}
\DeclareMathOperator{\rat}{rat}

\DeclareMathOperator{\res}{\mathrm{res}}

\DeclareMathOperator{\spec}{\mathrm{Spec}}
\DeclareMathOperator{\PGL}{\mathrm{PGL}}
\DeclareMathOperator{\PSL}{\mathrm{PSL}}

\DeclareMathOperator{\id}{\mathrm{id}}
\DeclareMathOperator{\Per}{\mathrm{Per}}

\usepackage{mathtools}

\title{Generalized rescaling limits of a sequence of rational maps}

\author{Charles Favre}
\address{CNRS - Centre de Math\'ematiques Laurent Schwartz, 
	\'Ecole Polytechnique, 
	91128 Palaiseau Cedex, France}
\email{\href{charles.favre@polytechnique.edu}{charles.favre@polytechnique.edu}}

\author{Chen Gong}
\address{Institute for Theoretical Sciences of Westlake University,  Westlake University, 310030 Hangzhou, China}
\email{\href{gongchen@westlake.edu.cn}{gongchen@westlake.edu.cn}}

\date{\today}
\thanks{C.G. is  supported by a CSC-202108070159 grant  from the chinese government.}

\begin{document}

\begin{abstract}
We consider a sequence  of complex rational maps $(f_n)$ of a fixed degree $d\ge2$. 
Building on the seminal work of Kiwi, we introduce the notion of generalized rescaling limits.
These are rational maps possibly defined over a non-Archimedean field
obtained by renormalizing at some scale a fixed iterate of the sequence $(f_n)$. 
We explain that the set of all generalized rescaling limits is naturally 
organized as a tree, and bound the size of this tree in term of the degree $d$. 

We apply our theory to quadratic rational maps. Using Kiwi's classification, we describe all 
possible trees in this case, and prove a uniform bound on the number of cycles with 
small multipliers. 
\end{abstract} 

\maketitle

\setcounter{tocdepth}{1}
\setcounter{secnumdepth}{4}

\tableofcontents

%\newpage 

\section*{Introduction}

Let $(f_n)$ be  a sequence of complex rational maps of degree $d\ge2$ that degenerates
in the moduli space in the sense that there is no sequence of Möbius transformations $M_n$ for which
$g_n:= M_n\cdot f_n\cdot M_n^{-1}$ converges to a degree $d$ complex rational map. 
 This condition does not prevent $g_n$ nor any of its iterate to converge to a complex rational map of lower degree, in which case interesting
 informations can be extracted on the dynamics of $f_n$ for $n$ large from the dynamical properties of $g$.
 Such limits were first studied by Stimson~\cite{ST93} and Epstein~\cite{Ep00} in the case of quadratic rational maps. The term rescaling limits was subsequently coined by Kiwi in his seminal work~\cite{zbMATH06455745KJ15} in which he proved the finiteness of non post-critically finite
(PCF) rescaling limits.
Applications of this theory were found by Nie--Pilgrim~\cite{NP20,NP22} who proved the boundedness of certain hyperbolic components.  More recently, Kiwi~\cite{Ki25} provided an estimate for the genus of the curve of quadratic rational maps with a periodic critical point.

In a dual perspective, Luo~\cite{luo2021trees,YLuo22} explained how to naturally attach to $(f_n)$ a limiting degree $d$ rational map defined over a complete non-Archimedean metrized field. This generalizes
the pioneering work of Kiwi~\cite{Ki06} which was further exploited and extended to other situations by 
many authors including~\cite{
DMM08,DF14,DF16,zbMATH07219254CFdegeneration,YLuo22}: in  case $(f_t)$ is a holomorphic family of rational maps
parameterized by the punctured unit disk which is meromorphic at $t=0$, and $f_n=f_{t_n}$ for a sequence $t_n\to 0$, then 
this limit is the map naturally induced by the 
family over the non-Archimedean field $\C((t))$ endowed with the $t$-adic norm.

Luo's method uses ultralimits of the hyperbolic $3$-space, following the approach developed by Bestvina~\cite{Bes88} and  Paulin~\cite{Pau88} 
in the context of Kleinian groups. He deduced from his construction  the appearance of monomial rescaling limits under suitable conditions, and 
gave a characterization of hyperbolic components containing a diverging sequence with bounded multipliers.
The authors~\cite{FG25} and more recently Poineau~\cite{poi25} and Roy~\cite{Roy25} revisited his construction using Berkovich techniques in the spirit of Morgan-Shalen approach~\cite{MS85} allowing for more flexibility. Note that 
these degeneration techniques have already proven to be a powerful tool in addressing uniformity problems in arithmetic dynamics, see, e.g.,~\cite{DHY20,poi25b,Yap25} and Diophantine geometry, see~\cite{LY26}. We also refer the reader to~\cite{Hul26} for a similar approach from the perspective of model theory.

Our aim is to define a notion of $g$-rescaling limit\footnote{$g$ stands for generalized} encompassing all these limiting procedures, 
and to develop a general framework in order to take into account 
all possible limits of the iterates of $(f_n)$ whether defined over an Archimedean or  a non-Archimedean field. 
We  prove that these limits can be naturally organized into a simplicial tree.
\[\diamond\]
Let us now explain in more details our main results.  

The first difficulty to overcome is that $(f_n)$ could be dense in the space $\Rat_d$  of rational map of degree $d$. It is thus necessary to extract a subsequence in order to have a chance to get limits.  Furthermore, we want the limit to capture non trivial dynamical contents.  In order to do so, one needs  to extract a subsequence of $(f_n)$
that converges \emph{simultaneously} in all projective compactifications of $\Rat_d$. 
Ultrafilters turn out to be the adequate tool to canonically extract converging subsequences in all compact set at the same time in a compatible way, see~\cite[\S{3}]{Gol22}.

We thus shall fix once and for all a non-principal ultrafilter $\om$ of the set of integers. It is an element of the power set of $\N$ whose elements are called $\om$-big, and
allows one to take ultralimits.
In particular, the sequence $(f_n)$ has a limit along $\om$ in \emph{any} projective compactification of $\Rat_d$. 

  A crucial aspect of our theory is the analysis of the degeneration of $(f_n)$ at multiple scales. This fact was noticed by Luo, see~\cite[Remark, p.3]{luo2021trees} and the discussion on~\cite[p.3805]{Luo24}, and naturally appears in studying degenerations of Riemann surfaces, see~\cite{BCGG24,NO22,NO25}.
  
  Formally, a scale\footnote{a similar but more restrictive notion is used in the recent paper~\cite{Hel25}} is determined by 
a bounded sequence of positive real numbers, and we identify two sequences $\eps_n, \eta_n$ when the limit along $\om$ of their
quotients $\eps_n/\eta_n$ is finite and non-zero. The set of scales is naturally totally ordered so that for instance
$(e^{-n}) < (1/n) < (1/\log n) < (1)$. It admits a unique maximal element $(1)$, which we refer to as the trivial scale. 
It is convenient to include the sequence $(0)$ as an extended scale which becomes a minimum among all scales. 

An (extended) scale $\eps$ determines an algebraically closed field $\sH_\eps$ defined by a variation on the construction of ultraproducts (see~\cite{zbMATH05768468} for a formal definition of this notion  and its extensive applications to commutative algebra).
By definition, $\sH_\eps$ is the quotient of the set of sequences of complex numbers $(z_n)$ such that $\lim_\om |z_n|^{\eps_n} <\infty$
by the ideal of sequences such that $\lim_\om |z_n|^{\eps_n} =0$. The function $|(z_n)|:= \lim_\om |z_n|^{\eps_n}$ defines a norm
for which $\sH_\eps$ is complete. 
When $\eps=(1)$, then $\sH_\eps$ is isomorphic to $\C$; when $(0) < \eps < (1)$, then $\sH_\eps$ is a spherically complete non-Archimedean field;
and when $\eps =(0)$, then it is a trivially valued complex analog of a model of the Robinson field of hyperreals. 

Let us write $f_n=[P_n:Q_n]$ in projective coordinates where $P_n, Q_n$ are homogeneous polynomials normalized so that the maximum of their coefficients have norm $1$.  
Given a scale $\eps$, each coefficient of $P_n$ and $Q_n$  determines an element in $\sH_\eps$, and 
in this way we get  a limit $f_\eps$ which is a (possibly constant) rational map of degree $\le d$  defined over $\sH_\eps$. 
We can now define a workable notion of $g$-rescaling. 

\begin{DefInt}
A restricted $g$-rescaling adap\-ted to $(f_n)$ is a pair of a sequence of Möbius transformations $M=(M_n)$ and a scale $\eps>(0)$ such that 
$(M \cdot f \cdot M^{-1})_\eps$ is a rational map of positive topological entropy. 
\end{DefInt}

When $\eps =(1)$,  positive entropy is equivalent to have degree at least $2$. When $(0) < \eps < (1)$, the condition
is stronger and is equivalent to say that the map has degree $\ge 2$ and does not have potential good reduction, see~\cite{CRL26}. 
A rescaling in the sense of Kiwi is a restricted $g$-rescaling with $\eps = (1)$. 
We shall say that two restricted $g$-rescalings $(M,\eps)$ and $(L,\eta)$ are equivalent iff $\eps= \eta$ and $(L^{-1}\cdot M)_\eps$
has degree $1$.

It is also important to consider $g$-rescalings possibly having degree $1$ limits, 
but the definition is more technical.  We refer to \S\ref{sec:g-rescaling} for the general definition.

With this terminology, we prove the basic fact that there exists a \emph{unique} adapted $g$-rescaling whose limit has degree $d$. 
We call it the fundamental $g$-rescaling of $(f_n)$. One can obtain it either through Luo's construction or
using Berkovich theory~\cite{FG25,Roy25}. The associated scale is obtained  by minimizing the normalized resultant
of $f_n$ on its $\PGL(2)$-orbit in $\Rat_d$, see \S\ref{sec:funda-scale}.

The following example features a fundamental rescaling, a Kiwi rescaling, and a $g$-rescaling that belongs to neither category, see Example~\ref{eg: ezchain} for details.

\begin{ExInt}
Consider $f_n(z)=e^{-n}z^4+n^{-1}z^3+z^2$. Here, $(\mathrm{id}, (1/n))$ is the fundamental rescaling and $(\mathrm{id}, (1))$ is a Kiwi rescaling whose limit is $z^2$. In contrast, $(\mathrm{id}, (1/\log n))$ is a $g$-rescaling that is neither fundamental nor Kiwi: its limit is a degree $3$ polynomial $a z^3+z^2$ defined over a non-Archimedean field with $|a|<1$.
\end{ExInt}

To have a full-fledged theory of $g$-rescalings, it is important to include $g$-rescalings adapted to iterates of $f_n$ which leads to the notion
of cycles of $g$-rescalings adapted to $(f_n)$. Such a  cycle is determined by a scale $\eps$, and a sequence of $g$-rescalings
$(M_{i,n}, \eps)$ which is periodic of period 
$l$
so that 
$M_{i+l,n}=M_{i,n}$
for all $i,n$. They are required to  satisfy the following conditions for all $i$:
\begin{itemize}
\item
$(M_{i+1} \cdot f \cdot M_i^{-1})_\eps$ has degree at least $1$;
\item
$(M_i,\eps)$ is adapted to 
$(f_n^l)$.
\end{itemize} 
We refer to \S\ref{sec: cycles} for more details. We call $(M_i \cdot f^l \cdot M_i^{-1})_\eps$ the (rescaling) limit associated with the cycle
even though it might depend on $i$. Note that all limits are semi-conjugated, therefore share many dynamical properties, e.g.,
being PCF.

We denote by $\cR_+(f)$ the set of equivalence classes of cycles of restricted $g$-rescalings adapted to $(f_n)$, and  endow $\cR_+(f)$ with a partial order $(M_*,\eps) \le (L_*, \eta)$
iff $\eps \le \eta$ and $(M_*,\eps)$ is equivalent to $(L_*, \eps)$. 
Shifting indices \[\sigma_f(M_*,\eps)=(M_{*+1},\eps)\] determines a bijective map on $\cR_+(f)$
for which every point has a finite orbit, and preserving the partial order.

\smallskip

The next theorem summarizes the main properties of the set $\cR_+(f)$.

\begin{ThmInt}~\label{thmint:1}
\begin{enumerate}
\item
The poset $\cR_+(f)$ is a simplicial tree having the fundamental rescaling as its unique minimal element. 
\item
Any increasing interval of $\cR_+(f)$ is of cardinality at most $2d-1$. 
\item
The set $\cR'_+(f)$ of cycles of $g$-rescalings 
whose rescaling limits are
not PCF  forms a finite subtree of $\cR_+(f)$. The set of $\sigma_f$-orbits of $\cR'_+(f)$ is finite of cardinality is bounded from above by a constant depending only on $d$.  
\end{enumerate}
\end{ThmInt}

The proof of (1) is constructive in the sense that we explain how to build $\cR_+(f)$ starting from the fundamental limit. This works as follows.
Consider a limit $g$ defined over $\sH_\eps$ associated with some cycle of $g$-rescaling adapted to 
$(f_n)$. If $\eps<(1)$, then $\sH_\eps$ is a non-Archimedean field, and we show that any periodic type II point $x$ of $g$ whose local degree is $\ge2$
determines a new cycle of $g$-rescalings that we call the baby $g$-rescaling associated with $x$. 
We proceed with proving that any $g$-rescaling is obtained from the fundamental one by taking successive baby $g$-rescalings 
in a unique fashion which gives the tree structure to $\cR_+(f)$.

The second statement follows from the fact that the set of scales $\eps$ for which there exists a cycle of $g$-rescalings $(M_*,\eps)$ adapted to $(f_n)$
is finite and has at most $2d-1$ elements. The proof of this fact is algebraic in nature, and follows the argument of~\cite{CG25} based on valuation theory (Abhyankar's inequalities). Note that this bound is optimal,  as we show in \S\ref{sec: optimal}.

The third statement is a consequence of McMullen's rigidity theorem~\cite{zbMATH04032044,zbMATH07691681}. 
All baby $g$-rescalings of a PCF limit give rise  to  either  monomial or Chebyshev limits. 
This implies $\cR'_+(f)$ to be connected.
The finiteness of $\cR'_+(f)/\langle \sigma_f\rangle$ is the generalization to our setting of the main theorem of~\cite{zbMATH06455745KJ15}. 
Following the strategy of Kiwi, we show that any non PCF $g$-rescaling limit attracts a critical point
of the sequence $(f_n)$. The bound on the size of $\cR'_+(f)$ then follows by combining this information with (2). 
\[\diamond\]

Our previous result can be seen through the lenses of renormalization theory. Recall that renormalization is an essential tool in dynamics, see~\cite{Av11} and~\cite{BGH25}, 
and more specifically in the dynamics of complex polynomials. 
We refer to the books of McMullen~\cite{mcmullen1994complex}, to the notes by Lyubich,~\cite{Lyu23}, or to the recent survey by Dudko~\cite{Dud25} for a thorough discussion of this theory and its far-reaching applications in holomorphic dynamics. 

An important observation is that we can encapsulate the degeneration of a sequence of rational map $(f_n)$ along an ultrafilter $\omega$ into 
a single degree $d$ rational map $f$ defined over the ultraproduct $\sH_0$ of the field of complex numbers mentioned above.

Cycles of $g$-rescalings adapted to $(f_n)$ are obtained by iterating and zooming $f$ at some scale, so that we can loosely interpret them as a form of renormalization of $f$. 
Since renormalization theory has proved to be an essential tool in understanding bifurcations, we expect the notion of $g$-rescalings to play an analogous role in the study of bifurcation phenomena at infinity in the moduli space of rational maps. 
Results by Luo~\cite{Luo22,Luo24} and Kiwi~\cite{Ki25} are first instances of this idea.

\[\diamond\]

Our theory of $g$-rescalings is sufficiently concrete to allow explicit calculations: explicit examples are detailed in the main body of the paper,
see Examples \ref{eg: compute-fundascale}, \ref{ex:concrete}, \ref{eg: deg1},~\ref{eg: ezchain}, \ref{eg: trivialtree}, \ref{ex:cycle} and \S\ref{sec:example}.

Suppose that  
$(f_n)$ is induced from a meromorphic family
$(f_t)$ parameterized by the unit disk so that $f_n = f_{t_n}$ for a sequence $t_n\to 0$. 
We prove that any adapted $g$-rescaling that is not the fundamental one is defined over the trivial scale (1)
hence is a rescaling in the sense of Kiwi. In that case, our results bring nothing new to~\cite{zbMATH06455745KJ15}. 

When $(f_n)$ is a sequence of quadratic rational maps, then we may
build on the classification of
their dynamical properties in the non-Archimedean case due to Kiwi~\cite{Ki14}, 
and give a classification of all possible trees of $g$-rescalings (Theorem~\ref{thm: quadratic}).

As a consequence of this classification, we obtain the following
bound on the existence of periodic cycles having small multipliers. 

\begin{ThmInt}\label{thm:int3}
For every $L \in \mathbb{N}^*$, there exists 
$\varepsilon=\eps(L) > 0$
such that the number of repelling periodic cycles $p, \cdots, f^{l-1}(p)$ of 
period
$l\le L$ whose multiplier satisfies
\[
\left| (f^l)'(p) \right| < (1+\varepsilon)^l
\]
is at most $4$ for all $f\in\Rat_2(\C).$
\end{ThmInt}
Our proof does not give any bound on $\eps(L)$. The Feigenbaum polynomial has periodic cycles of  period $2^l$ for all $l$ with uniformly bounded multipliers, so that 
one necessarily have $\eps(L) \lesssim 1/L$. Recall that a rational map
$f$ having no small cycles in the sense that there exists $\eps>0$ such that 
$\left| (f^l)'(p) \right|^{1/l} > (1+\varepsilon)$ for all $p$ of period $l$
 is topological Collet-Eckmann, see~\cite{RSP}.

\[\diamond\]
The plan of our paper is as follows. In \S\ref{sec:gen}, we collect some facts in non-Archimedean dynamics.
This section also contains a discussion on scales and their associated fields. 

Section~\ref{sec:seq} contains a preliminary discussion on the family of rational maps determined by $(f_n)$ when the scale is varying from $(0)$ to $(1)$. The important notion of hole is introduced following~\cite{zbMATH05004325demarcoiteration}, and its impact on critical points and on the multipliers of periodic points for $f$ at all scales is detailed. 

The notion of fundamental scale is studied in the short Section~\ref{sec:funda-scale}. 

We introduce the formal definition of $g$-rescalings adapted to $(f_n)$ in \S\ref{sec:g-rescaling}. We explain (Theorems~\ref{thm:baby2} and~\ref{thm:baby1}) how to build new $g$-rescalings from type II periodic points of the limits of a given $g$-rescaling. 
Theorem~\ref{thm:chain} is essential to the proof of the tree structure of the set of $g$-rescalings to be given in the next section. 

In Section~\ref{sec:tree}, we give the definition of cycles of $g$-rescalings, and prove the tree structure of the set $\cR(f)$ of adapted $g$-rescalings up to equivalence (Theorem~\ref{thm:tree-rf}). We then proceed with the bound on the set of scales of adapted $g$-rescalings (Theorem~\ref{thm:boundscales}). 

The next Section~\ref{sec:more-on-tree} contains deeper informations on $\cR(f)$ and completes the proof of Theorem~\ref{thmint:1}. 
A first key result is Theorem~\ref{thm:periodic} which shows that a cycle of $g$-rescalings adapted to an iterate of $f$ induces a cycle of $g$-rescaling adapted to $f$. 
We then describe $\cR(f)$ when $(f_n)$ is induced by a meromorphic family (Theorem~\ref{thm:family}). With all this material, we are able to bound the size of the subtree of $\cR(f)$ of cycles of $g$-rescalings whose limit is not PCF (Theorem~\ref{thm:anti-chain}).

We conclude our article with a series of examples (Section~\ref{sec:example}).  We classify all trees of $g$-rescalings in the case of quadratic rational maps based on~\cite{Ki14}, and deduce Theorem~\ref{thm:int3} from this analysis. We add a construction of a degenerating sequence of rational maps for which our bound on the number of adapted scales is optimal.

\medskip

\subsection*{Acknowledgements}
Both authors thank Y. Luo for sharing his insights on degeneration of rational maps, and X. Buff for discussions on small multipliers. 
This article is an expanded and more precise version of the second part of thesis of the second author. She extends her thanks to the referees of her Phd thesis J. Kiwi and J. Poineau
for their careful reading and their remarks.  She also thanks Z. Ji and J.-W. Yap for helpful discussions on the content of this paper.

%%%%%%%%%%%%
\section{Generalities}\label{sec:gen}
In this section, we fix some terminology and notation. We recall some basics on the dynamics of rational maps in metrized fields, see~\cite{zbMATH07045713benedetto};
and on scales as discussed in~\cite{FG25,CG25}. 
Our metrized fields are always algebraically closed, complete, and of residual characteristic $0$ (when non-Archimedean).

\subsection{Rational maps on metrized fields}

\subsubsection{Metrized fields}
Let $k$ be an algebraically closed field endowed with a complete multiplicative norm. 
We shall always assume that the norm is non-trivial, i.e., there exists $a\in k$ such that $|a|\not=1$. 
When the norm is Archimedean, then $k$ is isometric to $(\C, |\cdot|^\eps)$ for some $0< \eps \le 1$, where $|\cdot|$ denotes the standard Euclidean norm on $\C$.

When the norm is non-Archimedean, i.e., $|x+y|\le \max\{|x|, |y|\}$ for all $x,y$, then we denote by $k^\circ = \{|x| \le 1\}$
its ring of integers, $k^{\circ\circ} = \{|x|<1\}$ and $\tk = k^{\circ}/k^{\circ\circ}$ its residue field. 

We shall always assume that the characteristic of $\tk$ is $0$. This implies $k$ to be of characteristic $0$ as well. 
It is spherically complete if the intersection of any decreasing sequence of closed balls is non-empty.
We shall also impose all our fields to be spherically complete.

\subsubsection{Berkovich projective line}
Let $k$ be a complete metrized field as above. The Berkovich projective line
$\P^{1,\an}_k$ is the space of pairs $(\xi,|\cdot|)$ where $\xi$ is a scheme-theoretic point 
in $\P^1_k$ and $|\cdot|$ is a norm on its residue field. It is endowed with a canonical compact topology 
rendering all evaluation maps continuous $x \mapsto |P(x)|\in [0,+\infty]$ for all rational function $P$. 

Once an affine coordinate $z$ is fixed so that  $\P^1_k = \spec k[z] \cup \spec k[z^{-1}]$, then $\P^{1,\an}_k$ can be identified with the space $\A^{1,\an}_k $ of all multiplicative semi-norms 
on $k[z]$ restricting to the standard norm on $k$, together with the function sending any non-constant polynomial to $\infty$. We write $\P^{1,\an}_k = \A^{1,\an}_k \cup \{\infty\}$. 
Note that a point in $\P^{1,\an}_k$ is determined by its values on polynomials, and further since $k$ is algebraically closed on degree $1$ polynomials. 

The Berkovich projective line is also endowed with a canonical sheaf of analytic  functions, see~\cite{berkovich2012spectral}
for a definition. 
When $k=(\C,|\cdot|_\infty^\eps)$ is Archimedean, then $\P^{1,\an}_k$ is canonically isomorphic to the Riemann sphere. 
We shall restrict our attention to the case $k$ is non-Archimedean. Then  $\mathbb{P}^{1,\an}_k$  carries the structure of a compact \( \mathbb{R} \)-tree:  
any two points $x,y$ can be joined by a unique continuous injective map $[0,1]\to [x,y]\subseteq \P^{1,\an}_k$ up to reparameterization. 
A  point $x \in \P^{1,\an}_k$ belongs to one of the following class of examples. 

If there exists a rational function $R$ such that $|R(x)|=0$, then we say that $x$ is rigid or of type I. In that case $x$ can be identified with a point in $\P^1(k)$, and the induced semi-norm
is given on rational functions by the evaluations map $R\mapsto |R(x)|$. Any rigid point $x$ is an endpoint of the tree in the sense that $\P^{1,\an}_k\setminus\{x\}$ is still connected.
The set of rigid points is dense in $\P^{1,\an}_k$.

For any $a\in k$ and $r>0$, we denote by $\zeta(a,r)$ the point in $\P^{1,\an}_k$ defined by 
$|P(\zeta(a,r))|:= \sup_{|z-a|\le r} |P(z)|$ for any polynomial $P$. 
The point $\zeta(0,1)$ corresponding to the closed unit ball bears a special name
and is called the Gauss point. We denote it by $\xg$.

When $r$ belongs to the value group of $k$, we say 
$\zeta(a,r)$ is of type II. Otherwise it is of type III.  
When $k$ is spherically complete (which will be our case), then points in $\P^{1,\an}_k$ are either of type I, II or III.
When the value group $|k^*|$ is $\R^*_+$, then there is no type III points (again this will be our case). 
Note that for a fixed $a\in k$, the map $r\mapsto \zeta(a,r)$ defines a geodesic in $\P^{1,\an}_k$ for its tree structure. 

\subsubsection{Action of $\PGL(2,k)$}
The group $\PGL(2,k)$ acts naturally  on $\P^{1,\an}_k$. The action is continuous and extends the natural action on $\P^1(k)$. 
If  $M\in\PGL(2,k)$ and $x$ is a non-rigid point, then we have
$|R(M(x))|:= |(R\cdot M)(x)|$
for any rational function $R$. 

Assume $(k,|\cdot|)$ is non-Archimedean.
This action preserves the type of points. More precisely, it is $3$-transitive on type I points, and transitive on type II points. In other words, 
for any type II point, one can find 
$M\in\PGL(2,k)$ such that $M(x)=\xg$.

Write $\H_k= \P^{1,\an}_k\setminus \P^1(k)$. It is a connected subtree of $\P^{1,\an}_k$ endowed with a canonical distance $d_\H$
which is uniquely determined by the following conditions. 
We have $d_\H(\zeta(a,r),\zeta(a,r'))= \log|r/r'|$ for any $a\in k$, and any $0<r<r'$; and $\PGL(2,k)$ is acting by isometries on $\H_k$. 

Pick any type II point $x\in \H_k$. A direction $\vec{v}$ at $x$ is by definition a connected component of $\P^{1,\an}_k\setminus \{x\}$. 
When we want to emphasize that this connected component is open, then we write $U(\vec{v})$.
The set of directions at $x$ is denoted by $Tx$. A point $y \neq x$ determines a unique direction at $x$, and two
points $y,y'$ determine the same direction iff the two segments $[y,x)$ and $[y',x)$ intersect. Any direction is 
determined by at least one rigid point since the set of rigid points is dense. 

Choose any affine coordinate on the complement of a point in $\P^1(k)$. Two rigid points $y, y'\in k$ determines the same direction at $\xg$
iff either $|y-y'|<1$, or $\min\{|y|,|y'|\}>1$. It follows that there is a canonical bijection $T\xg \to \P^1(\tk)$. 
Changing coordinates is equivalent to post-compose by a Möbius transformation defined over $\tk$.
When $x$ is an arbitrary type II point, and $\xg=M(x)$ for some $M\in\PGL(2,k)$, then 
$M$ induces a bijection between $Tx$ and $T\xg$. We get a canonical 
bijection $Tx \to \P^1(\tk)$ defined up to post-composition by elements in $\PGL(2,\tk)$.

\subsubsection{Rational maps on the projective plane}
A rational map $f$ of degree $d\geq 1$ defined over $k$ can be defined in affine coordinates,
but it is more convenient to work in homogeneous coordinates $[z_0:z_1]$. 
Then $f$ is determined by two homogeneous polynomials $P, Q\in k[z_0,z_1]$ of degree $d$
without common factors.

The canonical map on rigid points $f\colon \P^1(k)\to \P^1(k)$ extends continuously to 
the Berkovich space $\P^{1,\an}_k$ so that
$|R(f(x))| = |(R\cdot f)(x)|$
for any rational function, and any $x\in\H_k$.
The map $f\colon \P_{k}^{1,\an}\rightarrow \P_{k}^{1,\an}$ is continuous, finite, open and surjective, see, e.g.,~\cite[Proposition 4.3]{Jonsson}.

One can define a local degree $\deg_{x}(f)\in\{1,\cdots,d\}$ at any point $x\in\P^{1,\an}_k$ so that 
\[\sum_{f(y)=x} \deg_y(f)=d\] for any $x\in\P^{1,\an}_k$.

When $x$ is rigid, and $f(z)=w=\sum_{k\ge 1} a_k z^k$ is expanded into power series in affine charts centered at $x$ and $f(x)$ respectively, then 
$\deg_x(f)$ is the least integer $k$ such that $a_k\neq 0$. A rigid point is critical iff $\deg_x(f)\ge 2$. The 
 set of all rigid critical points has cardinality at most $2d-2$.

Suppose that $(k,|\cdot|)$ is non-Archimedean. When $x$ is a type II point, then $f$ induces a canonical map $T_xf \colon Tx \to Tf(x)$.
Indeed if $\vec{v}$ is a direction such that $U(\vec{v})$ does not intersect $f^{-1}(f(x))$, then
$f(U(\vec{v}))=U(\vec{w})$ for a unique $w\in Tf(x)$, and we set $T_xf (\vec{v}):= \vec{w}$.
When $U(\vec{v})\cap f^{-1}(f(x))$ is non-empty, then we say that $\vec{v}$ is a bad direction
(implicitly for $f$ at $x$). In that case, the image of the connected component of 
$U(\vec{v})\setminus f^{-1}(f(x))$ having $x$ in its boundary is mapped onto some
$U(\vec{w})$ with $w\in Tf(x)$ and we set  $T_xf (\vec{v}):= \vec{w}$.
Note that the set of bad directions is always finite.

Choose canonical identifications as above $Tx \simeq \P^1(\tk)$ and $Tf(x) \simeq \P^1(\tk)$. Then there
exists a unique non-constant rational map $\tilde{f}_x \colon   \P^1(\tk)\to  \P^1(\tk)$
such that the following diagram commutes:
 \begin{center}
\begin{tikzcd}
T x \arrow[d,"T_{x}f"] \arrow[r,"\simeq"] 
  & \P^1(\tilde{k}) \arrow[d,"\tilde{f}_x"] \\
T f(x) \arrow[r,"\simeq"] 
  & \P^1(\tilde{k})
\end{tikzcd}
\end{center}
The rational map $\tilde{f}_x$ is defined up to post- and pre-composition by Möbius transformations in $\PGL(2,\tk)$, 
and its degree is equal to $\deg_x(f)$.

\subsubsection{Fatou-Julia theory and periodic cycles}
Let $f$ be any rational map of degree $d\ge 2$. The exceptional set $\cE$ of $f$ is the largest totally invariant finite subset 
of $\P^1(k)$. It has at most two points. The Julia set $J(f)$ consists of those points $x\in\P^{1,\an}_k$
such that $\cup_{n\ge 0} f^n(U)$ contains $\P^{1,\an}_k\setminus \cE$ for all neighborhoods $U$ of $x$.
It is a closed totally invariant subset, and its complement $F(f)$ is the Fatou set. 
The Fatou set is by definition an open set, and a connected component of the Fatou set is called a Fatou component.

Suppose $x \in \P^1(k)$ is a rigid point which is periodic of period $l$. Its multiplier $(f^l)'(x)\in k$ is then well-defined, and we say that 
$x$ respectively superattracting, attracting, indifferent or repelling when 
$|(f^l)'(x)|= 0$, respectively $|(f^l)'(x)|<1$, $|(f^l)'(x)|=1$, or $|(f^l)'(x)|>1$. 
The point $x$ lies in the Julia set iff it is repelling. 

Suppose that $(k,|\cdot|)$ is non-Archimedean. A  periodic point $x\in \H_k$ which lies in the Julia set is necessarily of type II. 
Conversely, suppose $x$ is a type II periodic point of period $l$. 
The degree of the cycle is the degree of $\deg_x(f^l)$. 
When $\deg_x(f^l)\ge2$, then $x$ lies in the Julia set, and we say it is repelling. When $\deg_x(f^l)=1$,
we say that the point is indifferent. 
In this case,
$x$ lies in the Julia set iff there exists at least one bad direction $\vec{v}\in Tx$
which has an infinite orbit under $T_xf$. The existence of an indifferent type II fixed point in the Julia set
implies the existence of a critical point having an infinite orbit by~\cite[Theorem~4.1]{zbMATH06455745KJ15}.

Finally, recall that $f$ is said to have good reduction when $J(f)=\{\xg\}$. 
Since $J(f)$ is totally invariant, this happens iff $\xg$ is fixed by $f$, and 
$\deg_{\xg}(f)$ is maximal equal to $d$. 

When $f$ is conjugated by a Möbius transformation to a map having good reduction, then 
we say that $f$ has potential good reduction. 
This is equivalent to say that $J(f)$ is reduced to a singleton, or that $f$ has zero entropy (see~\cite{FRL10}). 

\subsubsection{Fatou components}
Let $x$ be an attracting periodic point of 
period $l \ge 1$. Then there exists an open neighborhood $V$ of $x$ such that 
$f^{nl}(y)\to x$ for any $y\in V$. The attracting basin of the cycle
is then equal to $\bigcup_{n \ge 0} f^{-n}(V)$ and is included in the Fatou set.  
Any periodic connected component of the attracting basin
is called an attracting component. The immediate basin of the cycle is the union of all periodic components of the attracting basin.
Observe that an attracting component contains a unique
periodic point, and this point is attracting. 

Suppose $(k,|\cdot|)$ is Archimedean. A rotation domain $U$ is a periodic Fatou component of period $l$ which is conformally equivalent 
to either a disk or an annulus, and the dynamics of $f^l|_U$ is conjugated to an irrational rotation. 
A parabolic basin is a periodic Fatou component $U$ of period $l$ such that for any point $y\in U$,  $f^{nl}(y)$
converges to fixed point of $f^l$ whose multiplier is a root of unity. By 
the Fatou-Leau 
flower theorem, for any periodic point $x$ of period $l$ whose multiplier is a root of unity, then one can find
a parabolic basin $U$ containing $x$ in its boundary so that $f^{nl}|_U \to x$.

When $(k,|\cdot|)$ is non-Archimedean, we say a Fatou component $U$ for $f$  is called a Rivera domain if there exists an integer $l\ge 1$ such that $f^l\colon U\to U$ is injective. 
One can show in this case that $\partial U$ consists of a finite union of type II Julia periodic points. 
The following classification of periodic Fatou components is due to Fatou in the Archimedean case~\cite[\S4]{Be13}, and to Rivera-Letelier in the non-Archimedean case, 
see~\cite[Proposition 4.3, Theorem 9.14]{zbMATH07045713benedetto}.

\begin{thm}
\label{thm: classificationfatou}
When  $(k,|\cdot|)$ is Archimedean, then any periodic Fatou component is either an attracting component, or a parabolic basin, 
or a rotation domain.  

When $(k,|\cdot|)$ is non-Archimedean, then any periodic Fatou component is either an attracting component or a Rivera domain. 
\end{thm}

\subsection{Scales and Robinson fields}

\subsubsection{Non-principal ultrafilters}
Let $\om$ be a non-principal ultrafilter in $\N$. 
Recall that this is determined by a non-trivial finitely additive measure on $\N$
with values in $\{0,1\}$. 
More precisely, this is 
a function on the power set $m_\om\colon 2^\N \to \{0, 1\}$
such that $m_\om(E)\le m_\om(F)$ if $E\subset F$; $m_\om(E\cup F)= m_\om(E)+m_\om(F)$ if $E\cap F=\emptyset$;
$m_\om(\N)=1$, and $m_\om(\{n\})=0$ for all $n\in\N$.
The existence of a non-principal ultrafilter relies on the axiom of choice. 

It is customary to say that $E$ is $\om$-big when $m_\om(E)=1$.
The  ultrafilter technique allows one to extract a generalized subsequence in any compact Hausdorff space $K$, even when $K$ is not sequentially compact. More precisely, for any sequence $x=(x_n)\in K^{\N}$, there exists a (unique) point $x_\omega \in K$ such that for every neighborhood $V$ of $x_\omega$, the set
$\{n \in \N : x_n \in V\}$
is $\omega$-big. We call $x_\omega$ the limit of $x$ along $\omega$. Note that if $(x_n)$ converges in the usual sense, then $x_\omega$ coincides with the standard limit of a sequence.

\smallskip

We fix once and for all a non-principal ultrafilter $\om$ in $\N$. 
Since this ultrafilter is fixed, we shall drop the reference to $\om$ and simply write $\lim x = x_\om$.

\subsubsection{Scales}

We say that two bounded sequences of positive real numbers  
\((\eps_n)\) and \((\eta_n)\) are equivalent
iff  $\lim \frac{\eps_n}{\eta_n}\in(0,\infty)$. 
We write $\eps\asymp \eta$. 

A scale is an equivalence class of bounded sequence of positive reals. 

A scale is said to be trivial if it is represented by the constant sequence $(1)$. 
This is equivalent to say that \(\lim \eps_n > 0\).

We define a partial order on the set of scales by declaring \(\eps \le \eta\) if and only if
$\lim\frac{\eps_n}{\eta_n} < \infty .$
In particular, \(\eps < \eta\) holds precisely when
$\lim \frac{\eps_n}{\eta_n} = 0 .$
Since  $\lim \eps_n / \eta_n$ always exists in $[0,\infty]$,  the set of scales is a totally ordered set with a unique maximal element $(1)$.

It is convenient to add the constant sequence $(0)$ as an extended scale, and declare $(0) < \eps$ for any scale. 

\subsubsection{Complex Robinson fields} 
\label{sec:complexrobinsonfield}
References include~~\cite[\S6]{YLuo22} and~\cite{FG25,CG25}.
Let $\eps=(\eps_n)$ be any bounded sequence of positive real numbers. The complex Robinson field associated with $\eps$ is defined as follows: 
\[
\sH_{\eps} = \left\{(z_n) \in \mathbb{C}^{\mathbb{N}} \mid \lim |z_n|^{\eps_n} < \infty \right\} \Big/ \left\{(z_n) \mid \lim |z_n|^{\eps_n} = 0\right\}.
\]
Observe that the field only depends on the scale determined by $\eps$.
The map sending $c\in \C$ to the constant sequence is an embedding
$\C\subset \sH_\eps$ hence the latter field has characteristic $0$.
Endow $\sH_{\eps}$ with the norm \( |(z_n)| := \lim  |z_n|^{\min\{1,\eps_n\}} \).
Note that if $(\eps'_n)$ is equivalent to $(\eps_n)$, then the two norms are not equal, but there are related by 
the identity $|\cdot|'= |\cdot|^a$ with $a= \lim \min\{1,\eps'_n\}/\min\{1,\eps_n\}>0$.
In the sequel, we shall abuse notation and talk about the metrized field $\sH_\eps$ where $\eps$ is a scale, even though
the norm is defined up to exponentiation.

Observe that when $\eps = (1)$, then $\sH_\eps$ is isometric to the field of complex numbers. 
When $\eps < (1)$, then $\sH_{\eps}$ is a  non-Archimedean complete non-trivially valued field, of residual characteristic $0$ and spherically complete.

Recall that we introduced the extended scale $(0)$. The associated field $\sH_0$  is the set of all sequences of complex numbers 
modulo the ones that are $0$ on an $\om$-big set. It is trivially valued. 

\smallskip

For any scale $\eta$, we set  
\[
R^{+}_{\eta} = \left\{(z_n) \in \mathbb{C}^{\mathbb{N}},  \lim |z_n|^{\eta_n} <\infty\right\} \Big/ \left\{(z_n),  \{n\in\N, z_n= 0\} \text{ is } \om-\text{big}\right\}.
\]
and 
\[
R^{-}_{\eta} = \left\{(z_n) \in \mathbb{C}^{\mathbb{N}}, \lim |z_n|^{\eta_n} \leq 1 \right\} \Big/\left\{(z_n),  \{n\in\N, z_n= 0\} \text{ is } \om-\text{big}\right\}.
\]
When $\eta < (1)$, then $R^\pm_\eta$ are valuation rings such that 
$\C \subseteq R^-_\eta\subsetneq R^+_\eta\subseteq \sH_0$
since \( (\exp(1/\eta_n))\in R^{+}_{\eta} \setminus R^{-}_{\eta} \). 
Note moreover that there are canonical surjective
morphisms $R^+_{\eta}\to \sH_\eta$ and $R^-_{\eta}\to \sH^\circ_\eta$. 

When $\eta =(1)$, then  $R^-_{(1)}$ is not a ring, whereas $R^+_{(1)} =\sH_0$.

\smallskip

The next result follows from Abhyankar's inequality. 
A slightly different version was proved by the second author, see~\cite[Theorem 3.15]{CG25}.
 \begin{thm}
 \label{thm:numbervaluationring}
Let \( K \) be a field with \( \C \subseteq K \subseteq \sH_{0} \). Suppose that \( K \) has finite transcendence degree \( d \) over $\C$. Then the maximal number of distinct scales  
$\eta^1 < \eta^2 < \dots < \eta^l < (1)$ such that
\[
R^-_{\eta^i} \cap K \subsetneq R^+_{\eta^i} \cap K \quad \text{for each } i = 1, \dots, l,
\]
is bounded from above by $d$.
\end{thm}

\begin{proof}[Sketch of proof]
Pick any finite sequence of scales $\eta^1 < \eta^2 < \dots < \eta^l < (1)$ satisfying
$R^-_{\eta^i} \cap K \subsetneq R^+_{\eta^i} \cap K$ for all $i$. Note that 
$R^+_{\eta^{i+1}} \subset R^-_{\eta^i}$. It follows that 
$R^+_{\eta^l}\subsetneq R^+_{\eta^{l-1}}\subsetneq \cdots \subsetneq R^+_{\eta^1}$ forms a strictly increasing sequence of valuation rings
in $K$. They thus define a rank $l$ valuation on $K$ whose rank is bounded from above by $d$ by Abhyankar's inequality, see~\cite[\S5]{MV00}.
The result follows.
\end{proof}

Finally for any extended scale $\eps <(1)$ (including $\eps=(0)$), we also introduce the following ring
\[
B[\eps] :=
\left\{(z_n) \in \mathbb{C}^{\mathbb{N}}, \sup |z_n| <\infty\right\} \Big/\left\{(z_n), \lim |z_n|^{\eps_n}=0\right\}.
\]
This  is a valuation subring of $\sH_\eps$, and in particular, it is integrally closed by~\cite[Proposition~5.18]{zbMATH03279238}, 
and its fraction field is $\sH_\eps$. It is contained in $\sH^\circ_\eps$ when $\eps>(0)$ is a scale.

\begin{lemma}\label{lem:intclosed}
Let $K$ be an algebraically closed field of $\sH_0$, and consider the image 
$\tilde{B}_K[\eps]$ of $K\cap B[0]$ in $\widetilde{\sH}_\eps$ for a given scale $(0)<\eps<(1)$.

Any polynomial with coefficients in $\tilde{B}_K[\eps]$ splits in the fraction field of $\tilde{B}_K[\eps]$.
\end{lemma}
Note that 
$\tilde{B}_{\sH_0}[\eps] =
\left\{(z_n) \in \mathbb{C}^{\mathbb{N}}, \sup |z_n| <\infty\right\} \Big/\left\{(z_n), \lim |z_n|^{\eps_n}<1\right\}$.
\begin{proof}
We may suppose that 
the polynomial is the image in $\widetilde{\sH}_\eps$ of
 $P(T)= b T^d + a_{d-1} T^{d-1}\cdots + a_0$ with $b, a_i \in K\cap B[0]$, and $\tilde{b}_\eps\neq0$. 
Then $Q(T) = T^d + b a_{d-1} T^{d-1} + \cdots + b^{d-1} a_0$ splits over $K\cap B[0]$
since $K$ is algebraically closed and $B[0]$ is integrally closed. 
It follows that $\tilde{Q}_\eps(T)=\prod_i (T-\alpha_i)$ with $\alpha_i\in\tilde{B}_K[\eps]$, 
and 
$\tilde{P}_\eps(T)= \tilde{b}_\eps^{1-d} \prod_i (\tilde{b}_\eps T-\alpha_i)$.
\end{proof}

\subsubsection{Limits in projective varieties} \label{sec:limits-proj}

Suppose $X$ is a complex projective variety, and pick any field extension $k/\C$. 
Let $X_k$ be the projective variety over $k$ obtained by base change.
A (closed) point $x\in X(k)\subset X_k$ is determined by an 
affine chart $U\subset X$ and a ring morphism $\eva_{x,U,k}\colon k[U]=\C[U]\otimes_\C k \to k$ 
trivial on $k$ modulo the  equivalence $\eva_{x,U,k}|_{k[U\cap V]}= \eva_{x,V,k}|_{k[U\cap V]}$
for any other chart $V$. Observe that $\eva_{x,U,k}$ is determined by its restriction to $\C[U]$.

Pick any sequence $(x_n)\in X(\C)^\N$. Since the complex variety $X$ is compact, the limit 
$x_\C:= \lim x_n$ exists as a point in $X(\C)$. 
Pick an affine chart $U$ containing $x_\C$, and let $\eps<(1)$ be any scale. 
The map $\phi \mapsto (\phi(x_n))$ defines a morphism
$\eva_{x,\eps} \colon  \C[U] \to B[\eps]$, and this defines a point 
$x_\eps \in X(\sH_\eps)$. Composing the evaluation map with the canonical morphism 
$B[\eps]\to \sH^\circ_{\eps} \to \widetilde{\sH}_\eps$, we also get a point $\tilde{x}_\eps\in X(\widetilde{\sH}_\eps)$.

When $X = \P^d$,  this construction can be done as follows. Choose
homogeneous coordinates $[z_0: \cdots :z_d]$. 
Choose any normalized representative $x_n= [z_{0,n}: \cdots :z_{d,n}]$
in the sense that $\max |z_{i,n}| =1$ for all $n$. 
For any scale $\eps$, we let 
$z_{i,\eps}\in \sH_\eps$ be the point determined by the sequence $(z_{i,n})$. Then we have 
$x_\eps = [z_{0,\eps}: \cdots :z_{d,\eps}]$
(and similarly for $\tilde{x}_\eps$).

Observe that the constructed maps $X(\C)^\N \to X(\sH_\eps)$ (resp.  $X(\C)^\N \to X(\widetilde{\sH}_\eps)$) are natural in the sense
that for any regular map $f\colon X\to Y$ and any sequence $x\in X^\N$
then $f(x_\eps)=f(x)_\eps$ (and $f(\tilde{x}_\eps)=\widetilde{f(x)}_\eps$).

\smallskip

We now compare these maps when we vary the scales. Pick any pair of scales $\eps < \eta$. 
We first construct a natural map $X(\sH_{\eps})\to X(\sH_\eta)$ as follows. 
Pick any point $x\in X(\sH_\eps)$.  We can then find an affine chart $U$
such that $\eva_{x,U} \colon \sH_\eps[U]\to B[\eps]$ maps
$\C[U]$ to $B[\eps]$. Indeed, we can embed $X$ into a projective space $\P^d$, 
suppose that $x$ has homogeneous  coordinates $[1:x_1:\cdots :x_d]$
so that $x_i$ is represented by a sequence of complex numbers of norm $\le 1$, and take $U= X\cap \{z_0\neq0\}$.

Since $\lim |z_n|^{\eps_n} =0$ implies
$\lim |z_n|^{\eta_n} =0$,
we have a well defined morphism $B[\eps]\to \sH_\eta$, and
the composite morphism $\C[U]\to B[\eps]\to  \sH_\eta$ determines a point in $X(\sH_\eta)$.
This defines the required map $X(\sH_{\eps})\to X(\sH_\eta)$.

Observing that $\lim |z_n|^{\eps_n} <1$ implies
$\lim |z_n|^{\eta_n} =0$, we get a similar natural map $X(\widetilde{\sH}_{\eps})\to X(\sH_\eta)$.

To summarize the above discussion, the natural maps $X(\C)^\N\to X(\sH_\eps)$, and $X(\C)^\N\to X(\widetilde{\sH}_\eps)$
all fit into the following commutative diagram. 
\begin{figure}[h]
\centering
\begin{tikzcd}
X(\C)^{\N} \arrow[r, "x\mapsto x_\eps"] \arrow[rrd, swap,"x\mapsto x_\eta"] \arrow[rr, bend left, "x\mapsto \tilde{x}_\eps"]
 \arrow[rrrd, bend right, swap, "x\mapsto x_\C"]
& X(\sH_{\eps}) \arrow[r] \arrow[rd]
& X(\widetilde{\sH_{\eps}}) 
\arrow[d] & \\
&    & X(\sH_{\eta})   \arrow[r]  &  X(\C)                                                
\end{tikzcd}                    

\caption{Relating limits in projective spaces (with $\eps < \eta$)}
\label{diagram:residue}
\end{figure}

\begin{rmk}
We may  replace $X(\C)^\N$ by $X(\sH_{0})$ in the diagram above.
\end{rmk}

\begin{rmk}
If $(k,|\cdot|)$ is a non-Archimedean field, the projective distance on $\P^1(k)$ is defined
by setting \[d([z:1],[z':1])=|z-z'|/(\max(1,|z|)\max(1,|z'|).\]
The map $\P^1(\sH_\eps) \to \P^1(\sH_\eta)$ are $1$-Lipshitz for the projective distance if $\eps<\eta<(1)$ .
Observe that points at projective distance $<1$ from $[0:1]$ (i.e.,  the open unit ball in $\sH_\eps$) are all
mapped to the point $[0:1]$ in $\P^1(\sH_\eta)$ so that 
$\P^1(\sH_\eps) \to \P^1(\sH_\eta)$ does not extend continuously to the Berkovich analytifications. 
\end{rmk}

%%%%%%%%%%%%
\section{Sequences of rational maps}\label{sec:seq}
As in the previous section, we let $k$ be an algebraically closed field of characteristic $0$,
endowed with a non-trivial complete norm. 
In most cases, we have $k= \sH_\eps$ for some scale $\eps$, so that $k$ is of residual characteristic $0$
when non-Archimedean. 

We shall also fix homogeneous coordinates $[z_0:z_1]$ on $\P^1$, and affine coordinates $[z:1]$.

\subsection{The space of degree $d$ rational maps} 
 Let $\overline{\Rat}_{d}(k)$ be the space of pairs of homogeneous polynomials of degree $d\ge 1$ and coefficients in $k$
 modulo multiplication by scalars $k^*$. It is naturally identified with the space of $k$-points of the projective space of dimension $2d+1$. 
 
 A point in $\ovRat_d(k)$ can be naturally interpreted as a pair of a rational map $f$ of degree $\delta \le d$
 and an effective divisor $D$ on $\P^1_k$ of degree $d-\delta$. 
 The divisor $D$ is the divisor of zeroes of the common factor of the two polynomials. 
 The set of holes of an element $(f,D) \in \ovRat_d(k)$
 is by definition the support $|D|$ of $D$. 
 
\medskip
 
 The space $\ovRat_d(k)$ contains the set $\Rat_{d}(k)$ of pairs of polynomials $[P:Q]$ whose homogeneous resultant $\Res(P,Q)$ is non-zero. 
 It is the space of $k$-points of a Zariski open subset $\Rat_d$ inside $\ovRat_d$.
 The space $\Rat_d(k)$ can be identifed with the space of all rational maps $f\colon \P^1_k\to \P^1_k$ of degree $d$.

Since $k$ is a metrized field, and the homogeneous resultant is homogeneous of degree $2d$ in each variable, we may define
\[|\Res(f)|=\frac{|\Res(P,Q)|}{\max_{0\leq i, j\leq d}\{|a_i|,|b_j|\}^{2d}}\]
when $f=[P:Q]$, $P(z_0,z_1)=\sum a_i z_0^iz_1^{d-i}$, $Q(z_0,z_1)=\sum b_i z_0^iz_1^{d-i}$. 
This defines a continuous function on $\ovRat_d(k)$ which is positive exactly on $\Rat_d(k)$, i.e., 
$\{f\in\ovRat_d(k), |\Res(f)|>0\}= \Rat_d(k)$.

Observe that this function extends continuously to the Berkovich analytification of $\ovRat_d$ which is compact, 
hence it attains its maximum.  We denote by $C_d>0$ its supremum over $\Rat_d(k)$. 

 When $k$ is non-Archimedean, then we have $|\Res(f)|\le 1$ by the strong triangle inequality. Also we have $|\Res(f)|= 1$
 iff $f$ has good reduction. To see this, normalize $P$ and $Q$ so that $\max_{0\leq i, j\leq d}\{|a_i|,|b_j|\}=1$ and observe 
 that the reduction has non zero resultant in the residue field $\tk$.

 \medskip

 The group $\PGL(2)$ acts by conjugacy on $\Rat_d(k)$ by $(M,f) \mapsto M\cdot f \cdot M^{-1}$. 
 Assume from now on that $d\ge2$. From the existence of the moduli space $\rat_d:= \Rat_d/\PGL(2)$~\cite{SJ98} and Kempf-Ness theory, 
 one deduces the following lemma, see~\cite[Lemma 3.4]{FG25} for a proof.
\begin{lemma}
\label{lem:compareresultant}
 When $d\ge2$, there exist some positive real numbers $C,\alpha>1$ such that for all $M\in\PGL_{2}(k)$  and  $f\in\Rat_{d}(k)$, we have
\begin{multline*}
\label{eq:loj-GIT}
C \min\{|\Res(M)|, |\Res(f)|\}^{1/\alpha}
\ge 
\min\{|\Res(M\cdot f\cdot M^{-1})|,|\Res(f)|\}
\\
\ge 
C^{-1} \min\{|\Res(M)|, |\Res(f)|\}^\alpha.
\end{multline*} 
\end{lemma}

\begin{rmk}
Lemma~\ref{lem:compareresultant} is not valid when $\deg(f)=1$ as the centralizer of $f$ is non-compact. 
\end{rmk}

It is natural to introduce the minimal\footnote{the terminology is due to Rumely who considered the function $-\log|\res|$ which is natural to minimize along an orbit.} resultant function on $\Rat_d(k)$ by setting
\[
|\res(f)| := \sup\{|\Res(M\cdot f \cdot M^{-1})|, M\in \PGL(2,k)\}
\]
When $d\ge 2$, since each $\PGL(2)$-orbit is closed,
this quantity is attained for some $M\in\PGL(2,k)$. 
In particular, when $k$ is non-Archimedean, then $|\res(f)|\le1$ with equality iff $f$ has potential good reduction; and in the Archimedean case
$|\res|$ is bounded from above by $C_d$.

\begin{rmk}\label{rem:good1}
Assume $k$ to be non-Archimedean, and suppose $M\in \PGL(2,k)$. We shall say that 
$M$ has good reduction (resp. potential good reduction) if $|\Res(M)|=1$ (resp. $|\res(M)|=1$). 
Note that $M$ has good reduction iff it belongs to $\PGL(2,k^\circ)$; and it
does not have potential good reduction iff it is conjugate to a homothety $z\mapsto \la z$
with $|\la|<1$.
\end{rmk}

\subsection{Limits of a sequence of rational maps} 
Suppose that $f=(f_n)$ is a sequence of complex rational maps of degree $d\ge 1$. 
Since $\ovRat_d$ is a complex projective variety, it follows from the discussion in \S\ref{sec:limits-proj}, that 
for any scale $\eta$, we may consider the limit point of the sequence 
$(f_\eta,D_\eta(f)) \in \ovRat_d(\sH_\eta)$ (resp. $(\tilde{f}_\eta,\tilde{D}_\eta(f)) \in \ovRat_d(\widetilde{\sH}_\eta)$). 
Observe that $f_\eta$ is a rational map of degree $\delta \le d$, and $D_\eta(f)$ is a divisor
in $\P^1(\sH_\eta)$ of degree $d-\delta$. We call it the 
\emph{hole divisor}
associated with $(f_n)$. 
The same is valid over $\widetilde{\sH}_\eta$. 
\begin{prop}
\label{prop: compo}
   Let $(f_n)$ and $(g_n)$ be two sequences of rational maps of degrees $d$ and $d'$, respectively. Assume that $f_{\varepsilon}$ (resp. $\tilde{f}_\eps$) is not a constant rational map whose image is contained in the support of $D_{\varepsilon}(g)$ (resp. $\tilde{D}_{\varepsilon}(g)$). 
Then
$(g \cdot f)_{\varepsilon} = g_{\varepsilon} \cdot f_{\varepsilon}$, (resp. $\widetilde{(g \cdot f)}_{\varepsilon} = \tilde{g}_{\varepsilon} \cdot \tilde{f}_{\varepsilon}$)
and $
D_{\varepsilon}(g \cdot f) = d' \, D_{\varepsilon}(f) + f_{\varepsilon}^{*}\bigl(D_{\varepsilon}(g)\bigr)$, 
(resp. $\tilde{D}_{\varepsilon}(g \cdot f) = d' \, \tilde{D}_{\varepsilon}(f) + f_{\varepsilon}^{*}\bigl(\tilde{D}_{\varepsilon}(g)\bigr)$.
\end{prop}
In particular, when $\deg(f_\eta)\ge 1$, then for all integer $l\ge1$ we have
\begin{equation}\label{eq:hole-iterate}
d^{-l} D_\eta(f^l)=\sum_{j=0}^{l-1}d^{-j} f^{-j} D_\eta(f)~.
\end{equation}
Note that the same equation holds for $\tilde{D}_\eta(f)$.

\begin{proof}
We only treat the computation of 
$(g\cdot f)_{\varepsilon}$
since the proof for 
$\widetilde{g\cdot f}_\eps$
is completely analogous. 

Let \( f_n = [P_n : Q_n] \) (resp. \( g_n = [U_n : V_n] \)) be normalized representations of \( f_n \) (resp. \( g_n \)), so that the maximum of the absolute values of the coefficients of $P_n$ and $Q_n$ (resp. of $U_n$ and $V_n$) is equal to \(1\). 
Write $P_\eps$, etc. for the homogeneous polynomial induced by the sequence $P_n$ in $\sH_\eps$. 
Let \(H\) (resp. \(G\)) denote the greatest common divisor of \(P_\varepsilon\) and \(Q_\varepsilon\) (resp. of \(U_\varepsilon\) and \(V_\varepsilon\))
so that $P_\eps = H \hat{P}_\eps$, $Q_\eps = H \hat{Q}_\eps$, etc.
Then we have 
\begin{align*}
(U_n(P_n, Q_n))_\varepsilon 
&
= U_\eps (P_\eps, Q_\eps)
= H^{d'}U_\eps(\hat{P}_\eps, \hat{Q}_\eps)
\\
&
=
H^{d'}\, G(\hat{P}_\eps, \hat{Q}_\eps)\, \hat{U}_\eps(\hat{P}_\eps, \hat{Q}_\eps)
\end{align*}
and similarly 
\[
(V_n(P_n, Q_n))_\eps
= 
H^{d'}\, 
G(\hat{P}_\eps, \hat{Q}_\eps))
\, 
\hat{V}_\eps
(\hat{P}_\eps, \hat{Q}_\eps)~.
\]
If \( f_\varepsilon=[\hat{P}_\eps:\hat{Q}_\eps] \) is not a constant map with values in \( |D_{\varepsilon}(g)| =(G=0) \), then $G(\hat{P}_\eps, \hat{Q}_\eps)$ is not identically $0$. Since $\hat{U}_\eps(\hat{P}_\eps, \hat{Q}_\eps)$ and $\hat{V}_\eps(\hat{P}_\eps, \hat{Q}_\eps)$ are coprime, the result follows. 
\end{proof}

In the remaining of the paper, we consider the auxiliary function 
$\theta(s) = \frac{-1}{\log s}$ if $0\le s\le 1/e$, and 
$\theta(s)=1$ for all $s\ge 1/e$. 
Observe that $\theta$ is increasing, that
$s_n\to 0$ iff $\theta(s_n)\to 0$ and when it is the case, then 
$|s_n|^{\theta(s_n)}= 1/e<1$ for $n$ large enough.
Also for any $\alpha>0$, we have $\theta(s_n^\alpha) \asymp\theta(s_n)$.

\begin{thm}
\label{thm:trichotomy}
Pick any sequence of degree $d\ge 1$ rational maps  \( (f_n) \in \operatorname{Rat}_d(\mathbb{C})^{\N} \) and consider the scale 
$\eps(f) :=(\theta(|\Res(f_n)|))$.
Then the following holds:
\begin{enumerate}
    \item if $\eta<\eps(f)$, then $\deg(f_{\eta})=d$ and $f_{\eta}$ has good reduction;
    \item if $\eta=\eps(f)$, then $\deg(f_{\eta})=d$ and $f_{\eta}$ does not have good reduction if $\eps(f) <(1)$;
    \item if $\eta>\eps(f)$, then we have 
    \[\deg(f_{\eta})\leq \deg(\tilde{f}_{\eps(f)})\leq \deg_{x_g}(f_{\eps(f)})\leq d-1~.\]
\end{enumerate}
\end{thm}

\begin{rmk}
Observe that in the last case, $\deg(f_{\eta})\ge 1$ implies $\xg$ is fixed by $f_{\eps(f)}$.
\end{rmk}

\begin{proof}
Write $\eps_n= \theta(|\Res(f_n)|)$, and $\eps= (\eps_n)$.
To compute $f_\eta$, we represent each rational map $f_n$ by a pair of homogeneous 
polynomials $P_n, Q_n$ of degree $d$ normalized so that the maximum of the modulus
of all coefficients is $1$. Note that in that case, we have $|\Res(f_n)| = |\Res(P_n,Q_n)|$
so that  \[|\Res(f_\eta)|= \lim |\Res(f_n)|^{\eta_n} = \lim \exp(-\frac{\eta_n}{\eps_n})~.\] 
When $\eta< \eps$, we obtain  $|\Res(f_\eta)|=1$ hence $f_\eta$ has degree $d$ and good reduction. 
When $\eta= \eps$, we obtain  $0<|\Res(f_\eta)|<1$ hence $f_\eta$ still has degree $d$ but has not good reduction. 
Finally when  $\eta> \eps$, then $|\Res(f_\eta)|=0$ hence $\deg(f_\eta)$ has degree at most $d-1$. 

The upper bound $\deg(f_{\eta})\leq\deg_{x_g}(f_{\eps})$ follows from the naturality of the commutativity of the diagram
in Figure~\ref{diagram:residue}. Indeed, the natural map 
 $\P^1(\widetilde{\sH}_\eps)\to \P^1(\sH_\eta)$
semi-conjugates $\tilde{f}_\eps$ to $f_\eta$. Since $\deg(\tilde{f}_\eps)=\deg_{\xg}(f_\eps)$, the result follows.
\end{proof}

Later, we shall use the following observation. 
\begin{lemma}\label{lem:upto}
Suppose $(g_n)\in \Rat_d(\C)^\N$, and $(h_n)\in \Rat_{\delta}(\C)^\N$ are sequences of complex rational maps of degree $d$ and $\delta$
respectively such that 
$\tilde{g}_\eps=\tilde{h}_\eps$ for some scale $\eps<(1)$. Then we have $g_\eta=h_\eta$ for any scale $\eta>\eps$.
\end{lemma}
\begin{proof}
Let $a_n$ and $b_n$ be two sequences complex numbers of bounded norm $\le R$. 
Suppose that $\tilde{a}_\eps = \tilde{b}_\eps \in \widetilde{\sH}_\eps$. Observe that
$|a_\eps|=|b_\eps|=1$ since $\eps<(1)$.  Then we have $|a_\eps - b_\eps|<1$, i.e.
$\lim |a_n-b_n|^{\eps_n} <1$. This implies $\lim |a_n-b_n|^{\eta_n}=0$ for any scale
$\eps >\eta$. 

To conclude the proof we represent $g_n$ and $h_n$ by sequences of  polynomials
$A_n, B_n$, and $P_n, Q_n$  of degree $d$ and $\delta$ respectively, so that the maximum of the norm of their coefficients
is equal to $1$. Then $\tilde{g}_\eps=\tilde{h}_\eps$  iff 
$\tilde{A}_\eps \tilde{Q}_\eps = \tilde{P}_\eps \tilde{B}_\eps$ and the result follows from the previous remark. 
\end{proof}

\subsection{Critical and periodic points of the non-Archimedean limits} 
When $\eta = (1)$, then $f_\eta$ is a complex rational map, and it is well-known that $(f_n)$ converges uniformly on compact
subsets to $f_\eta$ outside the support of the hole divisor, see, e.g.,~\cite[Lemma 4.2]{zbMATH05004325demarcoiteration}, or~\cite[\S{5.1}]{zbMATH06455745KJ15}.
In general, we can still relate dynamical quantities of the sequence $(f_n)$ to those of $f_\eta$.

For any rational map $f\in \Rat_d(k)$, denote by 
\[\cC(f):=\sum_{p\in\P^1(k)}(\deg_f(p)-1) [p]\] its critical divisor,
whose degree is equal to $2d-2$. If $(f_n)$ is a sequence of complex rational maps, and $\eta$ is any scale, then
$\cC(f_n)$ admits a limit as a divisor in $\P^1(\sH_\eta)$ (resp. in $\P^1(\widetilde{\sH}_\eta)$) which we denote by $\cC(f)_\eta$ (resp. $\widetilde{\cC}(f)_\eta$). 
\begin{lemma}\label{lem:crit}
For any sequence of complex rational maps $(f_n)$ and for any scale $\eta$,
we have 
\[
\cC(f)_\eta = 
\cC(f_\eta) + 2 D_\eta(f),
\text{ and } \widetilde{\cC}(f)_\eta = 
\cC(\tilde{f}_\eta) + 2 \tilde{D}_\eta(f)
\]
as divisors in $\P^1(\sH_\eta)$ (resp. in $\P^1(\widetilde{\sH}_\eta)$).
\end{lemma} 

\begin{proof}
To see this, we work directly in homogeneous coordinates, and 
write $f_n= [P_n:Q_n]$ where the maximum of the modulus of
the coefficients of $P_n$ and $Q_n$ is equal to $1$. 
Then $P_\eta = H \hat{P}_\eta$ and $Q_\eta = H \hat{Q}_\eta$
where $D_\eta(f)$ is the divisor of zeros of $H$, and 
$f_\eta = [\hat{P}_\eta:\hat{Q}_\eta]$. 

For any homogeneous polynomial $R$, let 
$[R]$ be its divisor or poles, and 
$R_0$, respectively $R_1$
be the partial derivative with respect to $z_0$ and $z_1$ respectively. 
Since $\cC(f_n)$ is the divisor of zeroes of 
$P_{n,0}Q_{n,1} - P_{n,1}Q_{n,0}$, we get
\[
\cC(f)_\eta = 
[P_{\eta,0}Q_{\eta,1} - P_{\eta,1}Q_{\eta,0}]
= [H^2(\hat{P}_{\eta,0}\hat{Q}_{\eta,1} -\hat{P}_{\eta,1}\hat{Q}_{\eta,0}]
\]
using the identities 
$P_\eta = H \hat{P}_\eta$ and $Q_\eta = H \hat{Q}_\eta$.
The same argument applies over $\widetilde{\sH}_\eta$, and implies the result. 
\end{proof}

Denote by $\Per_l(f)$ the divisor of periodic points of period dividing $l$ of a rational map $f$, 
where each point is counted with its multiplicity as a fixed point of $f^l$. Note that the degree of  $\Per_l(f)$  is equal to $1+d^l$. 
If $(f_n)$ is a sequence of rational maps, let  $\Per_l(f)_\eta$ (resp. $\widetilde{\Per}_l(f)_\eta$)
be the limit of $\Per_l(f_n)$ in $\P^1(\sH_\eta)$ (resp. in $\P^1(\widetilde{\sH}_\eta)$).

\begin{lemma}\label{lem:period}
For any sequence of complex rational maps $(f_n)$, for any integer $l\ge1$, and for any scale $\eta$,
we have 
\[
\Per_l(f)_\eta = 
\Per_l(f_\eta) + D_\eta(f^l),
\text{ and }
\Per_l(f)_\eta = 
\Per_l(\tilde{f}_\eta) + \tilde{D}_\eta(f^l)\]
as divisors in $\P^1(\sH_\eta)$ (resp. in $\P^1(\widetilde{\sH}_\eta)$).

Moreover, let $p$ be a periodic point of period $l$ of $f_\eta$ (resp. of $\tilde{f}_\eta$) whose orbit does not intersect $|D_\eta(f)|$ (resp. $|\widetilde{D}_\eta(f)|$), and
$p_n\in |\Per_l(f_n)|$ be a sequence satisfying $p_\eta = p$. Then the multipliers are related by the identity
\[(f^l_\eta)'(p) = \left( (f^l_n)'(p_n)\right)_\eta \in \sH_\eta~.\]
\end{lemma} 
The next example shows that the assumption on the periodic cycle to avoid the support of the hole divisor is necessary.
\begin{example}
Pick any $a\in \C^*\setminus\{1\}$, and consider the sequence of quadratic maps
$f_n(z)= z\frac{z+an^{-1}}{z+n^{-1}}$.
Then $0$ is a fixed point having multiplier $a$, but the limit at the trivial scale $(1)$ is equal to 
$f_{(1)}(z)=z$.
\end{example}

\begin{proof}
We only argue over $\sH_\eta$, the same proof works over $\widetilde{\sH}_\eta$. 
Write $f_n[z_0:z_1]=[P_n:Q_n]$. We may suppose $l=1$ so that the divisor $\Per_1(f)$ is determined by the divisor
$[z_1P_{n}-z_0Q_{n}]$. Taking the limit at the scale $\eta$ and writing as in the previous proof 
$P_\eta = H \hat{P}_\eta$ and $Q_\eta = H \hat{Q}_\eta$ with $[H]=D_\eta(f)$ and $f_\eta=[\hat{P}_\eta:\hat{Q}_\eta]$, the result follows.

For the next statement, observe that since the orbit of $p$ avoids $|D_\eta(f)|$, then 
$p$ does not lie in the support of $D_\eta(f^l)$ by~\eqref{eq:hole-iterate}.
We may thus suppose that $l=1$, and further conjugate $f_n$ by a sequence $M_n\in\PGL(2,\C)$
sending $p_n$ to $0$ with $M_n$ converging to some $M\in\PGL(2,\C)$ (pick for instance $M_n$ in
a fixed maximal compact subgroup of $\PGL(2,\C)$).
Replacing $f_n$ by $M_n\cdot f_n \cdot M_n^{-1}$, we have
$p_n=0$ for all $n$, so that in the coordinates $[z:1]$, we can write
\[
P_n(z)= a_{1,n} z+ \cdots + a_{d,n}z^d
\text{ and } 
Q_n(z)= b_{0,n}+ b_{1,n} z+ \cdots + b_{d,n}z^d
\]
with 
$f'_n(p_n)=a_{1,n}/b_{0,n}$. 
Taking the limit at the scale $\eta$ yields
$P_\eta(z)= a_{1,\eta} z+ \cdots + a_{d,\eta}z^d$, and
$Q_\eta(z)= b_{0,\eta}+ b_{1,\eta} z+ \cdots + b_{d,\eta}z^d$. 
And since $H(0,1)\neq0$, we have 
$\hat{P}_\eta(z)= H(0,1)^{-1}  a_{1,\eta} z+ O(2)$, 
$\hat{Q}_\eta(z)= H(0,1)^{-1}  b_{0,\eta} z+ O(1)$,
so that 
\[(f'_n(p_n))_\eta=a_{1,\eta}/b_{0,\eta}=f'_\eta(p_\eta)\]
as required.
\end{proof}

%%%%%%%%%%%%
\section{The fundamental scale of a sequence of rational maps}
\label{sec:funda-scale}
Let $(f_n)$ be a sequence of complex rational maps of degree $d\ge 1$. 
We define its fundamental scale by setting $\eps_\star(f) := (\theta(|\res(f_n)|))$. 
In this section, we explore the limits of a sequence of rational maps $(f_n)$ of degree $d\ge2$ 
at the scale $\eps_\star(f)$.

Observe that for any sequence of complex Möbius transformations $(M_n)$
then $\eps_\star(f) = \eps_\star(M\cdot f\cdot M^{-1})$, and there always exist
some sequence $(M_n)$ such that 
$|\Res(M_n\cdot f_n\cdot M_n^{-1})|=|\res(f_n)|$ for all $n$.
This condition implies the equality of scales $\eps(M\cdot f \cdot M^{-1})=\eps_\star(f)$, 
but the converse is not true in general. 

We have $\eps_\star(f)=(1)$ iff $f_n$ is bounded in moduli, i.e., there exists a sequence of Möbius transformations
such that $M_n^{-1} \cdot f_n \cdot M_n$ converges to a degree $d$ rational map (or equivalently
the set of conjugacy classes $[f_n]$ is relatively compact in $\rat_d$).

\subsection{Changes of coordinates realizing the minimal resultant}
In this section, we give a characterization of those projective change of coordinates which minimizes the resultant. 

\begin{thm}\label{thm:unique-rescaling}
Let $(f_n)$ be any sequence of degree $d\ge 2$ complex rational maps, and 
let $(M_n)$ be a sequence of Möbius transformations.
The following statements are equivalent:
\begin{enumerate}
\item
$\eps(M \cdot f \cdot M^{-1}) = \eps_\star(f)$;
\item
$(M \cdot f \cdot M^{-1})_{\eps(M \cdot f \cdot M^{-1})} $ does not have potential good reduction;
\item
for all sequence of Möbius transformations $(L_n)$,
$(LM^{-1})_{\eps(L\cdot f \cdot L^{-1})}$ has degree $1$.
\end{enumerate}
\end{thm}

The last statement is equivalent to say that either
$LM^{-1}$ or $ML^{-1}$ belong to $\PGL(2,\sH_{\eps(L\cdot f \cdot L^{-1})})$.
Here is a way to use this result in a simple concrete example. 
\begin{example}
\label{eg: compute-fundascale}
Let \( f_n(z) = z^2 + n z \), and set 
$\eps = \left(1/\log n \right).$
Define $a =(n)\in\sH_0$, and observe that 
$|a|_\eps = e>1$.
Since $\deg(f_\eta)=0$ if $\eta>\eps$, and $\deg(f_\eps)=2$, it follows that $\eps(f)=\eps$ by Theorem~\ref{thm:trichotomy}.

Note that \( 0 \) is a repelling fixed point of $f_{\eps}$
hence this map does not have potential good reduction. Theorem~\ref{thm:unique-rescaling}
implies
$\varepsilon_\star(f) = \varepsilon(f) = ( 1/\log n)$.
\end{example}
\begin{proof}
The implication (1) $\Rightarrow$ (2) was claimed in~\cite[Theorem~4.17]{FG25}, but the argument was inacurate.
We thus provide  a new proof for the convenience of the reader.
To simplify notation set $\eps_\star= \eps_\star(f)$, and $g=M \cdot f \cdot M^{-1}$.

By Theorem~\ref{thm:trichotomy} (2), $g_{\eps_\star} $ is a rational map of degree $d$.
Since no complex rational map has potential good reduction by convention, 
we may suppose that $\eps_\star<(1)$ so that $\sH_{\eps_\star}$ is non-Archimedean.
 We need to show that $|\res(g_{\eps_\star})|\in(0,1)$.

Pick any sequence $(L_n)\in\PGL(2,\C)^\N$ such that 
$L_{\eps_\star} \in \PGL(2,\sH_{\eps_\star})$ and  $|\Res(L_{\eps_\star} \cdot g_{\eps_\star} \cdot L_{\eps_\star}^{-1})|=|\res(g_{\eps_\star})|$.
Observe that ${\eps_\star} \le \theta(|\Res(L)|)$ by Theorem~\ref{thm:trichotomy} (1). 
It follows from Lemma~\ref{lem:compareresultant} applied to $g_n$ and $L_n$
that 
\begin{align*}
\min\left\{
\theta(|\Res(L\cdot g\cdot L^{-1})|)
,
\theta(|\Res(g)|)
\right\}
\asymp 
 \min 
 \left\{
\theta(|\Res(L)|)
,
\theta(|\Res(g)|)
\right\}
\end{align*}
which implies 
\[
\theta(|\Res(L\cdot g\cdot L^{-1})|)
\asymp \eps_\star\]
since ${\eps_\star} \asymp \theta(|\Res(g)|)$.
We conclude that 
\[
|\res(g_{\eps_\star})|
=
|\Res(L_{\eps_\star} \cdot g_{\eps_\star} \cdot L_{\eps_\star}^{-1})|
=
\lim 
|\Res(L_n \cdot g_n \cdot L_n^{-1})|^{\eps_{\star,n}}
\in (0,1)\]
which implies the result.

\smallskip

We prove (2) $\Rightarrow$ (1) by contraposition. We may assume $M=\id$ and $\eps(f)< \eps_\star(f)$.
We need to show that $f_{\eps(f)}$ has potential good reduction. 
Pick any sequence $(L_n)$ of Möbius transformations such that $|\Res(L_n\cdot f_n \cdot L_n^{-1})|=|\res(f_n)|$ for all $n$.
Observe that $|\Res(L\cdot f\cdot L^{-1})|_{\eps(f)}=1$ hence $(L\cdot f\cdot L^{-1})_{\eps(f)}$ has good reduction.
By Lemma~\ref{lem:compareresultant} applied to 
$L_n\cdot f_n\cdot L_n^{-1}$ and $L_n^{-1}$, we get $\eps(L)\ge \eps(f)$
hence $\deg(L_{\eps(f)})=1$  by Theorem~\ref{thm:trichotomy}.
It follows from Proposition~\ref{prop: compo} that
$(L\cdot f\cdot L^{-1})_{\eps(f)}= L_{\eps(f)} \cdot f_{\eps(f)} \cdot L^{-1}_{\eps(f)}$
hence $f_{\eps(f)}$ has potential good reduction. 

\smallskip

Finally, the equivalence (1) $\Leftrightarrow$ (3) is a consequence of Lemma~\ref{lem:compareresultant} applied to 
$M\cdot f \cdot M^{-1}$ and $LM^{-1}$. 
Indeed we get $\eps(M\cdot f\cdot M^{-1}) \ge  \eps(L\cdot f\cdot L^{-1})$  for all $L$ iff
$\eps (LM^{-1}) \ge \eps(L\cdot f\cdot L^{-1})$
for all $L$.
\end{proof}

\subsection{Limits at the fundamental scale in degree $\ge2$}
We can now state the main result of this section.
\begin{thm}
\label{thm:trichotomy2}
Let $(f_n)$ be any sequence of degree $d\ge 2$ complex rational maps. 
Pick any sequence of Möbius transformations $(M_n)$ such that
$\eps_\star(f)=\eps(M \cdot f \cdot M^{-1})$. Set $g_n = M_n \cdot f_n \cdot M_n^{-1}$.
Then  the following holds. 
\begin{enumerate}
    \item If $\eta<\eps_\star(f)$, then $\deg(g_{\eta})=d$ and $g_{\eta}$ has good reduction;
    \item if $\eta=\eps_\star(f)$, then $\deg(g_{\eta})=d$ and $g_{\eta}$ does not have potential good reduction;
    \item if $\eta>\eps_\star(f)$, then we have 
    \[\deg(g_{\eta})\leq\deg(\tilde{g}_{\eps_\star(f)})\leq
    \deg_{x_g}(g_{\eps_\star(f)})\leq d-1~.\]
\end{enumerate}
\end{thm}
\begin{proof}
This result is a direct consequence of Theorems~\ref{thm:trichotomy}, and~\ref{thm:unique-rescaling}. 
\end{proof}
Note that having (potential) good reduction is invariant under iteration. 
The next result thus follows directly from Theorems~\ref{thm:trichotomy} and~\ref{thm:trichotomy2}.
\begin{cor}\label{cor:iterate}
For any sequence of rational maps $(f_n)$ of degree $d\ge1$, and for any integer $l\ge 1$, we have 
$\eps(f^l) =  \eps(f)$,
and 
$\eps_\star(f^l) = \eps_\star(f)$.
\end{cor}

\begin{rmk}
The former result implies the properness of the iteration map on the moduli space, a result due to DeMarco~\cite{LD07}, see also~\cite[Corollary~1.2]{FG25}.
\end{rmk}

\subsection{Fundamental scale and semi-conjugacy}

\begin{thm}\label{thm:semi-funda}
Let $(h_{1,n})$ and $(h_{2,n})$ be two sequences of rational maps of degree $d_1$ and $d_2$ respectively such that 
$d_1d_2\ge2$.
Let 
$f = h_1\cdot h_2$ and $g=h_2\cdot h_1$.

Then we have $\eps_\star(f)=\eps_\star(g)$. Moreover, if $\eps(f)=\eps_\star(f)$ and $\eps(g)=\eps_\star(g)$, then we have
$\deg(h_{1,\eps_\star})\ge 1$, and $\deg(h_{2,\eps_\star})\ge 1$ with $\eps_\star=\eps_\star(f)=\eps_\star(g)$. 
\end{thm}

\begin{rmk}
Combined with~\cite[Theorem~1.1]{zbMATH07673790}, we obtain that the fundamental scales of two sequences of semi-conjugated rational maps are identical. 
\end{rmk}

\begin{proof}
We may replace $h_1$ and $h_2$ by $M\cdot h_1 \cdot L $ and $L^{-1} \cdot h_2\cdot M^{-1}$ for suitable sequences of Möbius transformations $M, L$,
and assume that $\eps(f)=\eps_\star(f)$ and $\eps(g)=\eps_\star(g)$.
Observe first that for any integer $l$, the map 
$h_2\colon |\Per_l(f)| \to |\Per_l(g)|$ is bijective with inverse 
$(h_1\cdot h_2)^{l-1} \circ h_1$.

Pick any scale $\eps \le \min\{\eps(f),\eps(g)\}$. 
Note that $\deg(f_\eps) = \deg(g_\eps)=d_1d_2$, hence $D_\eps(f)=D_\eps(g)=0$. 
It follows from Lemma~\ref{lem:period} that for each $l$,
the canonical map $\varphi\colon \P^1(\sH_0) \to \P^1(\sH_\eps)$
sends the divisors $\Per_l(f)$, $\Per_l(g)$ to 
$\Per_l(f_\eps)$ and $\Per_l(g_\eps)$ respectively.

Observe that if a point in $\Per_l(f_\eps)$ has several preimages under $\varphi$, then 
its multiplicity as a periodic point is at least $2$, hence it belongs to a parabolic cycle. 
The set of parabolic cycles of $f_\eps$  is finite (as the characteristic of $\sH_\eps$ is zero), so that 
for all $l$ large enough, the map
\[\varphi \colon |\Per_{2l}(f)|\setminus |\Per_l(f)|\to |\Per_{2l}(f_\eps)|\setminus |\Per_l(f_\eps)|\] is bijective, 
and similarly for $g$.

Pick $l$ large enough so that $|\Per_{2l}(f_\eps)|\setminus |\Per_l(f_\eps)|$ contains at least 
two points $p_\eps\neq q_\eps$ outside $D_\eps(h_2)$. 
Denote by $p$ and $q$ their (unique) lift in $ |\Per_{2l}(f)|\setminus |\Per_l(f)|$. 
Since both maps 
$\varphi \colon |\Per_{2l}(g)|\setminus |\Per_l(g)|\to |\Per_{2l}(g_\eps)|\setminus |\Per_l(g_\eps)|$ and 
$\Per_{2l}(f)\to\Per_{2l}(g)$ are bijective, 
then $h_2(p)_\eps \neq h_2(q)_\eps$ belong to  $|\Per_{2l}(g_\eps)|\setminus |\Per_l(g_\eps)|$, 
hence $h_{2,\eps}$ has degree at least $1$. We get in particular
$h_{2,\eps} \cdot f_\eps = g_\eps \cdot h_{2,\eps}$.

Having potential good reduction is invariant under semi-conjugacy, so that
$\eps_\star(f)=\eps_\star(g)$, and the proof is complete.
\end{proof}

\section{$g$-rescalings adapted to  a sequence of rational maps}
\label{sec:g-rescaling}

\subsection{Definition of $g$-rescalings}
A generalized rescaling, or simply a $g$-rescaling, is a pair $(M,\eta)$
where $M\in\PGL(2,\C)^\N$ is a sequence of complex Möbius transformations and $\eta$ is a scale.

Two $g$-rescalings $(M,\eta)$ and $(L,\eps)$ are 
\emph{equivalent}
if $\eta=\eps$ and 
$(M\cdot L^{-1})_\eta\in\PGL(2,\sH_\eta)$.
This is equivalent to say that
$\theta(|\Res(M\cdot L^{-1})|) \ge \eta$.

Observe that this relation is an equivalence relation. If
$M$ is equivalent to $L$ and $L$ is equivalent to $T$, since the limit 
$(M\cdot L^{-1})_\eta$
is not a constant map, 
it then follows from Proposition~\ref{prop: compo} that
$(M\cdot T^{-1})_\eta = (M\cdot L^{-1})_\eta\circ (L\cdot T^{-1})_\eta$,
and 
$M$ is equivalent to $T$. 
Note also that by Theorem~\ref{thm:trichotomy}, $\eta \le\theta(|\Res(M)|)$ implies $(M,\eta)$ to be equivalent to  $(\id,\eta)$. 

The space of sequences of complex Möbius transformations 
is a group in which sequences $(M_n)$ such that $\deg(M_\eta)=1$ form a subgroup $K_\eta$.
It is the set of sequences such that $\eta \le\theta(|\Res(M)|)$.

It follows that the space of $g$-rescalings at a given scale $\eta$ modulo equivalence can be identified
with the space of left classes $K_\eta\backslash\PGL(2,\sH_0)$. 

\subsection{Adapted $g$-rescalings}\label{sec:adapted-gr}
Pick any sequence of rational maps $f=(f_n)$ of degree $d\ge1$. 
A $g$-rescaling $(M,\eta)$ is said to be 
\emph{adapted}
to $f$ if the following conditions are satisfied:
\begin{itemize}
\item[(A1)]
$(M\cdot f \cdot M^{-1})_\eta$ is a rational map of degree $\delta$ at least $1$; 
\item[(A2)] 
$(M\cdot f \cdot M^{-1})_\eta$ does not have potential good reduction if $\eta <(1)$;
\item[(A3)]
if $\delta=1$, then there exists a scale $\eps<\eta$ and a direction $v$
at $\xg$ such that :
\begin{itemize}
\item
$(M\cdot f \cdot M^{-1})_\eps$ fixes $\xg$  and has local degree $1$ at this point;
\item
$v$ is a bad direction;
\item
the image $w$ of $v$ in $\P^1(\sH_\eta)$
has an infinite orbit under  $(M\cdot f \cdot M^{-1})_\eta$.
\end{itemize}
\end{itemize}

Observe that (A3) implies $v$ to have an infinite orbit under  $\widetilde{(M\cdot f \cdot M^{-1})}_\eps$ so that 
$\xg$ lies in the Julia set of $(M\cdot f \cdot M^{-1})_\eps$.

It follows from the next lemma that $w$ lies in the support of the hole divisor 
$D_\eta(M\cdot f \cdot M^{-1})$.
\begin{lemma}\label{lem:baddy}
Suppose that $\xg$ is fixed by $f_\eps$ for some $\eps <(1)$, and $v\in T\xg$ is a bad direction. 
Then the image of $v$ in $\P^1(\widetilde{\sH}_\eps)$  lies in the support of the hole divisor $\tilde{D}_\eps(f)$.
\end{lemma}
A proof is given at the end of this section.

The limit of the $g$-rescaling adapted to $f$ is the rational map
$(M\cdot f \cdot M^{-1})_\eta$ over $\sH_\eta$ which has degree in $\{1,\cdots,d\}$. 
The hole divisor of the $g$-rescaling is $D_\eta(M\cdot f \cdot M^{-1})$.

Observe that if $M$ and $L$ are two $g$-rescalings which are equivalent
then $M$ is adapted to $f$ iff $L$ is adapted to $f$. 
In that case, Proposition~\ref{prop: compo} implies that the two limits $(M\cdot f \cdot M^{-1})_\eta$ and $(L\cdot f \cdot L^{-1})_\eta$ 
are conjugated by a Möbius transformation defined over $\sH_\eta$.

\begin{rmk}
The notion of rescaling limit was first introduced by Kiwi~\cite{zbMATH06455745KJ15} in the Archimedean case.
A rescaling in the sense of Kiwi is a $g$-rescaling in our sense for which $\eps=(1)$
and $\deg(M\cdot f \cdot M^{-1})_{(1)} \ge2$,  so that the limit is a complex rational map
of degree at least $2$.  
\end{rmk}

\begin{prop}
\label{prop:definedoverc}
Suppose $(M,\eta)$ is a $g$-rescaling adapted to $(f_n)$ whose limit
$(M\cdot f\cdot M^{-1})_\eta$ is conjugated over $\sH_\eta$ to a complex rational map. 

Then we have $\eta =(1)$.
\end{prop}

\begin{proof}
We may assume that $M=\id$, and 
\( f_{\eta} \) has complex coefficients. 
If \( \eta < (1) \), then $\lim \eta_n=0$, and 
we have \( |\Res(f_{\eta})| = 1 \), so that \( f_{\eta} \) has good reduction. This leads to a contradiction.    
\end{proof}

\begin{proof}[Proof of Lemma~\ref{lem:baddy}]
We may suppose that $f_\eps[z_0:z_1]=[P_\eps:Q_\eps]$ with $P_\eps, Q_\eps$ homogeneous polynomials of degree $d$
whose coefficients have all norm $\le 1$, and $v=[0:1]$. 
If $[0:1]$ does not lie in the hole divisor $\tilde{D}_\eps(f)$, then $\tilde{Q}_\eps(0,1) \neq 0$
so that the power series $z \mapsto P_\eps(z,1)/Q_\eps(z,1)$ has a radius of convergence equal to $1$. 
It follows that the image of the open unit ball $|z|<1$ is also a ball, hence $v$ cannot be a bad direction. 
\end{proof}

Let us explain how our definition works in concrete examples. 

\begin{example}\label{ex:concrete}
Let \( f_n(z) = e^{-n} z^3 + z^2 + nz \). By Example~\ref{eg: compute-fundascale}, the $g$-rescaling  $(\mathrm{id},\eps)$ with $\eps= (1/\log n)$
satisfies both conditions (A1) and (A2), hence is adapted to $f$. Write $a=(n)\in\sH_0$.
Its limit is $z^2+a_\eps$, and has degree $2$.
\end{example}

Here is a $g$-rescaling having a degree $1$ limit.
\begin{example}
\label{eg: deg1}
Let $f_n(z)=z+1+\frac{e^{-n}}{z}$. Set $\eps=(1/n)$ and $a=(e^{-n})$. Then $|a_\varepsilon|=e^{-1}$, and
\[
f_\varepsilon(z)=z+1+\frac{a_\varepsilon}{z}.
\]
Since the reduction satisfies $\tilde{f}_\varepsilon(z)=z+1$, the point $\xg$ has local degree $1$. Moreover, $0$ determines a bad direction at $\xg$, whose orbit under $f_{(1)}$ is infinite. Therefore, $(\mathrm{id},(1))$ satisfies the two conditions (A1) and (A3), hence defines a $g$-rescaling adapted to the sequence $(f_n)$.
\end{example}

\subsection{The fundamental rescaling}
\begin{thm}\label{thm:funda}
Pick any sequence of complex rational maps $(f_n)$ of degree $d\ge 2$.

Then there exists a $g$-rescaling $(M_\star,\eps_\star(f))$ adapted to $f$, and its limit has degree $d$.

Moreover, for any other $g$-rescaling $(M,\eta)$ adapted to $f$, we have $\eta \ge \eps_\star(f)$, and 
$(M,\eps_\star(f))$ is equivalent to $(M_\star,\eps_\star(f))$. In particular, there exists a unique
$g$-rescaling adapted to $f$ at scale $\eps_\star(f)$ up to equivalence.
\end{thm}

The (equivalence class of) $g$-rescaling defined in the above theorem is called the 
\emph{fundamental $g$-rescaling adapted to $(f_n)$.}

\begin{proof}
The existence of the $g$-rescaling follows by choosing for each $n$ a Möbius transformation
$M_n$ such that $|\Res(M_n\cdot f_n \cdot M_n^{-1})|=|\res(f_n)|$. Indeed, 
we have $\eps(M\cdot f \cdot M^{-1})=\eps_\star(f)$, and
Theorem~\ref{thm:trichotomy2} (2)
implies that $(M\cdot f \cdot M^{-1})_{\eps_\star(f)}$ has degree $d\ge 2$ and does not have potential good reduction. 
Hence $(M,\eps_\star(f))$ is a $g$-rescaling adapted to $f$, and the limit has degree $d$ as claimed. 

Replacing $f$ by $M\cdot f \cdot M^{-1}$ we may suppose that $\eps_\star(f)=\eps(f)$. 
Pick any $g$-rescaling $(L,\eta)$ adapted to $f$. We shall prove that $\eta\ge \eps_\star(f)$, and 
$(L,\eps_\star(f))$ is equivalent to $(\id,\eps_\star(f))$.

By Theorem~\ref{thm:trichotomy}, we have $\eps(L\cdot f \cdot L^{-1}) \le \eta$. 
It is clear when  $\deg (L\cdot f \cdot L^{-1})_\eta=1<d$. 
Otherwise, we get $\deg (L\cdot f \cdot L^{-1})_\eta\ge 2$, and 
(A2) implies that the map $(L\cdot f \cdot L^{-1})_\eta$ does not have potential good reduction.

If $\eps(L\cdot f \cdot L^{-1}) = \eta$, then $\deg (L\cdot f \cdot L^{-1})_\eta= d$
by Theorem~\ref{thm:trichotomy}, and
Theorem~\ref{thm:unique-rescaling} (2) and (3) directly apply:
we get  $\eps(L\cdot f \cdot L^{-1})=\eps_\star(f)$ and $\deg(L_{\eps_\star(f)})=1$ so that 
$(L,\eps_\star(f))$ is equivalent to $(\id,\eps_\star(f))$.

Suppose now that $\eps(L\cdot f \cdot L^{-1}) < \eta$. 
To simplify notation write $g= L\cdot f \cdot L^{-1}$.
The residue map $\widetilde{g}_{\eps(g)}$
is then semi-conjugated to $g_\eta$ under the canonical map 
$\P^1_{\widetilde{\sH}_{\eps(g)}}\to \P^1_{\sH_\eta}$.
Recall also that
the local degree of $g_{\eps(g)}$ at $\xg$
is greater or equal to $\deg(g_\eta)$.  It follows from (A1) that $\xg$ is fixed. 
We shall prove that $g_{\eps(g)}$ does not have potential good reduction. Granting this claim, then Theorem~\ref{thm:unique-rescaling} (1) and (3)
imply $\eps(g)=\eps_\star(g)$ and 
$\deg(L_{\eps_\star(g)})=1$,
and we conclude as before that  $(L,\eps_\star(f))$ is equivalent to $(\id,\eps_\star(f))$.

Recall that $g_{\eps(g)}$ is a rational map which does not have good reduction by Theorem~\ref{thm:trichotomy}, and 
is fixing $\xg$. Suppose by contradiction that  $g_{\eps(g)}$ has potential good reduction. Then its local degree at $\xg$ is necessarily $1$, and 
the hole divisor $\tilde{D}_{\eps(g)}(g)$ is supported on a  single point $w$ corresponding to the direction at $\xg$
pointing towards the unique totally invariant point of $g_{\eps(g)}$. Observe that $w$ is fixed by $\tilde{g}_{\eps(g)}$.
Since the degree of $\tilde{g}_{\eps(g)}$ and $g_\eta$ are the same, $D_\eta(g)$ is the push-forward of  
$\tilde{D}_{\eps(g)}(g)$, hence is a fixed point by $g_\eta$. This contradicts (A3) and concludes the proof of the theorem.
\end{proof}

\subsection{Type II fixed points of local degree at least $2$ and $g$-rescaling}
Pick any scale $\eps$, and any sequence of complex rational maps $(f_n)$. 
Its limit $f_\eps$ is a (possibly constant) rational map of degree $\delta \le d$ defined over $\sH_\eps$.
Its reduction $\tilde{f}_\eps$ is a rational map of degree $\le \delta$. 
In homogeneous coordinates, $f_\eps=[P:Q]$ where $P$ and $Q$ are homogeneous polynomials 
of degree $\delta$ whose maximum of the modulus of their coefficients is $1$ and
$\tilde{f}_\eps=[\tilde{P}:\tilde{Q}]$.
Recall that $\tilde{f}_\eps$ is non constant iff $f_\eps(\xg)=\xg$.
\begin{thm}\label{thm:baby2}
Pick any sequence of complex rational maps $(f_n)$ of degree $d\ge 2$. Let $(M,\eps)$ be any 
$g$-rescaling adapted to $f$ such that  the limit  $(M\cdot f\cdot M^{-1})_\eps$ 
fixes some  type II point  $x\in\P^{1,\an}_{\sH_\eps}$, and  has local degree at least $2$ at $x$.  

Then there exists a $g$-rescaling $(L,\eps_+)$ with $\eps_{+} >\eps$ such that the following holds:
\begin{itemize}
\item[(B1)]
$(L,\eps_+)$ is adapted to $f$;  
\item[(B2)]
$(L\cdot M^{-1})_\eps (x)=x_g$ (so that $(L,\eps)$ is equivalent to $(M,\eps)$);
\item[(B3)]
$\deg(\widetilde{L\cdot f \cdot L^{-1}})_\eps=\deg(L\cdot f \cdot L^{-1})_{\eps_+}$. 
\end{itemize}
This $g$-rescaling is also unique up to equivalence. 
\end{thm}

\begin{rmk}
Write $g:=L\cdot f \cdot L^{-1}$. Then
the rational map $g_{\eps_+}$ has the same degree as  $\tilde{g}_\eps$, and
since 
\[g_\eps(\xg) = (L\cdot M^{-1})_\eps \circ(M \cdot  f \cdot M^{-1})_\eps\circ (L\cdot M^{-1})^{-1}_\eps (\xg)
\]
the rational map  $\tilde{g}_\eps$ is conjugated to the tangent map $T_x(M\cdot f\cdot M^{-1})_\eps$.
\end{rmk}

Any $g$-rescaling $(L,\eps_+)$ as in the theorem is called \emph{the baby $g$-rescaling}
of $(M,\eps)$ associated with the fixed point $x$. 

\begin{proof}
Replacing $f$ by $M\cdot f\cdot M^{-1}$
we may suppose that $M=\id$. Since $\PGL(2,\sH_\eps)$ acts transitively on the set of type II points, 
we may conjugate further $f$ so that $M=\id$ and $x=\xg$.

We first construct a $g$-rescaling satisfying (B1) -- (B3).
By assumption $\deg(f_\eps)\ge2$, and since $(\id,\eps)$ is adapted to $f$, then $f_\eps$
does not have potential good reduction so that
\[
\delta:= 
\deg(T_{\xg}f_\eps)=
\deg(\tilde{f}_\eps) < \deg(f_\eps)~.
\] 
Choose any sequence of complex rational maps $(h_n)$ of degree $\delta$
such that $\tilde{h}_\eps = \tilde{f}_\eps$. Observe that $h_\eps$ has good reduction, hence
$\eps < \eps(h) \le \eps_\star(h)$. 

Choose $L_n$ such that $\eps(L\cdot h \cdot L^{-1})=\eps_\star(h)$. 
Since $\delta\ge2$, Theorem~\ref{thm:unique-rescaling} (3) implies $L_{\eps(h)} \in \PGL(2,\sH_{\eps(h)})$. 
 By Theorem~\ref{thm:trichotomy}, this implies $L_{\eps} \in \PGL(2,\sH_{\eps})$ has good reduction, hence
 $L_\eps(\xg)=\xg$, 
 and 
 $(L,\eps)$ is equivalent to $(\id,\eps)$ and (B2) holds. 
  
 Since $\tilde{f}_\eps=\tilde{h}_\eps$, we get $f_\eta = h_\eta$
 for all $\eta > \eps$ (by naturality of Figure~\ref{diagram:residue}), and it follows that 
 the limits of $L\cdot h \cdot L^{-1}$, and $L\cdot f \cdot L^{-1}$
 are the same at scale $\eps_\star(h)$. Set $\eps_+:= \eps_\star(h)$, and recall
 that $\delta\ge2$ by assumption.

 When $\eps_+=(1)$, then  $(L\cdot h \cdot L^{-1})_{\eps_+}=(L\cdot f \cdot L^{-1})_{\eps_+}$
does not have potential good reduction by convention. 
When  $\eps_+<(1)$, then Theorem~\ref{thm:unique-rescaling}
 shows that  $(L\cdot h \cdot L^{-1})_{\eps_+}$
does not have potential good reduction.  
This proves (B1), and (B3) follows since  
 \[
 \deg(\widetilde{L\cdot f \cdot L^{-1}})_\eps=
 \deg(\widetilde{L\cdot h\cdot L^{-1}})_\eps
 = \deg(L\cdot h \cdot L^{-1})_{\eps_+}= \deg(L\cdot f \cdot L^{-1})_{\eps_+}~.
\]
 Now we discuss the uniqueness of the $g$-rescaling. 
 Conjugating $f$ by a suitable sequence of Möbius transformations, we may suppose that
 $(\id,\eps_+)$ satisfies (B1), (B2)  and (B3), and 
pick  another
 $g$-rescalings $(L,\eta_+)$ satisfying the same three conditions. 
 We may assume that $\eps_+\le \eta_+$.

 We claim that $L_{\eps_+}$ is not constant. Grant this claim. 
By Proposition~\ref{prop: compo},
 $(L\cdot f\cdot L^{-1})_{\eps_+}=L_{\eps_+}\cdot f_{\eps_+}\cdot L^{-1}_{\eps_+}$ 
 is conjugated to
 $f_{\eps_+}$  which does not have potential good reduction by (B1). 
 By Theorem~\ref{thm:trichotomy}, it follows that 
 $\deg(L\cdot f\cdot L^{-1})_{\eta}<\deg(\widetilde{L\cdot f \cdot L^{-1}})_\eps$ for all $\eps_+<\eta$, which implies $\eta_+\le\eps_+$, 
 so that $\eta_+=\eps_+$, and both $g$-rescalings are equivalent.

 To prove the claim, we observe that the hole divisor $\widetilde{D}_\eps(f)$ has the same degree
 as $D_{\eps_+}(f)$ since by condition (B3) the rational maps $\widetilde{f}_\eps$ and $f_{\eps_+}$ have the same degree. 
 It follows that the image of $\widetilde{D}_\eps(f)$ under the canonical map $\P^1(\widetilde{\sH}_\eps)\to \P^1(\sH_{\eps_+})$
   is equal to $D_{\eps_+}(f)$. 
 By Lemma~\ref{lem:period}, it follows that for all integer $l$ that the divisor $\widetilde{\Per}_l(f)_{\eps}$
is also mapped to  $\Per_l(f)_{\eps_+}$. The same argument shows that $\widetilde{\Per}_l(L\cdot f\cdot L^{-1})_{\eps}$
is also mapped to  $\Per_l(L\cdot f\cdot L^{-1})_{\eps_+}$ since $\eps_+\le \eta_+$. 
\[\begin{tikzcd}
\P^1(\widetilde{\sH}_\eps)
 \arrow[r, "\tilde{L}_\eps"] \arrow[d] 
 &
\P^1(\widetilde{\sH}_\eps)\arrow[d] 
\\
\P^1(\sH_{\eps_+})
 \arrow[r, "L_{\eps_+}"] 
&
\P^1(\sH_{\eps_+})
 \end{tikzcd}                    
\]
Now choose $l$ large enough so that the support of both divisors  $\Per_l(f)_{\eps_+}$ and  $\Per_l(L\cdot f\cdot L^{-1})_{\eps_+}$
is not reduced to a single point. Since the diagram above commutes, $L_{\eps_+}$ maps  $\Per_l(f)_{\eps_+}$ to $\Per_l(L\cdot f\cdot L^{-1})_{\eps_+}$, 
hence $L_{\eps_+}$ is not a constant map, as required.
\end{proof}

\subsection{Type II Julia fixed points of local degree $1$ and $g$-rescaling}
Let $(f_n)$ be  a sequence of complex rational maps  of degree $d\ge 2$, and $(M,\eps)$  a 
$g$-rescaling adapted to $f$ such that  the limit  $(M\cdot f\cdot M^{-1})_\eps$ has degree $\ge 2$. 

Assume that
$x$ is a type II fixed point lying in the Julia set of $(M\cdot f\cdot M^{-1})_\eps$ having local degree $1$, 
and pick  a bad direction $v\in Tx$ having an infinite orbit. 

The tangent map induced by $(M\cdot f\cdot M^{-1})_\eps$ on $Tx$ is a Möbius transformation
defined over $\widetilde{\sH}_\eps$: it is either conjugated to a translation and we set $\la=1$; 
or to a homothety. In the latter case, the map is determined (up to conjugacy) by a sequence
of complex homotheties $z\mapsto \la_n z$ with $|\la_n|\le 1$ (when $|\la_n|=1$, we also impose
$\Im(\la_n)\ge0$ so that $\la_n$ is always uniquely determined). 
In that case, we write $\la := \lim_\om \la_n \in \bar{\D} \subset \C$.

\begin{thm}\label{thm:baby1}
Assume that either $\la$ is not  a root of unity, or $\la=1$. 
Then there exists a $g$-rescaling $(L,\eps_+)$ with $\eps_{+} >\eps$ such that the following holds:
\begin{itemize}
\item[(C1)]
$(L,\eps_+)$ is adapted to $f$;  
\item[(C2)]
$(L\cdot M^{-1})_\eps (x)=x_g$ (so that $(L,\eps)$ is equivalent to $(M,\eps)$);
\item[(C3)] 
the image of $T_x(ML^{-1})\cdot v\in\P^1({\widetilde{\sH}_\eps})$ in $\P^1(\sH_{\eps_{+}})$ 
lies in the support of the hole divisor of  $(L\cdot f \cdot L^{-1})_{\eps_+}$, and
has an infinite orbit under this rational map;
\item[(C4)]
$\deg (\widetilde{L\cdot f \cdot L^{-1}})_\eps= \deg(L\cdot f \cdot L^{-1})_{\eps_+}=1$.
\end{itemize}
Moreover, this $g$-rescaling is unique up to equivalence. 
\end{thm}

\begin{rmk}
Observe that in the case $\la$ is a primitive root of unity of order $N$, then 
we can apply the previous theorem to $f^N$. Note also that $\eps_+ = (1)$ iff $\la \neq 0$. 
\end{rmk}

As in the previous section, any $g$-rescaling $(L,\eps_+)$ as in the theorem is called \emph{the baby $g$-rescaling}
of $(M,\eps)$ associated with the fixed point $x$ and the bad direction $v$ of infinite orbit. 
In the remaining of the paper, we shall abuse terminology and say that the baby $g$-rescaling is only attached to 
to the fixed point $x$ without referring to the bad direction. 

\begin{proof} 
 Let us first construct the $g$-rescaling. 
 Conjugating by a suitable sequence of Möbius transformations  equivalent to $M$, we may assume that 
  $M=\id$ and   $x=\xg$. 
 It follows that $f_\eps$ is a rational map of degree $\ge 2$ on $\P^1_{\sH_\eps}$ fixing $\xg$, and 
 its residue map $\tilde{f}_\eps$ is a Möbius transformation of infinite order. The direction $v$ is identified with a point
 in $\tilde{v}\in \P^1(\widetilde{\sH}_\eps)$, and since $v$ is a bad direction, $\tilde{v}$ lies in the hole divisor 
 $\tilde{D}_\eps(f)$ by Lemma~\ref{lem:baddy}.
 
 Conjugating further by a suitable sequence of Möbius transformations
 (defined over $\sH^\circ_\eps$), we may assume in some affine coordinates $[z:1]$ that either 
  $\tilde{f}_\eps(z)=z+1$ and $v=1$; or $\tilde{f}_\eps(z)=\tilde{\mu}_\eps z$ with $\mu \in \sH_0$, and $v=1$. 
 
 In the former case, we set $\eps_+=(1)$. Then $f_{(1)}(z)=z+1$, and $v\in \P^1(\sH^\circ_\eps)$ is mapped to $1\in\P^1(\sH_{(1)})=\P^1(\C)$ which is a point of infinite orbit in the Riemann sphere. 
 It follows that $(\id,(1))$ is adapted and the three other conditions (C2) -- (C4) are satisfied. 
  
Suppose now we are in the latter case. We may represent $\mu$ by a sequence of complex numbers $(\mu_n)$
and assume that $|\mu_n|\le 1$ (possibly by conjugating $f$ by the inversion $z\mapsto z^{-1}$).
 Consider the scale $(\theta(\mu))$. When $(\theta(\mu))<(1)$, we set $\eps_+=(\theta(\mu))$ and observe that $|\mu|_{\eps_+} = 1/e<1$ in $\sH_{\eps_+}$
 so that $f_{\eps_+}(z)=\mu_{\eps_+} z$ does not have potential good reduction. Again $v$ is mapped to $1$ in $\P^1(\sH_{\eps_+})$
 and has infinite orbit,  hence (C1) -- (C4) are satisfied. 
  
 When $(\theta(\mu))=(1)$, then $f_{(1)}(z)=\la z$ (with $\la$ as in the statement of the theorem). 
 If $\la$ is not a root of unity, set $\eps_+=(\theta(\mu))=(1)$ as above. Then $v=1$ has an infinite orbit which implies  (C1) -- (C4). 
 
 When $\la$ is a root of unity, then  $\la$ is equal to $1$ by assumption. Set $\eps_0 = \theta(|\mu-1|)$.  
 Observe that $\eps_0> (0)$ since $\tilde{f}_\eps$ is of infinite order.
Choose $a_n\in \C^*$ and $b_n\in \C$ such that 
$b_n = (1-\mu_n)^{-1}$ and $a_n+b_n=1$, and set $L_n(z) = a_n z + b_n$. 
 It follows that $(L\cdot f \cdot L^{-1})_{(1)}= z+1$, and the image of $v$
 has infinite orbit. 
 The $g$-rescaling $(L,(1))$ satisfies all requirements. 
 
 \smallskip

 Let us now prove the uniqueness. We suppose that $M=\id$,  and assume that 
 $(\id, \eps_+)$ and $(L,\eta_+)$ with $\eps_+\le \eta_+$ both satisfy all  conditions (C1) -- (C4). 
 Note that in particular $x=\xg$.
 We claim that $\deg(L_{\eps_+})=1$. Grant this claim. Then $(L,\eps_+)$ is equivalent to $(\id,\eps_+)$ hence
 is adapted to $(f_n)$. When $\eps_+=(1)$ then we get $\eta_+=(1)$ and  $(L,\eta_+)$ and $(\id,\eps_+)$ are equivalent as required. 
 When $\eps_+<(1)$, then by (A2) $(L\cdot f\cdot L^{-1})_{\eps_+}$ does not have good reduction on $\P^1_{\sH_{\eps_+}}$
so that $\deg(L\cdot f\cdot L^{-1})_{\eta}=0$ for all $\eta >\eps_+$. It follows that $\eta_+=\eps_+$ also in this case and 
$(L,\eta_+)$ is  equivalent to $(\id,\eps_+)$ as well.

To prove that $\deg(L_{\eps_+})=1$ we consider the following commutative diagram:
\[\begin{tikzcd}
\P^1(\widetilde{\sH}_\eps)
 \arrow[r, "\tilde{L}_\eps"] \arrow[d] 
 &
\P^1(\widetilde{\sH}_\eps)\arrow[d] 
\\
\P^1(\sH_{\eps_+})
 \arrow[r, "L_{\eps_+}"] 
&
\P^1(\sH_{\eps_+})
 \end{tikzcd}                    
\]
Note that $\deg(\tilde{L}_{\eps})=1$ and by (C3) the image of $v$  in 
$\P^1(\sH_{\eps_+})$  has an infinite orbit under $f_{\eps_+}$. Similarly 
the image of $\tilde{L}_\eps(v)\in\P^1(\sH_{\eta_+})$  has an infinite orbit under $f_{\eta_+}$
which implies to have an infinite orbit under $f_{\eps_+}$ as well. 
We conclude that $L_{\eps_+}$ maps an infinite orbit to another one hence cannot have degree $0$.
\end{proof}

%%%%%%%%%%%%%%%%%%%

\subsection{$g$-rescalings with varying scales}

Let $(M,\eps)$ be any $g$-rescaling adap\-ted to $(f_n)$. 
We shall describe all $g$-rescalings adapted to $(f_n)$ of the form 
$(M,\eta)$ with $\eta < \eps$.

Recall that a  baby $g$-rescaling
of $(M,\eps)$ is any $g$-rescaling adapted to $(f_n)$ which is obtained
by choosing a  type II fixed Julia point of $(M\cdot f\cdot M^{-1})_\eps$ and apply either 
Theorem~\ref{thm:baby2} or~\ref{thm:baby1}.

The following theorem shows that all $g$-rescalings adapted to $(f_n)$ are descendants of the fundamental $g$-rescaling, i.e., 
obtained by iterating the baby construction. 

\begin{thm}\label{thm:chain}
Let $(M,\eps)$ be any $g$-rescaling adapted to a sequence of complex rational maps
$(f_n)$ of degree $d\ge2$. Suppose that $\eps_\star(f) < \eps$. 

There exists an increasing finite set of scales 
$\eps_\star(f)=\eps_0< \eps_1 < \cdots <\eps_m=\eps$ such that 
\begin{enumerate}
\item
 for all $i\ge 0$, $(M,\eps_i)$ is adapted to $(f_n)$;
\item
for all $i\ge1$ we have 
\[\deg(M\cdot f\cdot M^{-1})_{\eps_i} = \deg(\widetilde{M\cdot f\cdot M^{-1}})_{\eps_{i-1}}
< \deg(M\cdot f\cdot M^{-1})_{\eps_{i-1}};\]
\item 
for all $i\le m-1$, and for all scale $\eps_i < \eta < \eps_{i+1}$, the limit map $(M\cdot f\cdot M^{-1})_{\eps_i}$ has good reduction.
\end{enumerate}
Moreover, $(M,\eps_i)$ is a baby $g$-rescaling of $(M,\eps_{i-1})$ for all $i\ge 1$.
\end{thm}

\begin{rmk}
It follows from (2) that 
$\xg$ is a fixed point of 
$(M\cdot f\cdot M^{-1})_{\eps_i}$ for all $i=0, \cdots,m-1$.
\end{rmk}

\begin{rmk}\label{rmk:length}
It is convenient to define the length of the adapted $g$-rescaling $(M,\eps)$ as the integer $m$ such that $\eps_m=\eps$. 
Theorem~\ref{thm:funda} states that there exists a unique adapted $g$-rescaling of length $0$
up to equivalence. Note that any baby $g$-rescaling of a
$g$-rescaling of length $m$ has length $m+1$.
\end{rmk}

\begin{proof}[Proof of Theorem~\ref{thm:chain}]
It is sufficient to treat the case $M=\id$. 
Assume we have proved:
\begin{lemma}\label{lem:adapt-cool}
We have $\eps(f) = \eps_\star(f)$ and 
the $g$-rescaling $(\id,\eps_\star(f))$ is adapted to $f$. 
\end{lemma}

Set $\eps_0=\eps_\star(f)<\eps$ and choose a sequence $(h_n)$ of complex rational maps of  degree
$d_1=\deg(\tilde{f}_{\eps_0})$ such that $\tilde{f}_{\eps_0}=\tilde{h}_{\eps_0}$. 
By Lemma~\ref{lem:upto}, we have $f_\eta=h_\eta$ for all $\eta>\eps_0$. Note that $h_{\eps_0}$
has good reduction so that $\eps_1:=\eps(h) > \eps_0$ by Theorem~\ref{thm:trichotomy}. 
We also infer that $f_\eta (=h_\eta)$ has good reduction of constant degree 
$d_1$ when $\eps_1>\eta>\eps_0$. In particular, we get $\eps \ge \eps_1$. 

\smallskip

Suppose first that $d_1=1$. 
Observe that $\deg(f_\eps) \le d_1$ since $\eps>\eps_0$, hence $f_\eps$ is a degree $1$ map. 
Note also that $\eps = \eps_1$ otherwise we would have $\eps_1<(1)$ and 
 $f_{\eps_1} =h_{\eps_1}$ would have good reduction
by Theorem~\ref{thm:trichotomy}, which is absurd.

We now show that $(\id,\eps)$ is a baby $g$-rescaling of $(\id,\eps_0)$. 
Since $(\id,\eps)$ is adapted to $f$, by (A3) there exists a scale 
$\eps' < \eps$ such that $\xg$ is fixed by $f_{\eps'}$ and has local degree $1$, and
there is a bad direction $v$ at $\xg$ whose image $w$ in $\P^1(\sH_\eps)$ has
an infinite orbit. The orbit of $v$ under $\tilde{f}_{\eps'}$ is also infinite, 
hence $\xg$ lies in the Julia set  so that $f_{\eps'}$ does not have
potential good reduction. We infer that $\eps'=\eps_0$.

The map $\tilde{f}_{\eps_0}$ is either conjugated to a translation or to a homothety determined by a 
sequence of complex dilatations $z\mapsto \la_n z$ with $|\la_n|\le 1$. We need to show that 
$\la:=\lim_\om \la_n \in\bar{\D}$ is equal to $1$ when it is a root of unity. 
If $\eps < (1)$, then $f_\eps$ is conjugated to a homothety of multiplier $|\mu|<1$ in $\sH_\eps$ so that $\la=\mu_\C=0$.
When $\eps=(1)$, then either $f_\eps$ is a homothety of infinite order with multiplier $\la$, so that $\la$ is not a root of unity;
or $f_\eps$ is a translation and $\la$ is the multiplier at the unique fixed point of $f_\eps$ hence is equal to $1$.

We now observe that $(\id,\eps)$ satisfy all conditions of Theorem~\ref{thm:baby1} so that
$(\id,\eps)$ is a baby $g$-rescaling of 
$(\id,\eps_0)$
thereby finishing the proof in this case.

\smallskip

Suppose next that $d_1\ge2$. Then we may apply Lemma~\ref{lem:adapt-cool} which implies $(\id,\eps_1)$ to be adapted
to $h$, hence to $f$. Observe that $(\id,\eps_1)$ satisfies all conditions of Theorem~\ref{thm:baby2} (with $M=L=\id$, $\eps= \eps_0$ and $\eps_+=\eps_1$). 
This implies that $(\id,\eps_1)$ is a baby $g$-rescaling of $(\id,\eps_0)$. 
When $\eps_1=\eps$, then the proof is complete. 
When $\eps_1<\eps$, then we repeat the above construction by choosing a new sequence of rational maps
$(g_n)$ such that $\tilde{g}_{\eps_1}=\tilde{h}_{\eps_1}=\tilde{f}_{\eps_1}$ and $\deg(g_n)=d_2:=\deg(\tilde{g}_{\eps_1})<d_1$ for all $n$.
When $d_2=1$, then the previous argument applies and the proof is complete. Otherwise $d_2=2$, and we proceed  until
either $d_m\ge2$ and $\eps_m=\eps$, or $d_m=1$. This concludes the proof.
\end{proof}

\begin{proof}[Proof of Lemma~\ref{lem:adapt-cool}]
Suppose by contradiction that $\eps':= \eps(f) < \eps_\star(f)$. Then $f_{\eps'}$ has potential good reduction 
by Theorem~\ref{thm:funda}. On the other hand $f_{\eps'}$ does not have good reduction by Theorem~\ref{thm:trichotomy}.
Denote by $x$ the support of the Julia set of $f_{\eps'}$: it is the unique type II point which is totally invariant. 
The Gauss point is fixed by $f_{\eps'}$ as $\deg(\tilde{f}_{\eps'}) \ge \deg(f_\eps)$, and and its local degree is $1$.
It follows that  $\deg(\tilde{f}_{\eps'}) = \deg(f_\eps)=1$.

Now the hole divisor $\tilde{D}_{\eps'}(f)$ is supported on the single direction at $\xg$ pointing towards $x$, and this direction 
is fixed. On the other hand, the image of $\tilde{D}_{\eps'}(f)$ in $\P^1(\sH_\eps)$ is $D_\eps(f)$ because the degree 
of $\tilde{f}_{\eps'}$ and $f_\eps$ are the same. We infer that the only point  in the support of the hole divisor  $D_\eps(f)$
is fixed which violates the condition (A3)  of adapted $g$-rescaling by Lemma~\ref{lem:baddy}.
\end{proof}

The following example take, from the introduction provides a concrete illustration of Theorem~\ref{thm:chain}.
\begin{example}
\label{eg: ezchain}
  Set \[f_n(z) = e^{-n}z^4 + n^{-1}z^3 + z^2.\]
and $\eps_\star:=(1/n)$. Then $f_{\eps_\star}(z)= az^4+bz^3+z^2$ with $|a|<1$ and $|b|=1$, and $\xg$ is fixed of local degree $3$ so that $f_{\eps_\star}$ does not have potential good reduction. We infer from Theorem~\ref{thm:funda} that $\eps(f)=\eps_\star(f)=\eps_\star$.

The baby $g$-rescaling associated with the fixed point $\xg$ is 
$(\mathrm{id}, \eps_1)$ with $\eps_1 = (1/\log n)$. Note that 
$f_{\eps_1}(z)= b'z^3+z^2$ with $|b'|<1$, and it is fixing $\xg$ which has local degree $2$.
The associated baby $g$-rescaling is $(\mathrm{id}, (1))$ which is  a Kiwi rescaling having limit $z^2$.
\end{example}

%%%%%%%%%%%%%
\section{The tree of adapted \texorpdfstring{$g$}{g}-rescalings}\label{sec:tree}
Let $(f_n)$ be a sequence of complex rational maps of degree $d\ge2$. 
It may occur that $(f_n)$ admits a single $g$-rescaling (the fundamental one), 
while some of its iterate admit several. 
Such a perspective naturally leads to the notion of cycles of $g$-rescalings.

We then organize the set of all cycles of $g$-rescalings up to equivalence of $(f_n)$
and prove it naturally defines a tree, and subsequently bound the number of adapted scales.

%%%%%%%%%%%%%

\subsection{Cycles of $g$-rescalings}
\label{sec: cycles}
A  cycle of Möbius transformations of period $l\in\N_{+}$ is a mapping \( M_* \colon \mathbb{Z} \to \PGL_2(\mathbb{C})^{\mathbb{N}} \) that factors through \( \mathbb{Z}/l\mathbb{Z} \), so that \( M_{i+l} = M_i \) for all \( i \in \mathbb{Z} \).

A  \emph{cycle of \( g \)-rescalings} is a pair \( (M_*, \eta) \), where \( M_* \) is a cycle of Möbius transformations, and \( \eta \) is a scale.  The smallest positive integer $l$ such that $M_{i+l}=M_{i}$ for all $i\in\Z$ is the period of the cycle. 
Observe that a $g$-rescaling is a cycle of period $1$.

We extend the notion of equivalence of $g$-rescaling as follows. Two cycles of $g$-rescalings \((M_*, \eta)\) and \((L_*, \eta')\) are \emph{equivalent} if  for all \( i \in \mathbb{Z} \), \( (M_{i}, \eta) \) is equivalent to \( (L_i, \eta') \) as $g$-rescalings. In particular, we have $\eta=\eta'$. 
\begin{defi}
\label{def: cycle}
We say that a cycle of $g$-rescalings $(M_*,\eta)$ of period $l$ is adapted to \( (f_n) \) if it satisfies the following conditions:  
\begin{enumerate}
    \item[(D1)] for all $i\in\Z$,  $(M_{i+1} \cdot f \cdot M_{i}^{-1})_{\eta}$ is of degree at least $1$ ;
    \item[(D2)] for all \( i \in \mathbb{Z} \),  the $g$-rescaling $(M_i,\eta)$ is adapted to $(f_n^l)$. 
\end{enumerate}
\end{defi}
In the remaining of this section, we collect some elementary facts on cycles of $g$-rescalings. 

\begin{lemma}\label{lem:semi-conj}
Suppose that $(M_{*},\eta)$ is a cycle of $g$-rescalings of period $l$ adapted to $(f_n)$. 
Then $g_i = (M_i\cdot f^l \cdot M_i^{-1})_\eta$ is a non-constant rational map defined over $\sH_\eta$ for all $i$, 
and all these rational maps are semi-conjugated over $\sH_\eta$.
\end{lemma}

\begin{proof}
By Proposition~\ref{prop: compo},
for each $i$ we have
\[(M_{i,n} \cdot f_n^l \cdot M_{i,n}^{-1})_{\eta}
= 
(M_{i,n} \cdot f_n \cdot M_{{i-1},n}^{-1})_{\eta}\circ
\cdots \circ
(M_{i+1,n} \cdot f_n \cdot M_{i,n}^{-1})_{\eta}
\]
hence
\[
(M_{i+1,n} \cdot f_n \cdot M_{i,n}^{-1})_{\eta}\circ g_i = g_{i+1} \circ (M_{i+1,n} \cdot f_n \cdot M_{i,n}^{-1})_{\eta}
~.\]
which implies the result.
\end{proof}

\begin{lemma}\label{lem:shift}
Suppose that $(M_{*},\eta)$ is a cycle of $g$-rescalings adapted to $(f_n)$. 
For any integer $e\in\N_+$ and for any $k\in\Z$, the cycle of $g$-rescalings 
$(M_{e*+k},\eta)$ is adapted to $f^e$.
\end{lemma}
Note that the period of $M_{e*+k}$ is the greatest common divisor of $e$ and $l$.

\begin{proof}
It is clear that $(M_{e*+k},\eta)$ is adapted to $f^e$ iff $(M_{e*},\eta)$ is also, so we may suppose $k=0$. 
By Proposition~\ref{prop: compo}, we have
\[
(M_{e(i+1)} \cdot f^e \cdot M_{ei}^{-1})_{\eta}
=
(M_{e(i+1)} \cdot f \cdot M_{e(i+1)-1}^{-1})_{\eta}
\circ\cdots\circ
(M_{ei+1}\cdot f \cdot M_{ei}^{-1})_{\eta}
\]
so that (D1) is satisfied.
We need to check that $(M_{ei},\eta)$ is adapted to $f^{el}$
which boils down to prove that $(L,\eta)$ is adapted to $f^e$
as soon as it is adapted to $f$.
To check this fact, observe that again by Proposition~\ref{prop: compo}, we have
\[
(L\cdot f^e \cdot L^{-1})_\eta=
\underbrace{(L\cdot f \cdot L^{-1})_\eta\circ \cdots\circ (L\cdot f \cdot L^{-1})_\eta}_{e \text{ times}}
\]
This implies (A1) and (A2), as a rational map has potential good reduction iff one of its iterate does (a rational map has potential good reduction iff its Julia set is reduced to a point).
Finally (A3) holds since the hole divisor of $L\cdot f^e \cdot L^{-1}$ is more effective than $D_\eta(L\cdot f \cdot L^{-1})$ by~\eqref{eq:hole-iterate}.
\end{proof}

\begin{rmk}
We shall see later that any $g$-rescaling adapted to $(f_n^l)$ gives rise to a unique cycle of $g$-rescalings
adapted to $(f_n)$ up to equivalence. 
This will imply that it is only necessary to check (D2) for some index $i$.
\end{rmk}

\begin{lemma}
\label{lem:equivcycle}
Let $(M_*,\eta)$, $(L_*,\eta)$ be two cycles of $g$-rescalings adapted to $(f_n)$. 
Then $(M_*,\eta)$ is equivalent to $(L_*,\eta)$ if and only \( (M_i, \eta) \) is equivalent to \( (L_i, \eta) \) for some $i\in \Z$.
\end{lemma}

\begin{proof}
Suppose, without loss of generality, that \( (M_0, \eta) \) is equivalent to \( (L_0, \eta) \), so that   $
(M_{0} \cdot L_{0}^{-1})_{\eta} \in \mathrm{PGL}_2(\sH_{\eta}).$
We may assume that both cycles have the same period $l$. 
Pick any integer $1\leq i\leq l$. 
Since 
$(M_{j+1}\cdot f\cdot M_{j}^{-1})_{\eta}$ 
has degree at least $1$ for all $j$, by Proposition~\ref{prop: compo},
we obtain that the map
\[(M_{i} \cdot f^i \cdot M_{0}^{-1})_{\eta}
= 
(M_{i} \cdot f \cdot M_{{i-1}}^{-1})_{\eta}
\circ\cdots \circ
(M_{1} \cdot f \cdot M_{0}^{-1})_{\eta}
\]
has degree at least $1$. 
Similarly, we get $\deg(L_{i,n}^{-1} \cdot f_n^i \cdot L_{0,n})_{\eta}\geq 1$. Finally, we have
\[
(M_{i} \cdot f^i \cdot M_{0}^{-1})_{\eta}
= (M_{i} \cdot L_{i}^{-1})_{\eta}
\circ (L_{i}^{-1} \cdot f^i \cdot L_{0})_{\eta}
\circ (L_{0}^{-1} \cdot M_{0})_{\eta}.
\]
which implies \( \deg (M_{i} \cdot L_{i}^{-1})_{\eta} = 1 \). Since both \( M_* \) and \( L_* \) are periodic, we conclude that \( (M_*, \eta) \) is equivalent to \( (L_*, \eta) \). 
\end{proof}

\begin{lemma}
\label{lem:equiv-adapt}
Suppose $(M_*,\eta)$ and $(L_*,\eta)$ are equivalent cycles of $g$-rescalings.
If $(M_*,\eta)$ is adapted to $(f_n)$ then $(L_*,\eta)$ as well.
\end{lemma}
\begin{proof}
Let $l$ be a common period  of  $M_*$ and $L_*$.
Since $M_*$ and $L_*$ are equivalent, $M_i \cdot L_i^{-1}\in \PGL(2,\sH_\eta)$ for all $i$, 
it  follows from Proposition~\ref{prop: compo} that
\[
(L_{i+1}\cdot f \cdot L_i^{-1})_\eta
=
(L_{i+1}\cdot M_{i+1}^{-1})_\eta\circ(M_{i+1}\cdot f \cdot M_i^{-1})_\eta\circ
(M_i\cdot L_i^{-1})_\eta
\]
has degree at least $1$. 
It is clear that $f^l$ is adapted to $L_0$ so that both conditions (D1) and (D2) are satisfied, and the result follows.
\end{proof}

\subsection{The tree of adapted cycles of $g$-rescalings}
\label{sec: simplicialtree}
An equivalence class of cycles of $g$-rescalings is said to be 
\emph{adapted}
to $(f_n)$ iff
it contains one cycle adapted to $(f_n)$. 
By Lemma~\ref{lem:equiv-adapt}, an equivalence class is adapted iff all cycles 
of $g$-rescalings in this class are adapted. 

We denote by $\cR(f)$ the set of equivalence classes of cycles of $g$-rescalings adapted to $(f_n)$. 
We endow $\cR(f)$ with a partial order $(M_*,\eps) \le (L_*, \eta)$
iff $\eps \le \eta$ and $(M_*,\eps)$ is equivalent to $(L_*, \eps)$.

Recall that a poset $(X,\le)$ is a (simplicial) tree rooted at $x_\star$, if
$x_\star$ is the unique minimal element of $X$; for any $x\in X$
the set $\{ x_\star \le y \le x\}$ is isomorphic (as a poset) to an
interval $\{1,\cdots, N\}$ of the integers endowed with its canonical ordering.
We say that a pair of points $x, x'$ is consecutive if $x<x'$ and
$\{ y,  x \le y \le x'\} = \{x, x'\}$. We attach
an edge $[x,x']$ to any consecutive points. 
\begin{thm}\label{thm:tree-rf}
The poset $\cR(f)$ is a simplicial tree rooted at the fundamental $g$-rescaling. 

Moreover, pick any consecutive points $(M_*,\eps)<(L_*,\eta)$, and denote by $l$ the period of $M_*$.
Then $(M_0\cdot f^l \cdot M_0^{-1})_\eta$ admits a type II periodic Julia point $x$,
and $(L_0,\eps)$ is the baby $g$-rescaling associated with $x$ and $f^l$. 
\end{thm}

\begin{example}
\label{eg: trivialtree}
When $\eps_\star(f)=(1)$, then $\cR(f)$ is reduced to a single point. 
\end{example}

\begin{proof}
We claim that the fundamental $g$-rescaling as defined in Theorem~\ref{thm:funda} is the unique minimal element
of $\cR(f)$. We may and shall assume that this $g$-rescaling is given by $(\id,\eps_\star(f))$. 
Recall from Corollary~\ref{cor:iterate} that $\eps_\star(f^l)=\eps_\star(f)$ for all $l\in \N_+$. 

Let $(M_*,\eps)$ be any cycle of $g$-rescalings of period $l$ adapted to $f$. 
For each $i\in\Z$, the $g$-rescaling $(M_i,\eps)$ is adapted to $f^l$. By Theorem~\ref{thm:funda} applied to $f^l$, we get $\eps\ge \eps_\star(f^l) = \eps_\star(f)$,
and  $(M_i,\eps_\star(f))$ is equivalent to $(\id,\eps_\star(f))$. 
Lemma~\ref{lem:equivcycle} implies that $(M_*,\eps_\star(f))$ and $(\id,\eps_\star(f))$ are equivalent, 
and this proves the fundamental $g$-rescaling is the unique minimal element of $\cR(f)$. 

\smallskip

We now prove that $\cR(f)$ has a tree structure. Let $(M_*,\eps)$ be any cycle of $g$-rescalings of period $l$ adapted to $f$. 
Theorem~\ref{thm:chain} applied to $f^{l}$ and $(M_0,\eps)$ yields a finite set of scales \[\eps_0=\eps_\star(f)< \eps_1 < \cdots < \eps_m=\eps\]
such that $(M_0,\eps_j)$ is adapted to $f^{l}$, 
\[\deg (M_0\cdot f^{l}\cdot M_0^{-1})_{\eps_j} > \deg  (M_0\cdot f^{l}\cdot M_0^{-1})_{\eps}\ge1\]
if $j < m$, and the map $(M_0\cdot f^{l}\cdot M_0^{-1})_\eta$ has good reduction if 
$\eta \in (\eps_\star(f),\eps)$ and $\eta \notin\{\eps_j\}_{j=0}^m$.
Observe that $\deg(M_{i+1}^{-1} \cdot f \cdot M_{i})_{\eta}\ge \deg(M_{i+1}^{-1} \cdot f \cdot M_{i})_{\eps}\ge 1$, for any $\eta<\eps$. 
It follows that for any $j$, the cycle of $g$-rescalings $(M_*,\eps_j)$ is adapted to $f$ as it satisfies both conditions (D1) and (D2). 
We have proved that the set of points in $\cR(f)$ that are lower or equal to $(M_*,\eps)$ contains
$(M_*,\eps_0) < \cdots < (M_*,\eps_m)=(M_*,\eps)$.

Pick any other cycle of $g$-rescalings $(L_*,\eta)$ which is adapted to $f$, and whose point in $\cR(f)$ is lower or equal to $(M_*,\eps)$. 
Then $\eta\le \eps$ and $(L_*,\eta)$ is equivalent to $(M_*,\eta)$.  Let $p$ be the least common multiple
of the periods of $M_*$ and $L_*$.
Then $(L_0\cdot f^p \cdot L_0^{-1})_\eta$
is conjugated to $(M_0\cdot f^m \cdot M_0^{-1})_\eta$.  If $\eta$ does not belong to $\{\eps_j\}_{j=0}^m$, then we also have that $(L_0\cdot f^p \cdot L_0^{-1})_\eta$ has potential good reduction, and this contradicts (D2).
We conclude that $\eta\in\{\eps_j\}_{j=0}^m$ hence the set of points in $\cR(f)$ that are lower or equal to $(M_*,\eps)$
is equal to $\{(M_*,\eps_j)\}_{j=0}^m$. We have proved that  $\cR(f)$ is a simplicial tree rooted at the fundamental $g$-rescaling. 

Now suppose that $(L_*,\eta)<(M_*,\eps)$ are consecutive points in $\cR(f)$. By what precedes, 
$(L_*,\eta)$ is equivalent to $(M_*,\eps_{m-1})$ and the theorem follows.
\end{proof}

%%%%%%%%%%%%%%%%%%

%%%%%%%%%

\subsection{Boundedness of the set of adapted scales}
In this section, we prove a uniform bound depending only on the degree on the set of scales appearing in adapted $g$-rescalings. Compare with~\cite[Theorem~B]{CG25}.
\begin{thm}\label{thm:boundscales}
The number of non-trivial scales $\eps<(1)$ for which there exists a $g$-rescaling $(M,\eps)$ adapted to some iterate $f^l$ with $l \in \mathbb{N}_+$ is bounded above by $2d - 2$. In particular, the total number of adapted scales is at most $2d - 1$.
\end{thm}
This bound is optimal, see \S\ref{sec: optimal}.
Endow $\cR(f)$ with the unique tree metric such that the length of any edge is equal to $1$. 
\begin{cor}\label{cor:boundlength}
The maximal length of a totally ordered segment in $\cR(f)$ is $\le 2d-2$.
\end{cor}

\begin{proof}
This is a direct consequence of the previous theorem since $\eps > \eps'$
if $(L_*,\eps) > (L_*',\eps')$ in $\cR(f)$.
\end{proof}

\begin{proof}[Proof of Theorem~\ref{thm:boundscales}]
Conjugating $f_n$ by suitable Möbius transformations, we may assume that $|\res(f_n)|=|\Res(f_n)|$ for all $n$, 
so that $\eps(f)=\eps_\star(f)$. In addition, we write $f_n[z_0:z_1]=[P_n:Q_n]$ 
where $P_n$ and $Q_n$ are homogeneous polynomials of degree $d$
normalized so that the maximum of the norms of all coefficients of $P_n$ and $Q_n$ is equal to $1$. 
In other words, we have $f=[P:Q]$ with $P, Q \in B[0] [z_0,z_1]$, where $B[0]$
is the ring of bounded sequences of complex numbers modulo equalities on an $\om$-big set, see \S\ref{sec:complexrobinsonfield}.

Consider the intermediate field extension $\C\subset K' \subset \sH_0$ generated by the coefficients of $P$ and $Q$, and 
let $K$ be its algebraic closure. It is still an intermediate extension $\C\subset K \subset \sH_0$ whose
transcendence degree over $\C$ is at most $2d-2$ by~\cite[Lemma~4.6]{CG25}.

Recall that for any scale $\eps$, there is a canonical map $B[0]\to \sH_\eps$.

\begin{lemma}\label{lem:special}
Let $(M,\eps)$ be a $g$-rescaling adapted to $f^l$ for some $l\in\N_+$. 

Then one can find a $g$-rescaling  $(L, \eps)$ which is equivalent to $(M,\eps)$ and such
that  the limit $(L\cdot f^l \cdot L^{-1})_\eps$ is a rational map whose coefficients belong
to the image of $K\cap B[0]$ in $\sH_\eps$.
\end{lemma}

Now pick any  $g$-rescaling $(M,\eps)$ adapted to $f^l$ for some $l\in\N_+$ such that $\eps < (1)$. 
Write $g:=(L\cdot f^l \cdot L^{-1})_\eps$.  We claim that 
$R^-_\eps \cap K \subsetneq R^+_\eps \cap K$. The theorem then follows from Theorem~\ref{thm:numbervaluationring}.

\smallskip

If $\deg(g)=1$, then by (A2) $g$ does not have potential good reduction hence is conjugated to a homothety
of multiplier $|\la|>1$ since $\eps<(1)$. By the previous lemma, one can find $\mu \in K\cap B[0]$ whose image
in $\sH_\eps$ is equal to $\la$, so that $\mu$ belongs to $R^+_\eps \cap K\setminus R^-_\eps \cap K$, and the claim follows.

Suppose $\deg(g)=2$. Since $g$ does not have potential good reduction, its Julia set contains at least two points. 
These two points $x,x'$ belong to the closure of the set of periodic type I points by~\cite[Theorem~B]{FRL10}. 
We may thus find  periodic type I points $p_1, p_2$ in a neighborhood of $x$, and  $p_3, p_4$ close to $x'$
such that $[p_1,p_2]\cap [p_3,p_4]=\emptyset$, and
$[p_1,p_3]\cap [p_2,p_4]$ is a non-trivial segment $[y,y']$. 
In other words, their cross-ratio
\[
\la :=(p_{1}, p_{3}; p_{2}, p_{4}) = \frac{(p_{1} - p_{2})(p_{3} - p_{4})}{(p_{1}- p_{4})(p_{3} - p_{2})},\]
satisfies 
$|\la^{-1}| =  \exp(d_\H(y,y'))>1$. 

By the previous lemma, the coordinates of the four points
belong to the image of $K\cap B[0]$ in $\sH_\eps$ so that 
one can find $\mu \in K\cap B[0]$ whose image
in $\sH_\eps$ is equal to $\la$, which concludes the proof.
\end{proof}

\begin{proof}[Proof of Lemma~\ref{lem:special}]
For any scale $\eps$, denote by $B_K[\eps]$ (resp. $\tilde{B}_K[\eps]$) the image of the ring $K\cap B[0]$ in $\sH_\eps$ (resp. in $\tilde{\sH}_\eps$).

We argue by induction on the length of the $g$-rescaling as defined in Remark~\ref{rmk:length}.
When the length is $0$, then there is nothing to prove since we assumed that $\eps=\eps(f)=\eps_\star(f)$
so that $(\id,\eps)$ is equivalent $(L,\eps)$, and $f$ (hence $f^l$) has its coefficients in $K\cap B[0]$. 

Suppose now that the length of $(L,\eps)$ is equal to $n\ge 1$, and consider the $g$-rescaling $(L,\eps')$
with $\eps'< \eps$ which is adapted to $f^l$ and of length $n-1$. By Theorem~\ref{thm:chain}, $(L,\eps)$ is a baby $g$-rescaling
of $(L',\eps')$.
By the induction hypothesis, 
$(L,\eps')$ is equivalent to a $g$-rescaling $(L',\eps')$ defined over $K\cap B[0]$, and the limit is defined
over $B_K[\eps']$. Set $g = (L' \cdot f^l \cdot (L')^{-1})_{\eps'}$.

Since $(L,\eps)$ is a baby $g$-rescaling, it is associated with a type II periodic Julia point $x\in \P^1(\sH_{\eps'})$
of $g$ of period $m$. When $\deg_x(g^m)=1$, then there is a critical value $y$
such that $y, g^m(y)$ and $g^{2m}(y)$ satisfy $x = [y, g^m(y)] \cap [g^m(y), g^{2m}(y)] \cap [y, g^{2m}(y)]$.
Since the three points can be chosen to have projective coordinates in $B_K[\eps]$, we may conjugate by a Möbius transformation with coefficients in  
$B_K[\eps]$ and assume that $x=\xg$. 
When $\deg_x(g^m)\ge2$, then we can find three  type I periodic points $y_1, y_2, y_3$ satisfying 
$x = [y_1,y_2] \cap [y_2,y_3]\cap [y_3,y_1]$, and suppose as well that $x=\xg$.

We now look at the tangent map $\tilde{h} := T_{\xg} g^m$. We claim that it is determined by two homogeneous polynomials 
over $\tilde{B}_K[\eps]$.  Indeed,
we can be write in homogeneous coordinates $g^m = [R:S]$
so that the coefficients of $R$ and $S$ lie in $B_K[\eps]$.
By Lemma~\ref{lem:intclosed}, up to multiplication by a nonzero element of $\tilde{B}_K[\varepsilon]$, we may assume that $\tilde{R}$ and $\tilde{S}$ admit factorizations
$\tilde{R} = H R_1,$ and  $\tilde{S} = H S_1$
in $\widetilde{\sH}_\varepsilon$, where $H$, $R_1$, and $S_1$ have coefficients in $\tilde{B}_K[\varepsilon]$, and $R_1$ and $S_1$ are coprime.
This proves the claim.

Recall that the baby $g$-rescaling is  obtained by choosing
a lift $h$ of $\tilde{h}$ as a rational map of the same degree defined over $\sH_0$, 
and then by taking its fundamental $g$-rescaling. By what precedes, we may choose $h$ with coefficients in $K\cap B[0]$. 
In order to conclude the proof, we need to prove the existence of $L\in \PGL(2,K)$ such that
$(L,\eps_\star(h))$ is adapted to $h$. 

Pick $L_1\in \PGL(2,\sH_0)$ such that $\eps_\star(h) = \eps(L_1 \cdot h \cdot L_1^{-1})$. 
Choose three distinct points $p_1, q_1$ and $r_1$ in $\P^1(\sH_0)$ whose image in $\P^1(\sH_{\eps_\star(h)})$
are periodic for the limit  $(L_1 \cdot h \cdot L_1^{-1})_{\eps_\star(h)}$. 
Then we can find three periodic points $p,q$ and $r$ 
of $h$ in $\P^1(\sH_0)$ such that $L_1(p)=p_1$, $L_1(q)=q_1$ and $L_1(r)=r_1$.
Since $p,q$ and $r$ are defined over $K$, we may choose $L\in \PGL(2,K)$ so that 
$L(p)=0$, $L(q)=1$ and $L(r)=\infty$. 
Then $L_1\cdot L^{-1} (0)= p_1$, $L_1\cdot L^{-1} (1)= q_1$ and $L_1\cdot L^{-1} (\infty)= r_1$ hence $L_1\cdot L^{-1} \in \PGL(2,\sH_{\eps_\star(h)})$.
 It follows that $(L_1,\eps_\star(h))$ is  
equivalent to $(L,\eps_\star(h))$ and the proof is complete.
\end{proof}

%%%%%%%%%%%%%%%%%
\section{Iteration on $\cR(f)$ and PCF vertices}\label{sec:more-on-tree}

We continue our exploration of the structure of $\cR(f)$. 
We explain that it is invariant under iteration, and its upper-level trees depend only on the limit at the corresponding vertex. We then describe the situation in the case of analytic families and use this fact to prove an
upper bound on the number of non PCF limits. 
This ends the proof of Theorem~\ref{thmint:1} from the introduction.

%%%%%%%%%%%%%%%%%%

\subsection{The tree of $g$-rescalings and iteration}
Recall that if a cycle of $g$-rescalings $(M_*,\eps)$ is adapted to $f$, then 
$(M_{l*},\eps)$ is also adapted to $f^l$ for any $l\in \N_+$, so that there is a canonical map
$\cR(f)\to \cR(f^l)$.

A tree isomorphism  is a bijective map which preserves the order relation. 

 \begin{thm}\label{thm:periodic}
Let $f=(f_n)$ be a sequence of complex rational maps of degree $d\ge2$. 
\begin{enumerate}
\item
The map $\sigma_f(M_*,\eps) = (M_{*+1},\eps)$ on cycles of $g$-rescalings induces
a tree isomorphism on $\cR(f)$ fixing the root. 
\item
For any positive integer $l\in \N$, the canonical map $\cR(f) \to \cR(f^l)$ 
is a tree isomorphism which conjugates $(\sigma_f)^l$ to $\sigma_{f^l}$.
\end{enumerate}
\end{thm}

\begin{rmk}\label{rem:idiot}
Suppose 
 $(M_*,\eps)$ is a cycle of $g$-rescalings of period $m$ adapted to $f$. Then its associated vertex in $\cR(f)$ has period dividing $m$ under $\sigma_f$. 

 If the limit has degree $\ge2$, one can prove that any vertex in $\cR(f)$ of period $m$ under $\sigma_f$ can be represented by a cycle of the same period $m$.
 \end{rmk}

\begin{proof}
We prove (2). 
Observe first that the canonical map $\cR(f)\to \cR(f^l)$ is injective, as the equivalence relation on cycles of
$g$-rescalings does not depend on the sequence $(f_n)$. 
To prove the surjectivity, we rely on the next lemma.
\begin{lemma}\label{lem:goingtol}
Suppose that $(M,\eps)$ is a $g$-rescaling adapted to $f^l$ for some $l\in\N_+$. 
Then there exists a cycle of $g$-rescalings $(M_*,\eps)$ of period $l$ adapted to $f$
such that $M_0=M$, and this cycle is unique up to equivalence. 
\end{lemma}
Let $(M_*,\eps)$ be a cycle of period $m$ which is adapted to $f^l$. 
By Lemma~\ref{lem:goingtol} applied to $(M_0,\eps)$ which is adapted to $f^{ml}$, there exists a cycle of $g$-rescalings
$(L_*,\eps)$ 
of period dividing $lm$
and adapted to $f$ such that $L_0=M_0$.
By Lemma~\ref{lem:equivcycle} applied to $f^l$, $(L_{l*},\eps)$ and $(M_*,\eps)$ are equivalent
proving the canonical map $\cR(f)\to \cR(f^l)$  to be surjective. 

\smallskip

We now prove (1) and focus on $\sigma_f$. 
Let $(M_*,\eps)$ be any cycle of $g$-rescalings adapted to $f$. 
Then by Lemma~\ref{lem:shift}, $(M_{*+1},\eps)$ is also adapted. Moreover
$(M_*,\eps)$ is equivalent to $(L_*,\eps)$ iff $(M_{*+1},\eps)$ is equivalent to $(L_{*+1},\eps)$. It follows
that $\sigma_f(M_*,\eps):= (M_{*+1},\eps)$ descends to a well-defined bijective map on $\cR(f)$ with inverse
$\sigma^{-1}_f(M_*,\eps):= (M_{*-1},\eps)$. 
Observe that the period of $(M_{*+1},\eps)$ is the same as $(M_*,\eps)$. 
Since the fundamental $g$-rescaling is determined by a cycle of period $1$, 
it follows that it is fixed by $\sigma_f$. 

Note that $\sigma_f$ preserves the order relation on cycles of $g$-rescalings as well as the scale.
On any maximal chain $\mathfrak{C}$ of $\cR(f)$, the scales are strictly increasing, and
by Theorem~\ref{thm:boundscales}, $\mathfrak{C}$ is finite of cardinality at most $2d-1$. 
We conclude that  $\sigma_f\colon\cR(f)\to\cR(f)$ maps any maximal chain to another maximal chain, hence
induces a tree isomorphism.

Finally $(\sigma_f)^l=\sigma_{f^l}$ under the canonical identification $\cR(f) \to \cR(f^l)$ is clear, and the proof is complete. 
\end{proof}

\begin{proof}[Proof of Lemma~\ref{lem:goingtol}]
The uniqueness part directly follows  from Lemma~\ref{lem:equivcycle}. 
Let us construct the cycle of $g$-rescalings adapted to $f$. To simplify notation we assume $M=\id$.
By Theorem~\ref{thm:chain} applied to $f^{l}$, there exists a finite set of scales
$\eps_0:= \eps_\star(f^l) < \eps_1 < \cdots < \eps_n=\eps$ such that 
$(\id,\eps_i)$ is adapted to $f^l$. Moreover, $(\id,\eps_i)$ is a baby $g$-rescaling 
of $(\id,\eps_{i-1})$. 
In order to simplify the arguments we assume that $n=1$ referring to the end of the proof for the general case. 

By Theorem~\ref{thm:trichotomy}, $x_g$ is a fixed point of $f_{\eps_0}^l.$
Let $\{x_j\}:=\{f^j_{\eps_0}(\xg)\}$ be the orbit of the Gauss point under $f_{\eps_0}$. This forms a periodic cycle of 
type II Julia points of period dividing $l$. 

\smallskip

Suppose first that $\deg(T_{x_0}f^l_{\eps_0}) \ge 2$. All maps
$T_{x_j} f^l_{\eps_0}$ have the same degree.
We may thus apply Theorem~\ref{thm:baby2} to $f^l$, $M:=\id$ and $x_j$ with $j=1, \cdots, l-1$, and we obtain
a set of 
$g$-rescalings $(M_0, \eta_0), (M_1, \eta_1), \cdots, (M_{l-1},\eta_{l-1})$ adapted to $f^l$. 
By the uniqueness we may suppose that $M_0=M$ and $\eta_0=\eps_1$. 

We claim that $\eta_j= \eps_1$ for all $j$, and  the cycle of $g$-rescalings
$(M_*,\eps_1)$ is adapted to $f$ with $M_j= M_{j\, \mathrm{mod }\,l}$. 

To prove these claims, we choose for each $j\in\{0, \cdots, l-1\}$ a sequence of rational maps
$h_{j,n}$ such that
\[\tilde{h}_{j,\eps_0}=(\widetilde{M_{j+1} \cdot f \cdot M^{-1}_j})_{\eps_0}
\text { and } \deg (h_{j,n})=\deg(\tilde{h}_{j,\eps_0})~.\]
Observe that  $h_{j,\eps_0}$ represents $T_{x_j}f_{\eps_0}$
in projective coordinates given by $M_j$ as the source, and $M_{j+1}$ at the target.
Set $h_j = h_{j\, \mathrm{mod }\,l}$ for all $j\in\Z$.

By construction, we have  $\eta_j = \eps(g_j)=\eps_\star(g_j)$, with $g_j= h_{j+l-1}\circ \cdots \circ h_j$.
Since  $g_j =(h_{j+l-1} \cdots h_{l}) \circ (h_{l-1} \cdots h_j)$, and  $g_0=  (h_{l-1} \cdots h_j)\circ (h_{j+l-1} \cdots h_{l})$, 
Theorem~\ref{thm:semi-funda} applies, so that $\eps_1=\eps_\star(g_0)=\eps_\star(g_j)$,
and the degree of 
$(M_{j+1} \cdot f \cdot M_j^{-1})_{\eps_1}$
is at least $1$ for all $j$.
It follows that (D1) is valid.  Since 
$(M_0\cdot f^l \cdot M_0^{-1})_{\eps_1}= f^l_{\eps_1}$
and $(M,\eps_1)$ is adapted to $f$, then (D2) is also true
proving our claims.

\smallskip 

Suppose now that $\deg(T_{x_0}f^l_{\eps_0})=1$ so that $\deg(T_{x_j} f_{\eps_0})=1$ for all $j$. 
Recall that $(\id,\eps_1)$ is adapted to $f^l$, and observe that all maps $T_{x_j}f^l_{\eps_0}$ are conjugated.
It follows that 
there is a bad direction $v$ of $f^l_{\eps_0}$ at $x_0=\xg$ whose orbit under $f^l_{\eps_1}$ is infinite (so that $T_{x_0}f^j_{\eps_0}(v)$ is bad at $x_j$), and 
all maps 
$T_{x_j}f^l_{\eps_0}$ are simultaneously conjugated to a translation or to a linear map 
$z\mapsto \lambda z$ of infinite order (with $|\lambda|_{\eps_1}<1$ when $\eps_1< (1)$).
Note  that conjugating $f$ by a suitable sequence of Möbius transformations, we may (and shall) assume that 
the limit $f^l_{\eps_1}$ is either a homothety of infinite order, or a translation. 

We may thus apply Theorem~\ref{thm:baby1} to $f^l$, $M:=\id$, $x_j$ and to  $T_{x_0}f^j_{\eps_0}(v)$ for all $j=1, \cdots, l-1$, and we obtain
a set of 
$g$-rescalings $(M_0, \eps_1), (M_1, \eps_1)$, ...,  $(M_{l-1},\eps_1)$ adapted to $f^l$. 
By  the uniqueness we may suppose that $M_0=\id$. As before, (D2) is automatically true. 
By Proposition~\ref{prop: compo}, $(M_{i+1}\cdot f\cdot M_i^{-1})_{\eps_1}$ sends the infinite orbit of the bad direction at $x_j$ (given by the image of 
$T_{x_0}f^j_{\eps_0}(v)$ in $\P^1(\C)$) to the infinite orbit of the bad direction at $x_{j+1}$ up to finitely many exceptions, hence (D1) is true.

When $n$ is greater than $1$, then we proceed by induction, and construct in a similar way cycles
at scales $\eps_2, \eps_3, \cdots$ until we reach a cycle of $g$-rescalings at the required scale $\eps_n$.
\end{proof}

%%%%%%%%%

\subsection{Upper-level trees}
Let $s =(M_*,\eps)$ be any vertex in $\cR(f)$. An edge $e$ containing $s$ is said to be decreasing (resp. increasing) if 
the associated scale at the other endpoint of $e$ is $<\eps$ (resp. $>\eps$). Note that there exists a single decreasing edge, the one 
contained in the segment joining the fundamental $g$-rescaling to $s$. 

Our aim is to describe all increasing edges containing a given vertex. To simplify the discussion, we shall assume that
$s$ is represented by a cycle of period $1$ determined by a $g$-rescaling $(M,\eps)$.
By Theorem~\ref{thm:tree-rf}, if an increasing edge joins $(M,\eps)$  to $(L_*,\eta)$, then 
$(L_0,\eta)$ is a baby $g$-rescaling of $g:=(M\cdot f \cdot M^{-1})_{\eps}$
hence defines a type II periodic Julia point $x(e)$ of $g$ of some period $l$. 

When $\deg_{x(e)}(g^l)\ge2$, we say that $e$ is repelling; when  $\deg_{x(e)}(g^l)=1$, then we say that $e$ is indifferent.

The next result is a  direct consequence of Theorems~\ref{thm:baby2},~\ref{thm:baby1} and~\ref{thm:chain}.

 \begin{thm}\label{thm:tree-edge}
Let $(M,\eps)$ be a $g$-rescaling adapted to $f$, and denote by $g:= (M^{-1} \cdot f \cdot M)_\eps$ its limit. 
\begin{enumerate}
\item
The map $e \mapsto x(e)$ induces a bijection between the set of repelling increasing edges at $(M,\eps)$
and the set of repelling type II periodic  Julia points of $(M\cdot f \cdot M^{-1})_{\eps}$.
\item
The map $e \mapsto x(e)$ induces a surjection between the set of indifferent increasing edges at $(M,\eps)$
and the set of  indifferent type II periodic  Julia points of $(M\cdot f \cdot M^{-1})_{\eps}$. 
\end{enumerate}
\end{thm}

\begin{rmk}
It may happen that  different bad directions determine different baby $g$-rescalings, see \S\ref{sec:diffbad} for an example.
In particular, the map in (2) needs not be injective. 

Suppose that $(f_n)$ is induced by an analytic family as discussed later in~\S\ref{sec:analytic}. 
If $x$ is an indifferent type II fixed Julia point 
of the fundamental limit, it follows from Theorem~\ref{thm:family} below that $x$ defines a unique baby $g$-rescaling. 
In that case \emph{all} bad directions at $x$ determine the same adapted $g$-rescaling. 
\end{rmk}

The next example shows that an indifferent type II fixed Julia point might determine a cycle of $g$-rescalings of a higher period.
Compare with the assumption on the multiplier in Theorem~\ref{thm:baby1}, and with Remark~\ref{rem:idiot}.

\begin{example}\label{ex:cycle}
\normalfont
Let $f_n(z)=-(1-\tfrac{1}{n})z+\frac{e^{-n}}{z-1}.$ Set $\eps:=(1/n)$, and 
let $a,b\in \sH_{\eps}$ be  the elements represented by $(e^{-n})$ and $\bigl(-1+\tfrac{1}{n}\bigr)$, respectively. 
Note that 
$|a|<1$,   $|b|=1$, and  $f_{\eps}(z)=bz+\frac{a}{z-1}$. 
The Gauss point $\xg$ is fixed, and we have 
$\tilde f_{\eps}(z)=\tilde b z$.  Moreover, the direction defined by $1$ is a bad direction at $\xg$ whose forward orbit is infinite. 
It follows that $\xg$ is an indifferent fixed Julia point so that $f_\eps$ does not have potential good reduction and $(\id,1/n)$
is the fundamental rescaling.

Note  that we cannot apply Theorem~\ref{thm:baby1} to construct a baby $g$-rescaling from this fixed point as
$-1+\tfrac{1}{n}$ converges to $-1 \neq 1$. However, we have a $g$-rescaling adapted to $f^2$ which induces 
a cycle of $g$-rescalings  of period $2$, and is given explicitely as follows. 
Set
\[
M_{i,n}(z)=
\begin{cases}
\dfrac{-(n-1)^2}{2n-1}z+\dfrac{n^2}{2n-1}, & \text{if $i$ is even},\\[6pt]
\dfrac{(n-1)^3}{n(2n-1)}z+\dfrac{n^2}{2n-1}, & \text{if $i$ is odd}.
\end{cases}
\]
Direct computations yield
\[
\widetilde{(M_{i+1}\cdot f\cdot M_{i}^{-1})}_{\eps}=
\begin{cases}
\tilde b^{2}z+1, & \text{if $i$ is even},\\
\id, & \text{if $i$ is odd}.
\end{cases}
\]
From Lemma~\ref{lem:upto}, we infer 
$(M_{i}\cdot f^2\cdot M_i^{-1})_{\C}=z+1$
for all $i\in\N.$
Note that $M_{0,\eps}(1)=0$ defines a bad direction $v$ of  $T_{x_g}(M_0\cdot f^2\cdot M_0^{-1})_{\eps}$
whose image in $\mathbb{P}^1(\mathbb{C})$ has an infinite forward orbit. 
Hence $(M_0,(1))$ satisfies both conditions (A1) and (A3) for $(f^2)$, and 
$(M_*,(1))$ is a cycle of $g$-rescalings adapted to $(f_n)$ of period $2$.
\end{example}

\medskip

Our next result indicates that upper-level trees only depends on the limit of a given vertex. 
 \begin{thm}\label{thm:upper-tree}
Let $f=(f_n)$ (resp. $g=(g_n)$) be a sequence of complex rational maps of degree $d\ge2$ (resp. $\delta\ge2$).
Suppose that $(M,\eps)$ (resp. $(L,\eps)$) is a $g$-rescaling adapted to $f$ (resp. to $g$), and 
the limits $(M\cdot f\cdot M^{-1})_\eps$, $(L\cdot g\cdot L^{-1})_\eps$ are identical. 

The map
$(N_*,\eta)\mapsto(N_*\cdot M^{-1}\cdot L,\eta)$
is then a tree isomorphism 
between $\{x \in \cR(f), x \ge  (M,\eps)\}$ and $\{y \in \cR(g), y \ge  (L,\eps)\}$
sending $(M,\eps)$ to $(L,\eps)$, and conjugating $\sigma_f$ to $\sigma_g$.\end{thm}

\begin{proof}
It is only necessary to check that
$(N_*\cdot M^{-1}\cdot L,\eta)$
is adapted to $g$ and $> (L,\eps)$ whenever  $(N_*,\eta)$ is adapted and $> (M,\eps)$. 

Replacing $f$ by an iterate, we may assume that $(N_*,\eta)$ has period $1$, and we denote it by  $(N,\eta)$. 
Since it is assumed to be $> (M,\eps)$, we have $\eta>\eps$, and 
$\deg(N\cdot M^{-1})_\eps=1$.
Observe that
$(N\cdot M^{-1}\cdot L,\eps)$
is equivalent to $(L,\eps)$ since 
$\deg(N\cdot M^{-1}\cdot L\cdot L^{-1})_\eps=\deg(N\cdot M^{-1})_\eps=1$.

By Theorem~\ref{thm:chain}, there exist scales
$\eps= \eps_0 < \cdots < \eps_m=\eta$ such that 
$(N,\eps_i)$ is a baby $g$-rescaling of $(N,\eps_{i-1})$ for each $i$.

Write 
$L_+ =N\cdot M^{-1}\cdot L$.
Recall that we assumed $(M \cdot f \cdot M^{-1})_{\eps} = (L \cdot g \cdot L^{-1})_{\eps}$ so that conjugating by 
$(N\cdot M^{-1})_\eps$,
we obtain that 
$(N \cdot f \cdot N^{-1})_{\eps} =(L_+ \cdot g \cdot L_+^{-1})_{\eps}$.
It follows from Lemma~\ref{lem:upto} that 
\begin{equation}\label{eq:go-down}
(N \cdot f \cdot N^{-1})_{\eps_1} =(L_+ \cdot g \cdot L_+^{-1})_{\eps_1}.
\end{equation}
Hence, \((N,\eps_1)\) satisfies (A1)--(A3) for \(f\) if and only if \((L_+,\eps_1)\) satisfies (A1)--(A3) for $g$.
In particular, \((L_+,\eps_1)\) is adapted to $g$. 
By~\eqref{eq:go-down}, the limits of $g$-rescalings
$(N,\eps_1)$ for $f$ and $(L_+,\eps_1)$ for $g$ are identical, 
so that we can repeat the same argument. 
After $m$ steps, we obtain that $(L_+,\eta)$ is adapted to $g$
as required. 
\end{proof}

%%%%%%%%%%%%%%%%%%%%%

\subsection{Tree of $g$-rescalings for analytic families}\label{sec:analytic}
Let $\D$ be the unit disk in the complex plane. 
 A meromorphic family $(f_t)$  of rational maps of degree $d\ge 2$
 parametrized by $\D$ having a pole at $t=0$ is given in projective coordinates by two  
 homogeneous polynomials $P_t$ and $Q_t$ of degree $d\geq 2$ whose coefficients are
 holomorphic over $\D$ such that  $f_t[z_0:z_1] = [P_t : Q_t]$ for all $t\in \D^*$, and 
$P_t$ and $Q_t$ have no common factors for all $t \in \D^*$.

It follows that $|\Res(f_t)| = c |t|^k (1+o(1))$ for some $c>0$ and $k\in \N$. 
When $k$ is positive, then $\theta(|\Res(f_t)|)\asymp \theta(|t|)$.
We consider the field of formal Laurent series $\C((t))$, equipped with the $t$-adic norm normalized so that $|t| = e^{-1}$, and denote by 
$f_\na$ the rational map induced by the family $(f_t)$ over $\C((t))$.

Recall that $(f_t)$ is said to be bounded in moduli iff  $f_t$ induces a bounded family  in the moduli space of rational map of degree $d$ as $t\to0$.
By~\cite{zbMATH06455745KJ15} and~\cite{favre-blow-up}, this is equivalent to say that the rational map $f_{\na}$
has potential good reduction. This condition is also equivalent to the existence (after base change) of a holomorphic family of Möbius transformations
$(M_t)$ such that $g_t:=M_t\cdot f_t \cdot M_t^{-1}$ extends as a holomorphic family at $0$, i.e., $g_0$ is a rational map of degree $d$.

\begin{thm}\label{thm:family}
Let $(f_t)$ be a meromorphic family of rational maps having a pole at $t = 0$, and let $(t_n)_{n \in \N}$ be a sequence in $\D^*$ converging to $0$. Set $f_n := f_{t_n}$, and 
$\eps = \theta(|t_n|)$. 
Then, the following holds.
\begin{enumerate}
    \item The  map $f_\eps$ has potential good reduction iff $(f_t)$ is bounded in moduli. When it is the case, we have $\eps_\star(f)=(1)$ so that $\cR(f)$ is reduced to a point, and the only adapted scale is trivial.
    \item When $f_\eps$ does not have potential good reduction, then $\eps_\star(f)=\eps<(1)$, and the only other adapted scales is $(1)$.  Moreover, there is a natural bijective map 
 from the set of all type II periodic Julia points  of $f_\eps$ onto the set of non-fundamental cycles of $g$-rescalings adapted to $f$, and this map conjugates $f_\eps$ to $\sigma_f$.  
 \end{enumerate}    
\end{thm}

\begin{proof}
Suppose first that  $(f_t)$ is bounded in moduli. Conjugating by a a suitable family of Möbius transformations, we may suppose that 
$(f_t)$ is a holomorphic family at $0$ and $f_0$ has degree $d$. The limit of $f_n$ at the trivial scale $(1)$ is then equal to $f_0$. 
By Theorem~\ref{thm:unique-rescaling} (2), we conclude that  $\eps_\star(f)=(1)$. Since $\eps <(1)$, Theorem~\ref{thm:trichotomy}  implies that 
$f_\eps$ has good reduction. This gives one implication in (1).

\smallskip

Suppose now that $(f_t)$ is not bounded in moduli. Then $f_{\na}$ does not have potential good reduction, see, e.g.,~\cite[Theorem~1.5]{favre-blow-up}.
Now observe that there exists a unique isometric embedding $\C((t))\subset \sH_\eps$ sending $t$ to the image of the sequence $(t_n)$ in $\sH_\eps$ 
(note that $|t| = \lim_n |t_n|^{\theta(|t_n|)}=1/e$), and  $f_\eps$ is the base change of $f_\na$ under this field extension. 
It follows that $f_\eps$ does not have potential good reduction either, so that $\eps_\star(f)=\eps$. 
This implies (1).

Let $\L$ be the completed algebraic closure of $\C((t))$ in $\sH_\eps$. Recall from~\cite[Corollaires~3.7 et~3.14]{Po13} that we have a canonical continuous 
embedding $\P^{1,\an}_\L\subset \P^{1,\an}_{\sH_\eps}$. The Julia set of $f_\eps$ is included in $\P^{1,\an}_\L$. Indeed it is contained in the closure of the set of rigid periodic
points which are all defined over $\L$.  It follows that any type II fixed Julia point $x$ of $f_\eps$ belongs to $\P^{1,\an}_\L$ and the tangent map $T (f_\eps)_x$ is defined over
the residue field of $\L$ which is the field of complex numbers. By Proposition~\ref{prop:definedoverc}, we conclude that the baby $g$-rescaling $(L,\eta)$ attached to $x$ 
has a trivial scale $\eta =(1)$. This implies that the only non-fundamental adapted scale is $(1)$. 

In view of Theorem~\ref{thm:tree-edge}, to conclude the proof it is sufficient to argue that an indifferent type II fixed Julia point $x$ gives rise to a single
baby $g$-rescaling no matter which bad direction we choose. This follows from the fact that all bad directions and all fixed directions of $T(f_\eps)_x$ are defined over $\C$, 
hence stay different at the trivial scale $(1)$. 
\end{proof}

%%%%%%%%%%%%%%%%%%%%%%%%

\subsection{Flexible Lattès maps}\label{sec:flexible}
Recall that a flexible Lattès map (over an algebraically closed field $k$) is a rational map $f$ such that 
there exists an elliptic curve $E/k$, a two-to-one map $\pi\colon E\to \P^1_k$, a point $z_0\in E(k)$, and an integer $n$ such that 
$f \circ \pi (z) = \pi ( nz+z_0)$. Suppose that $\car(k)=0$, and pick any square integer $d=n^2\ge4$. 
For any Weierstrass equation $y^2 = x^3+ ax +b$  with $4a^3+27b^2\neq0$ determining an elliptic curve $E_{a,b}$, there is a rational map $\FL^+_{a,b}$ of degree $d$
such that
$\FL^+_{a,b}(x)=x(nP)$ where $nP$ is the $n$ sum of the point $P=(x,\pm y)\in E_{a,b}$.
Observing that $E_{a,b}$ is isomorphic to $E_{s^{-4}a,s^{-6}b}$, 
we obtain an algebraic family of flexible Lattès map $\FL^+_t$ parametrized by an affine curve $\La_+$. 
When $d$ is even, any flexible Lattès maps belongs to this family. When $d$ is odd, 
any pair of a Weierstrass equation and a $2$-torsion point on $E_{a,b}$ gives rise to a flexible Lattès map, and we get
another algebraic family $\FL^-_t$ parametrized by an affine curve $\La_-$. 
Any flexible Lattès map is then conjugated to an element in one of the two families $\FL ^\pm$. 
See~\cite[\S{5}]{zbMATH05354061} or~\cite[\S{6.5}]{silverman2010arithmetic}.

\begin{thm}\label{thm:flexible-lattes}
Let $(f_n)$ be a sequence of flexible Lattès maps of degree $d\ge2$, and denote by $E_n$ the complex elliptic curve associated with $f_n$.
\begin{enumerate}
\item
The fundamental scale $\eps_\star(f)$ is trivial iff $(E_n)$ converges in the moduli space of elliptic curve
$\H/\PSL(2,\Z)$.
\item
When $\eps_\star(f)$ is non-trivial, then there exist infinitely many non-equivalent cycles of $g$-rescaling adapted to $f$, and
all $g$-rescaling limits of $f$ are conjugated to a monomial map or to a Tchebychev polynomial. In particular, the only non-trivial adapted
scale is $\eps_\star(f)$.
\end{enumerate}
\end{thm}

\begin{proof}
Observe that  up to conjugacy we have $f_n = \FL^\pm_{t_n}$ for some parameter $t_n\in\La^\pm$ so that we may find a meromorphic family $(f_t)$ of flexible Lattès maps
parametrized by the unit disk and $t_n \to0$ such that $f_n=f_{t_n}$. By construction, we also get a meromophic family of elliptic curve $E_t$ 
associated with $f_t$. 
By Theorem~\ref{thm:family}, $\eps_\star(f)$ is trivial iff $f_t$ is bounded in moduli. When $(E_t)$ is bounded in moduli, then $f_n$ converges to 
the Lattès map associated with $E_0$ and $\eps_\star(f)=(1)$. When $(E_t)$ is not bounded in moduli, up to a base change, we may suppose
that the degeneration of $E_t$ is semi-stable and locally isomorphic to $E_t = \C^*/\langle z \mapsto t^qz \rangle$ for some $q\ge1$, see~\cite[\S C.12]{zbMATH05549721}. 
By~\cite[Proposition~5.2]{FRL10}, then 
the Julia set of $f_{\na}$ is a segment $I$, and the set of type II repelling periodic points is dense in $I$.
Each of them determines a $g$-rescaling so that there exists infinitely many non-equivalent cycles of $g$-rescaling adapted to $f$
by Theorem~\ref{thm:tree-edge}.
The tangent map at any periodic point in the interior of $I$ is conjugated to a monomial map
since both direction given by $\partial I$ are totally invariant. The tangent map of any extremal point of $I$ is a polynomial whose multipliers
are all but finitely many equal to the integer $\sqrt{d}$. It follows from~\cite{zbMATH07691681,zbMATH07680768} that it is either monomial or a Tchebychev polynomial. 
We conclude observing that all non-fundamental $g$-rescaling limits of $f$ are defined over $\C$ and associated with a periodic point in $I$. 
\end{proof}

%%%%%%%%%%%%%%%%%%%%%

\subsection{PCF vertices}
We say that a cycle of $g$-rescalings $(M_*,\eps)$ of period $l$ and adapted to $f$ is PCF
if there exists an integer $i\in\Z$ such that $(M_i\cdot f^l \cdot M_i^{-1})_\eps$ has degree at least $2$, and  
is a post-critically finite rational map. 
Since $(M_i\cdot f^l \cdot M^{-1}_i)_\eps$ is semi-conjugated to $(M_0\cdot f^l \cdot M^{-1}_0)_\eps$ by Lemma~\ref{lem:semi-conj}, 
we may always take $i=0$. 
Note also that  equivalent cycles are PCF simultaneously.

We may thus declare that a vertex $v\in \cR(f)$ is PCF, when any cycle of $g$-rescalings determining $v$ is itself PCF. 
\begin{thm}\label{thm:non-PCF}
If $\eps_\star(f)<(1)$, then the subset $\cR'(f)$ of non-PCF vertices forms a $\sigma_f$-invariant connected subtree of $\cR(f)$
containing the fundamental $g$-rescaling. 
\end{thm}
\begin{rmk}
Using Theorem~\ref{thm:upper-tree}, McMullen's rigidity theorem and the results
of \S\ref{sec:flexible}, one can show that a PCF vertex is either an endpoint of $\mathcal{R}(f)$ or has infinitely many outgoing edges that lead to ends.
\end{rmk}
Note that $\cR'(f)$ is empty iff $f$ is itself PCF.
The key step in the proof of Theorem~\ref{thm:non-PCF} is given by the next lemma, which is essentially due to Kiwi~\cite{zbMATH06455745KJ15}.

\begin{lemma} \label{lem:assocriti}
Suppose that \( (\mathrm{id}, \eps) \) is a \( g \)-rescaling adapted to \( (f_n)\). Let \( c_\eps \) be a critical point of \( f_\eps \) whose orbit is infinite. Then, there exist an integer $J\in\N$ and a  critical point \( c'\in\P^1(\sH_0) \) of \( f \) such that $ (f^{j}(c'))_\eta$ has an infinite orbit for all scale $\eta \le \eps$, and
\[  (f^{j}(c'))_\eps=f_\eps^{j+J}(c_\eps)\]
for all $j\in\N$.
\end{lemma}
\begin{proof}
By Theorem~\ref{thm:chain}, there exists a finite set of scales
$\eps_0:= \eps_\star(f) < \eps_1 < \cdots < \eps_m=\eps$ such that 
$(\id,\eps_i)$ is adapted to $f$. Moreover, $(\id,\eps_i)$ is a baby $g$-rescaling 
of $(\id,\eps_{i-1})$. 

We shall construct critical points for maps $f_{\eps_j}$ for all $j= 0, \cdots, m$. 
Set $c_m:= c_\eps$. 
Since the degree of $\tilde{f}_{\eps_{m-1}}=T_{\xg}f_{\eps_{m-1}}$ is equal to $f_{\eps_m}$, $c_m$ can be lifted to a critical point
$\tilde{c}_{m-1}$ of $\tilde{f}_{\eps_{m-1}}$ such that $\tilde{f}^j(\tilde{c}_{m-1})_{\eps}=f_\eps^{j}(c_m)$ for all $j$.
Then two cases may appear. Denote by $v$ the direction  at $\xg\in\P^1(\sH_{\eps_{m-1}})$ determined by $\tilde{c}_{m-1}$, and note it has an infinite orbit.
Then either $\tilde{f}_{\eps_{n-1}}^j(v)$ is not bad for all $j$, or $\tilde{f}_{\eps_{m-1}}^j(v)$ is bad for some maximal integer $j=j_0$.
In the former case, $f_{\eps_{n-1}}$ admits a critical point $c_{m-1}$ in the direction of $v$, and 
  $(f^{j}_{\eps_{m-1}}(c_{m-1}))_{\eps_m}=f_{\eps_m}^{j}(c_m)~,$ for any $j$. 
 In the latter case, by~\cite[Lemma~4.2]{zbMATH06455745KJ15} the open ball determined by the direction $\tilde{f}_{\eps_{m-1}}^{j_0}(v)$ contains a critical point $c_{m-1}$
such that 
  \[  (f^{j}_{\eps_{m-1}}(c_{m-1}))_{\eps_m}=f_{\eps_m}^{j+j_0}(c_m)~,\] 
  for all $j$.
  
  We  apply the same argument to $c_{m-1}$, and we get a critical point $c_{m-2}$ of 
  $f_{\eps_{m-2}}$ such that $(f^{j}_{\eps_{m-2}}(c_{m-2}))_{\eps_{m-1}}=f_{\eps_{m-1}}^{j+j_1}(c_{m-1})$. 
 The proof follows by iterating this construction $m$ times. 
\end{proof}

\begin{proof}[Proof of Theorem~\ref{thm:non-PCF}]
Let $(M_*,\eps)$ be any cycle of $g$-rescalings of period $m$ adapted to $f$ such that $g:=(M_0\cdot f^m\cdot M_0^{-1})_\eps$ is not PCF.
If $\deg(g)\ge2$, then $g$ has a critical point $c$ having an infinite orbit. By Lemma~\ref{lem:assocriti}, there is a critical point of $M_0\cdot f^m\cdot M_0^{-1}$ in $\sH_0$
having an infinite orbit at all scales $\eta \le \eps$. This implies that none of the rational maps $(M_0\cdot f^m\cdot M_0^{-1})_\eta$ 
is PCF for all $\eta \le \eps$. In particular, the segment in $\cR(f)$ joining the fundamental vertex to $(M_*,\eps)$ contains no PCF vertex. 

When $\deg(g)=1$, consider the cycle of $g$-rescalings $(M_*,\eps')$ with $\eps'<\eps$ such that $(M_*,\eps')$ and $(M_*,\eps)$ are consecutive edges in $\cR(f)$. 
Then $g' =(M_0\cdot f^m\cdot M_0^{-1})_{\eps'}$ has degree at least $2$ and contains an indifferent periodic type II Julia point. This implies the existence of 
a critical point for $g'$ whose orbit is infinite so that $g'$ is not PCF. We conclude by applying the previous arguments to $g'$.
\end{proof}

We proceed and prove a bound on the size of $\cR'(f)$. 
To state this bound properly, recall that $\sigma_f$ is a tree isomorphism
fixing the fundamental vertex so that the quotient space 
$\cR(f)/\langle\sigma_f\rangle$ is again a simplicial tree in the sense of \S\ref{sec: simplicialtree}. 
Also $\sigma_f$ preserves $\cR'(f)$, and $\cR'(f)/\langle\sigma_f\rangle$ is 
a simplicial tree as well. 

Recall that an anti-chain is a subset for which no two distinct elements are  comparable.
\begin{thm}\label{thm:anti-chain}
 Any anti-chain in $\cR'(f)/\langle\sigma_f\rangle$ has cardinality at most $2d-2$.
\end{thm}

\begin{cor}
\label{cor:generaltree}
The number of non-PCF vertices in $\cR(f)$ is finite, and
its cardinality in $\cR(f)/\langle\sigma_f\rangle$ is at most \(4d^2-8d+5 \).
\end{cor}

We refer to \S\ref{sec:example} for pictures of the tree 
$\cR(f)/\langle\sigma_f\rangle$ in some explicit examples, e.g., for quadratic rational maps. 

\begin{rmk}
For any $d\ge 2$, Luo~\cite[Proposition~7.2]{Luo222}
has constructed a sequence of degree $d$ rational maps admitting infinitely many non-monomial, pairwise non-equivalent Kiwi rescalings. 
\end{rmk}

\begin{proof}
 By Corollary~\ref{cor:boundlength}, the maximal length of a maximal chain is at most $2d-2$. By Theorem~\ref{thm:anti-chain},  there are at most $2d-2$ non-PCF vertices which are not comparable. The total number of non-PCF vertices modulo $\sigma_f$ is thus bounded from above by  
$(2d - 2)(2d - 2) + 1 = 4d^2 - 8d + 5.$
\end{proof}

The rest of this section is devoted to the proof of Theorem~\ref{thm:anti-chain}.
We suppose that $\eps_\star:= \eps(f)=\eps_\star(f)$ so that $f_{\eps_\star}$ does not have potential good reduction, and
we pick an anti-chain $W\subset \cR'(f)$. 

We shall say that a critical point $c$ of $f$ is attached to a cycle of $g$-rescalings $(M_*,\eps)\in \cR'(f)$ of period $l$
if $\{(M_i\cdot f^{ln}(c))_\eps\}$ forms an infinite set in $\P^1(\sH_\eps)$ for some $i\in \Z$.

We rely on the next proposition.
\begin{prop}\label{prop:attach-crit}
Let  $(M_*,\eps)$ be any non-PCF cycle of $g$-rescalings adapted to $f$. Then there exists a
critical point $c$ of $f$ attached to $(M_*,\eps)$.
\end{prop}

\begin{proof}[Proof of Theorem~\ref{thm:anti-chain}]
    In view of the previous proposition, and since $f$ admits at most $2d-2$ critical points, it is sufficient
to show the following statement. Suppose a critical point $c$ is attached to two non-PCF vertices $v,v' \in \cR'(f)$ which are not in the same orbit of $\sigma_f$. Then
we have either $v\leq v'$ or $v'\leq v$. 

 Let \( c\in\P^1(\sH_0)\) be a critical point of $f$ that is attached to two distinct vertices $v, v'$.   Choose representatives \((M_*, \eta)\) and  \((L_*, \eta')\)  defining $v$ and $v'$ respectively, and let \( m \) be the least common multiple  of their periods. Without loss of generality, we may suppose that both sequences  
\[
\left\{ \left(M_{0} \cdot f^{jm}(c)\right)_\eta \right\}_{j \geq 0} \quad \text{and} \quad \left\{ \left(L_{0} \cdot f^{jm}(c)\right)_{\eta'} \right\}_{j \geq 0}
\]  
are infinite, and \( \eta \le \eta' \). By the latter condition,  the set  
$\{ \left(L_{0} \cdot f^{jm}(c)\right)_\eta \}_{j \geq 0}$  is also infinite. Therefore, for $j$ large enough,  $(L_{0} \cdot f^{jm}(c))_\eta$ does not belong to the support of the hole divisor 
$D_{\eta}(M_{0} \cdot L_{0}^{-1})$.
By Proposition~\ref{prop: compo}, we have
\begin{align*}
 (M_{0} \cdot f^{jm}(c))_\eta=(M_{0} \cdot L_{0}^{-1})_\eta\,  (L_{0} \cdot f^{jm}(c))_\eta.
\end{align*}
Hence $(M_{0} \cdot L_{0}^{-1})_\eta$ is not a constant map, i.e., $\deg(M_{0} \cdot L_{0}^{-1})_\eta= 1$.
This implies $(M_*,\eta)$ and $(L_*,\eta)$ to be equivalent so that 
$v' \geq v.$
\end{proof}

 \begin{proof}[Proof of Proposition~\ref{prop:attach-crit}]
Let $(M_*,\eps)$ be any cycle of $g$-rescalings of period $l$ adapted to $f$. Suppose that $M_0=\id$ and 
$f ^l_\eps$ is not PCF.  Then there exists a critical point \(c_\eps\) of \(f^l_\eps\) with an infinite orbit.  
By Lemma~\ref{lem:assocriti}, there exist a  critical point \(c'\) of \(f^l\)  and an integer $J$ such that  $(f^{jl}(c'))_\eps=f_\eps^{jl+J}(c_\eps)$ for all $j$.
In particular, we get an index $0\le i\le l-1$ such that
$c=f^i(c')$ is critical for $f$.  

By (D1) and Proposition~\ref{prop: compo}, the map
\[
(M_i\cdot f^i)_{\varepsilon}
=
(M_i\cdot f^i\cdot M_0^{-1})_{\varepsilon}
=
(M_{i}\cdot f\cdot M_{i-1}^{-1})_{\varepsilon}\circ
 \cdots 
\circ(M_{1}\cdot f\cdot M_0^{-1})_{\varepsilon}.
\]
is not constant. 
For all sufficiently large \(j\in\mathbb{N}\), we have $
(f^{jl}(c'))_{\varepsilon} \notin \bigl|D_{\varepsilon}(M_i\cdot f^{i})\bigr|.$
It follows that
\[
(M_i \cdot f^{jl}(c))_{\varepsilon}=(M_i \cdot f^{jl+i}(c'))_{\varepsilon}
=
(M_i \cdot f^{i})_{\varepsilon}
\, (f^{jl}(c'))_{\varepsilon},
\]
which forms an infinite set as $j$ varies.
Therefore, \(c\) is attached to \((M_*,\varepsilon)\).
 \end{proof}

 \subsection{Proof of Theorem~\ref{thmint:1}}
Let us explain how the results of this section and the preceding ones imply Theorem~\ref{thmint:1}.

Recall that $\cR_+(f)$ is defined in the introduction as  the set of equivalence classes of cycles of $g$-rescalings adapted to $f$
whose limit has degree $\ge2$. It is a subset of the poset $\cR(f)$ containing the fundamental limit so that Theorem~\ref{thm:boundscales} already implies (2).
We claim that $\cR_+(f)$ is a subtree of $\cR(f)$. Note that this claim and Theorem~\ref{thm:tree-rf} (resp. Theorems~\ref{thm:non-PCF} and~\ref{thm:anti-chain}) implies (1) (resp. (3)).

Suppose that $(M,\eps)$ is a $g$-rescaling adapted to $f$. Assume its limit $(M\cdot f\cdot M^{-1})_\eps$ has degree $\ge 2$. 
Then Theorem~\ref{thm:chain} implies that for all $g$-rescalings 
$(M,\eta)$ with $\eta<\eps$, the limit has degree $\ge \deg(M\cdot f\cdot M^{-1})_\eps$. This implies our claim.

\section{Examples}\label{sec:example}

In this last section, we review some examples 
illustrating our theory of $g$-rescalings, and prove Theorem~\ref{thm:int3}.

\subsection{Quadratic rational maps}
Let $(f_n)$ be a sequence of complex quadratic rational maps. 
We shall always assume that $\eps_\star:=\eps(f)=\eps_{\star}(f)$. 
To simplify notation we write $f_\star= f_{\eps(f)}$, denote by $J_\star$
its Julia set,  and by $\H_\star$
the hyperbolic space in $\P^{1,\an}_{\sH_{\star}}$ with the convention
$\H_\star=\emptyset$ when $\eps_\star =(1)$.

\begin{thm}
\label{thm: quadratic}
Let $(f_n)_{n\in\mathbb{N}}$ be a sequence of quadratic rational maps as above. 
Then exactly one of the following mutually exclusive cases occurs:
\begin{enumerate}
    \item One has $J_\star \cap \H_\star=\emptyset$, and $\cR(f)$ is reduced to the fundamental vertex $v_{\star}$.
\[
\begin{tikzcd}
(\eps_\star, 2)
\end{tikzcd}
\]
    \item The set $J_\star \cap \H_\star$ is the grand orbit of an indifferent periodic cycle, and the quotient tree
    $\mathcal{R}(f)/\langle\sigma_f\rangle$ is a segment joining the fundamental vertex to the baby $g$-rescaling associated with this cycle. 
\[ \begin{tikzcd}
(\eps_\star, 2) \ar[r] &(\eta,1)
\end{tikzcd}\]
\item The set  $J_\star \cap \H_\star$ is the grand orbit of a repelling periodic point $p$ of period at least $2$. 
    \begin{itemize}
    \item[(3a)] The tree $\mathcal{R}(f)/\langle\sigma_f\rangle$ is a segment joining the fundamental vertex to the baby $g$-rescaling associated with $p$. The non-fundamental
    rescaling is at the trivial scale, and its limit is a complex quadratic rational map with a multiple (i.e., parabolic) fixed point.
\[
\begin{tikzcd}
(\eps_\star, 2) \arrow{r}{p} &((1),2)
\end{tikzcd}
\]
    \item[(3b)] The tree $\mathcal{R}(f)/\langle\sigma_f\rangle$ is totally ordered, and contains three points: the fundamental vertex; the baby $g$-rescaling associated with $p$
    at a scale $\eta < (1)$ and whose limit is a quadratic rational map with a multiple fixed point; and a $g$-rescaling at the trivial scale whose limit is a translation. 
\[
\begin{tikzcd}
(\eps_\star, 2) \arrow{r}{p} &(\eta,2) \ar[r]& ((1),1)
\end{tikzcd}
\] 
\end{itemize}
\item The set  $J_\star\cap \H_\star$ contains two distinct repelling periodic cycles $p,q$, and  $\cR(f)/\langle\sigma_f\rangle$ has three vertices: the fundamental vertex and 
the two $g$-rescalings associated with $p$ and $q$. One of the rescaling is at the trivial scale, and the limit is a complex quadratic rational map with a multiple fixed point $\alpha$, 
and a critical point eventually mapped to $\alpha$. 
The other limit is a quadratic polynomial at a scale $\eps_\star < \eta \le (1)$. 
\[
\begin{tikzcd}
(\eps_\star, 2) \arrow{r}{q} \arrow{rd}{p}  &(\eta,2)
\\&((1),2)
\end{tikzcd}
\]
\end{enumerate}
\end{thm}
In the statement above, we depicted a geometric realization of the graph $\cR(f)/\langle \sigma_f\rangle$ in each case. 
Each vertex is marked by a pair $(\eps,\delta)$ consisting 
of the scale of the $g$-rescaling and the degree of its limit. 

\medskip

We shall need the following well-known results on quadratic rational maps having a parabolic fixed point. 
\begin{lemma}
\label{lem:easyform}
Over any field of characteristic $0$, any quadratic rational map $f$ having a parabolic fixed point is conjugate to a map of the form
\[
z \longmapsto a+\frac{z^2}{z-1},
\]
for some $a\in k$. When $k$ is non-Archimedean with residue characteristic $0$, then
$f$ does not have potential good reduction iff $|a|>1$. When it is the case,
$\zeta(0,|a|)$ is an indifferent fixed Julia point whose tangent map is a translation, and  the iterates of both critical points 
converge to $\zeta(0,|a|)$.
\end{lemma}

\begin{proof}
    The normal form is obtained by putting the multiple fixed point at $\infty$, and assuming that $1$ is a preimage of $\infty$, and $0$ is a critical point with infinite forward orbit.
It is clear that $f$ has good reduction when $|a|\le 1$. Suppose $|a|>1$. 
Then $\widetilde{(M\cdot f \cdot M^{-1})}_\eps = 1+ z$
with $M(z)=a^{-1}z$ so that $\zeta(0,|a|)$ is fixed by $f$ and the direction pointing to $1$ is bad. 
It follows that $\zeta(0,|a|)$ belongs to the Julia set, and $f_\eps$ does not have potential good reduction. 
A direct computation shows that $f'(z)=\frac{z^2-2z}{(z-1)^2}$, hence the two critical points are given by $c_1=0$ and $c_2=2$. 
Observe that $|x_1|=|x_2|=|a|>1$ with $x_i:= f(c_i)$. 
By induction, we get 
$|f^k(x_i)|= |a|$ and $|f^k(x_i)-x_i-ka|\leq 1$
for all $k\geq 1$, 
which concludes the proof. 
\end{proof}

We now give explicit examples of all cases appearing in the previous theorem. 

\subsubsection{Case (1)}
\label{sec: QR1}
Pick any scale $\eps < (1)$, and set
\[
f_n(z) = z^2 + e^{\frac{1}{\eps_n}}.
\]
Observe that $(f_n)$ arises from the meromorphic family $f_t(z)=z^2+1/t$ and that the quadratic polynomial map $f_{\na}$ defined over $\C((t))$ does not
have any type II periodic point. From Theorem~\ref{thm:family}, we get $\eps(f)=\eps$ and we fall into case (1).

\subsubsection{Case (2)}
\label{sec: QR2}
Pick any scales $\eps < \eta \le (1)$, and consider $a ,b \in \sH_0$ with
$b_n = \exp(1/\eta_n)$ and  $a_n = \exp(-1/\eps_n)$. 
Define \[
f_n(z) = 1+ b_n z + \frac{a_n}{z-1}.\]
Observe that  $|a_\eps|<1$ and $|b_\eps|=1$, so that 
$\xg$ is fixed with local degree $1$ 
and $1$ is the (unique) bad direction at $\xg$ with infinite orbit. 
Hence $x_g$ belongs to the Julia set of $f_\eps$.
It follows that
$f_{\eps}(z) = b_\eps z + \frac{a_\eps}{z-1}$
does not have potential good reduction,
and the baby $g$-rescaling associated with $\xg$
and the direction $1$ is equal to $(\id,\eta)$. Its limit is then $z \mapsto 1+ b_\eta z$. 

\subsubsection{Case (3a)}
\label{sec:QR3}
Pick any $\eps < (1)$, set $a_n=\exp(-1/\eps_n)$ and
\[
f_n(z) = -a_n - \frac{1 + a_n^2}{z} + \frac{a_n}{z^2}.
\]
Then $(f_n)$ arises from a meromorphic family so that $(\id,\eps)$ is the fundamental $g$-rescaling.
Note that  $|a_\eps| < 1$, so that $f_\eps$ 
maps the point $x_1 = \zeta(0,e^{-1})$ to $x_2 = \zeta(0,e)$ and vice-versa. 
A direct computation yields
\[
\widetilde{(M \cdot f^2 \cdot  M^{-1})_\eps}(z) = -1 + \frac{z^2}{z-1},
\]
with $M(z)=a^{-1} z$ hence the cycle $\{x_1,x_2\}$ is repelling. Its associated $g$-rescaling is
$(M,(1))$ which is adapted to $f^2$, and its limit 
is $z \longmapsto -1 + \frac{z^2}{z-1}$. 
For this map, $\infty$ is a parabolic fixed point, and both critical points $0$ and $2$ have infinite forward orbits.

\subsubsection{Case  (3b)}
\label{sec: QR4}
This example was described by Kiwi.
Pick any scales  $\eps < \eta < (1)$,  set
$a_n = \exp(-1/\eps_n)$, $b_n = \exp(1/\eta_n)$, and
\[
f_n(z) = b_n a_n - \frac{1 + a_n^2}{z} + \frac{a_n}{z^2}.
\]
As in the previous case, $x_1 = \zeta(0,e^{-1})$ to $x_2 = \zeta(0,e)$ forms a repelling $2$-cycle
for $f_\eps$, as 
\[
\widetilde{(M \cdot f^2 \cdot  M^{-1})_\eps}(z) = \tilde{b}_\eps + \frac{z^2}{z-1},
\]
with $M(z)=a^{-1} z$. It follows that $(\id, \eps)$ is the fundamental $g$-rescaling. 
Note that  $|b|_\eta>1$, and $(M \cdot f^2 \cdot  M^{-1})_\eta$
has a rigid repelling fixed point at $z = \frac{b_{\eta}}{b_{\eta} + 1}$ with multiplier $-b_{\eta}^2 - b_{\eta}$. 
It follows that  $(M,\eta)$ is the baby $g$-rescaling
associated with the cycle $\{x_1,x_2\}$. Its limit is $b_\eta + \frac{z^2}{z-1}$
which admits a  type II Julia fixed point $\zeta(0,|b_\eta|)$ whose tangent map is a translation. 
And $(L,(1))$ is an adapted $g$-rescaling with
$L(z) = b^{-1}a^{-1} z$.

\begin{rmk}\label{rem:noKiwi}
Observe that  these examples admit  a non-fundamental $g$-rescaling whose limit has degree $\ge2$, but no Kiwi rescalings. 
\end{rmk}

\subsubsection{Case (4)}
\label{sec:QR5}
This example was also described by Kiwi.
Pick any scales  $\eps < \eta \le (1)$,  set
$a_n = \exp(-1/\eps_n)$, $b_n = \exp(1/\eta_n)$, and
\[
f_n(z) = a_n - b_n a_n^{5} - \frac{1 + a_n^2}{z} + \frac{a_n}{z^2}.
\]
Direct computations show
the existence of a two cycle: $f_\eps(y_0)=y_1$ and $f_\eps(y_1)=y_0$
with $y_0=\zeta(0,e)$ and $y_1 =\zeta(0,e^{-1})$. The tangent map
is given by $T_{\zeta(0,e)}f^2= z/(z-1)^2$ which has a parabolic fixed point at $0$. 

We also have a repelling $3$-cycle $f_\eps(x_0) = x_1$, $f_\eps(x_1)=x_2$ and $f_\eps(x_2)=x_3$ 
with  $x_0 = \zeta(0, e^{-3})$, 
$x_1= \zeta(0, e^5)$, and $x_2 = \zeta(a_\eps, e^{-5})$. 
Define $M_0(z) = a^{-3} z$,  
$M_1(z) = a^{5} z$, and $  
M_2(z) = a^{-5} z - a^{-4}.$
We get 
$\widetilde{(M_1 \cdot f_{\eps} \cdot M_0^{-1})}(z) = 1/z^2$, $
\widetilde{(M_2 \cdot f_{\eps} \cdot M_1^{-1})}(z) = -\tilde{b} - 1/z$, and $ 
\widetilde{(M_0 \cdot f_{\eps} \cdot M_2^{-1})}(z) = -z$
so that
\[\widetilde{(M_0 \cdot f_{\eps}^3 \cdot M_0^{-1})}(z) = z^2 + \tilde{b}.\] 
We conclude that $(\id,\eps)$ is the fundamental $g$-rescaling, that
$(M_0,\eta)$ is a $g$-rescaling adapted to $(f^3)$ with a quadratic polynomial limit, and 
$(L,(1))$ is a $g$-rescaling adapted to $(f^2)$ with a complex quadratic rational map having parabolic fixed point
with $L(z)=a z$.

\begin{rmk}
The proof of Theorem~\ref{thm: quadratic} below is based on the description of the dynamics of a 
quadratic rational map defined over the field of complex Puiseux series given by Kiwi in~\cite{zbMATH06455745KJ15}. 
Note however that most results in op. cit. including Theorem 1, Proposition 4.1, and  Lemma 5.1
remain valid (with the same proof) over any field with residue characteristic $0$.
\end{rmk}
\begin{proof}[Proof of Theorem~\ref{thm: quadratic}]
Recall that $(f_n)$ is a sequence of quadratic rational maps such that $\eps_\star:=\eps(f)=\eps_\star(f)$. 
Let $f_\star$ be the fundamental limit. If $\eps_\star =(1)$, then we are in Case (1). 
Assume from now on that $\eps_\star<(1)$. If $J_\star \cap \H_\star=\emptyset$, then 
$f_\star$ has no type II periodic points, hence there are no non-fundamental $g$-rescaling and we are again in Case (1). 

\smallskip

Suppose that $J_\star \cap \H_\star$ contains an indifferent type II periodic cycle $\cO=\{x, \cdots, f^{l-1}(x)\}$ of length $l$. 
By \cite[Theorem~1]{zbMATH06455745KJ15}, $J_\star \cap \H_\star$ coincides with the grand orbit of $\cO$.
By \cite[Proposition~4.2]{zbMATH06455745KJ15}, $\cO$ is the boundary of a fixed Rivera domain $U$. 
We may suppose that the critical locus is contained in the component of $U$ containing $x$. 
It follows that the direction at $x$ pointing to the critical locus, is 
the unique bad direction having an infinite orbit under $f^l$ (the other is pointing to $U$).
By Theorem~\ref{thm:baby1}, we are in Case (2).

\smallskip

Suppose that $J_\star \cap \H_\star$ contains a unique repelling periodic cycle $\cO$ of period $l$. By \cite[Theorem~1(3)]{zbMATH06455745KJ15}, 
$J_\star \cap \H_\star$ is the grand orbit of $\cO$, 
and for any $x \in \cO$, the  
reduction map $T_x f_\star^l$ is a quadratic rational map with a multiple fixed point.
By Theorem~\ref{thm:baby2} and Lemma~\ref{lem:goingtol},
the point $x$ induces a cycle of $g$-rescalings $(M_*, \eta)$ adapted to $f$, and 
its rescaling limit is a quadratic rational map with a multiple fixed point. Note that $\sigma_f^j(M_*, \eta)$ is the $g$-rescaling associated with $f_{\star}^j(x)$
so that  all these $g$-rescalings are identified to a single point in the quotient tree $\cR(f)/\langle \sigma_f\rangle$. 

If $\eta = (1)$, then we are in Case (3a). If $\eta <(1)$,  
the limit 
$(M\cdot f^l\cdot M^{-1})_\eta$
does not have potential good reduction. By Lemma~\ref{lem:easyform},
it has a unique indifferent type II fixed Julia point 
whose tangent map in a translation. 
We are thus in Case (3b). 

\smallskip

Finally if none of the above cases occurs, then by \cite[Theorem~1]{zbMATH06455745KJ15},
$f_\star$ admits two  repelling periodic cycles $\cO$ and $\cO'$, and 
$J_\star \cap \H_\star$ is the union of their grand orbits. 
They give rise to two cycles of $g$-rescalings $(M_*, \eta)$ and $(M'_*, \eta')$. 
The limit associated with $M'$ is a quadratic polynomial, and with $M$ is a quadratic rational map $g$ possessing a multiple fixed point $p$.
Observe that $g$ has one critical point having an infinite orbit. We claim that the other critical point is eventually mapped to $p$. 
It follows from Lemma~\ref{lem:easyform} that $\eta =(1)$, and the proof is complete. 

In order to prove the claim, we need some more information on the dynamics of $f_\star$. 
It is proved in op. cit., that $\cO = \{x_0, \cdots, x_{l-1}\}$ where $f(x_{i-1})=x_{i\, \mathrm{mod}\, l}$, and 
$\deg_{x_0}(f_\star)=2$, and $\deg_{x_i}(f_\star)=1$ for all $i\neq 0$. 
Moreover there is a fixed (starlike) Rivera domain $U$ such that $\partial U= \cO$. 

There is no bad direction at $x_0$, and exactly one  $v_i\in Tx_i$
for $i\neq 0$: this is the one pointing
towards any point in $U$. It follows that $f^l_\star$ admits exactly two bad directions at $x_0$: 
the direction $v_0$ pointing to $U$, and another one $v$ which is mapped under $T_{x_0}f_\star$ to $v_1$. 

Recall that $x_0$ belongs to the geodesic joining the two critical points of $f_\star$. 
One critical direction is good for all iterates $f_\star^j$ and has infinite orbit. 
The other critical direction $w\in T_{x_0}f_{\star}$ contains a point in $\cO'$. 
It cannot be good for all iterates $f_\star^j$ since an analytic self-map of an open ball
cannot have a repelling type II periodic point. It follows that $w$ is eventually mapped
by $T_{x_0} f_\star^l$ to $v_0$. 
This implies our claim. 
\end{proof}

\subsection{Proof of Theorem~\ref{thm:int3}}

We proceed by contradiction and suppose that we can find some integer $l\in\N_+$ 
 and a sequence of quadratic rational maps $f_n$
having five periodic points $p^1_n, \cdots, p^5_n$  in distinct cycles such that $f_n^l(p^i_{n})=p_n^i$ and  $|df_n^l(p^i_n)| \le (1+1/n)^l$ for all $i\in\{1,\cdots ,5\}$. We thus get a quadratic rational map $f$ defined over $\sH_0$ which admits $5$ distinct periodic cycles containing $p^1,\cdots, p^5$  respectively 
such that $\mu(p^i):= |df^l(p^i)|_{(1)}^{1/l}
= \lim |df_n^l(p^i_n)|^{1/l} \le 1$ for all $i$.

Up to conjugating by a sequence of Möbius transformations, we may assume that $|\res(f_n)|=|\Res(f_n)|$ for all $n$.
If $(f_n)$ does not degenerate in the moduli space, then $f_{(1)}$ is a quadratic rational map defined over $\C$
with $5$ indifferent periodic cycles. This contradicts~\cite[Corollary~1]{Mitsu}.

If $(f_n)$ degenerates, then its fundamental limit $f_\star$ is defined over a non-Archimedean field, and 
the multiplier of each periodic point $p^i$ has norm $\le 1$. 
It follows that the Fatou component $U_i$ of $f_\star$ containing $p^i$ is a Rivera domain. 
Note that $\bigcup_i \partial U_i$ is a finite set of periodic type II Julia points. 

This excludes Case (1) of Theorem~\ref{thm: quadratic}.
In Case (2), there is a unique cycle $x, \cdots, f_\star^{l-1}(x)$ contained in $J_\star\cap \H_\star$, and 
$\deg(T_xf_\star^l)=1$. We may suppose that $x \in \partial U_i$ for all $i$.
However, since $x$ lies in the Julia set, 
 $T_xf_\star^q$ is not of finite order and has exactly two periodic cycles.
  This forces $f$ to contain
 at most $2$ indifferent type I periodic points which is a contradiction. 
 
 In the remaining cases, $f_\star$ admits  a unique Rivera domain $U$ which is fixed, see~\cite[Lemma 5.1]{zbMATH06455745KJ15}. Its boundary
 consist of a cycle $x, \cdots, f_\star^{l-1}(x)$ with $\deg(T_xf_\star^l)=2$, and the direction $v$ pointing to $U$ at $x$
 is a fixed point of multiplicity $2$. It follows from Rivera-Letelier's fixed point formula~\cite[Theorem~10.2]{zbMATH07045713benedetto}, that $U$
 contains $2$ fixed points, and no other periodic cycles. 
 Up to reordering, we may thus assume that $p^1, p^2$ and $p^3$ do not belong to $U$. 
 
 In Case (3), $x, \cdots, f_\star^{l-1}(x)$ is the unique periodic cycle in $J_\star\cap \H_\star$. 
 It follows that $U^1, U^2$ and $U^3$ are necessary open balls whose boundary belong to this cycle. 
 We may suppose that $\partial U^i=\{x\}$ for all $i$. Each direction $v_i$ at $x$ determined by $U_i$, or equivalently by $p^i$
 is periodic. The $g$-rescaling limit associated with $x$ is a quadratic rational map $g$
 defined at a scale $\eta > \eps_\star$. The direction $v$ is mapped in $\P^1(\sH_\eta)$ to a multiple fixed point, 
 and each $v_i$ to the type I periodic point $p^i_\eta$. By Lemma~\ref{lem:period}, the multiplier of $p^i_\eta$
 for $g$ is the image of the multiplier of $p^i$ for $f_\star$, hence has norm $\le 1$ in $\sH_\eta$. 
  In Case (3a), we have $\eta = (1)$, hence $g$ is a quadratic rational map having  one parabolic fixed point, 
  and three other indifferent periodic cycles, which is a contradiction. 
  In Case (3b), $\eta <(1)$ and $g$ is a quadratic rational map having a unique indifferent  type II periodic cycle in its Julia set. 
  By our analysis of Case (2), $g$ admits at most $2$ indifferent rigid periodic cycles, and we get a contradiction. 
  
Finally in Case (4), we have two type II periodic cycles  $x, \cdots, f_\star^{l-1}(x)$ and $y, \cdots, f_\star^{m-1}(y)$ in the Julia set. 
We may assume that $\deg(T_xf_\star)=\deg(T_yf_\star)=2$ and $T_xf^l_\star$ is a quadratic rational map having a parabolic fixed point;
 and $T_yf^m_\star$ is a quadratic polynomial. 

\begin{lemma}\label{lem:1233}
There exists at most
one cycle of Rivera domains whose closure does not intersect the cycle $x, \cdots, f_\star^{l-1}(x)$. 
\end{lemma}

We may thus suppose that $U_1$ and $U_2$ are Rivera domains lying in different cycles, distinct from 
$U$ and having $x$ in their boundaries. The orbit of $U_1$ and $U_2$ does not intersect the critical set, 
hence the closure of both $U_1$ and $U_2$ is disjoint from the cycle $y, \cdots, f_\star^{m-1}(y)$. 
It follows that $U_1$ and $U_2$ are open balls whose directions are periodic. 

Consider the $g$-rescaling limit associated with $x$, and recall its scale is trivial $(1)$. 
It is thus a complex quadratic rational map having the 
image of  the direction pointing to 
$U$
as a multiple fixed point. Both directions associated with $U_1 $ and $U_2$
descend to $\P^1(\sH_{(1)})$ as indifferent periodic cycles, which again contradicts
Shishikura's theorem. 

The proof is complete.

\begin{proof}[Proof of Lemma~\ref{lem:1233}]
Suppose $V$ is a Rivera domain 
whose boundary  does not intersect the cycle $x, \cdots, f_\star^{l-1}(x)$. We may suppose that  $y \in \partial V$. 
Recall that the direction $v_x$ pointing to $x$ at $y$ is the unique bad direction which is totally invariant under $T_y f^m$. 
Since $V$ is a Rivera domain, it does not intersect the critical set so that it does not define $v_x$. 
We conclude that $V$ is an open ball, and that the direction $v_1\in Ty$ determined by $V$ 
is a periodic cycle. 
Let $g$ be the $g$-rescaling limit associated with $y$. It is a quadratic polynomial defined over $\sH_\eta$
 for some $\eta >\eps_\star$. 
 If $\eta <(1)$, then $g$ does not have potential good reduction, hence has only repelling periodic cycle which is absurd. 
 It follows that  $\eta =(1)$, and $g$ contains at most one indifferent periodic cycle. 
 This forbids the existence of another Rivera domain satisfying the assumption of the lemma.  
 \end{proof}

\subsection{Examples with a maximal number of adapted scales}
\label{sec: optimal}
Let us fix  any integer $d \ge 2$ and $2d-1$ scales
\[
\eps_{2d-2} < \cdots < \eps_0 = (1). 
\]
We claim that there exists a sequence $(f_{n})$ of degree $d$ rational maps for which each of the scales $\eps_i$ arise as an adapted scale. In particular, the bound in Theorem~\ref{thm:boundscales} is optimal.

\smallskip

Our construction proceeds by induction on the degree. More precisely, for each \( k\ge 2 \), we construct a sequence \( (f_{k,n})\) of degree $k$ rational maps satisfying the following properties:
\begin{enumerate}
\item $\infty$ is a non critical fixed point;
\item \( (\id, \eps_{2k-2}) \) is the fundamental \( g \)-rescaling of \( (f_{k,n}) \);
\item for each $i\in\{0,\cdots, 2k-2\}$, \( f_{k} \) admits an adapted cycle of $g$-rescalings 
at the scale $\eps_i$ whose limit has degree at least $2$.
\end{enumerate}

We begin with the case \( k = 2 \). Set
$a_n = \exp(-1/\eps_{2,n})$,  $
b_n = \exp(-1/\eps_{1,n})$, and 
define $
g_n(z)
= a_n
- b_n a_n^5
- \frac{1 + a_n^2}{z}
+ \frac{a_n}{z^2}.$
By the discussion in \S\ref{sec:QR5}, $(\id, \eps_2)$ is the fundamental scale, and 
\(g\) admits two other cycles of $g$-rescalings with adapted scales
$\eps_1$ and  $\eps_0 = (1)$ respectively. 
All corresponding rescaling limits have degree \(2\).

Pick any non critical fixed point $p$ of $g$, and a sequence of Möbius transformations $M$
such that $M(p)=\infty$, and $M_{\eps_2}\in \PGL(2,\sH^\circ_{\eps_2})$. Then 
$f_2:= M\cdot g \cdot M^{-1}$ satisfies all required properties.

\smallskip

Now suppose that the sequence \( (f_{k,n}) \) has already been constructed satisfying properties (1)--(3).
Write
\[
f_{k,n}(z)
=
\frac{a_{k,n} z^k + \cdots + a_{0,n}}
{b_{k-1,n} z^{k-1} + \cdots + b_{0,n}}
\]
in normalized form so that
$\max\{ |a_{i,n}|, |b_{j,n}| \} = 1.$
Define
\[
f_{k+1,n}(z)
=
\frac{\alpha z^{k+1} + a_{k,n} z^k + \cdots + a_{0,n}}
{\alpha \beta z^k + b_{k-1,n} z^{k-1} + \cdots + b_{0,n}},
\]
where  $\alpha_n = \exp\big(-\eps_{2k,n}^{-1}\big)$ and $
\beta_n = \exp\big(\eps_{2k-1,n}^{-1}\big).$ Note that 
$|\alpha|_{\eps_{2k-2}} = |\alpha\beta|_{\eps_{2k-2}}=0$
as $\eps_{2k} < \eps_{2k-1}< \eps_{2k-2}$,
and
\( (\id, \eps_{2k-2}) \) is a \( g \)-rescaling adapted to \( f_{k+1} \) whose limit is the degree $k$ rational map
$f_{k,\eps_{2k-2}}$.
 By the induction step and Theorem~\ref{thm:upper-tree}, all scales $\eps_i$ with $i=0,\cdots, 2k-2$,
 are adapted scales for $f_{k+1}$, and their associated limits have degree at least $2$. 

Since $\eps_{2k}$ is the largest scale $\eps$ for which $\deg(f_\eps)=k+1$, it follows
from Lemma~\ref{lem:adapt-cool} that  \( (\id, \eps_{2k}) \) is the fundamental \( g \)-rescaling of \( f_{k+1} \). 

Set $M(z)=\alpha z$ so that 
\[
(M\cdot f_{k+1}\cdot M^{-1})_{\eps_{2k-1}}(z)= \frac{z^{2}+a_kz}{\beta z+b_{k-1}}.
\]
Note that \( |\beta_k^{-1}|_{\varepsilon_{2k}} > 1 \), which is the multiplier of \((M\cdot f_{k+1}\cdot M^{-1})_{\eps_{2k-1}} \) at \( \infty \).
It follows that $(M,\eps_{2k-1})$ is adapted to $f_{k+1}$ and its limit has degree $2$,
which concludes the proof of the induction step. 

Iterating this construction yields the desired sequence \( (f_{d,n}) \). 

\medskip

Observe that all the $g$-rescaling limits constructed above are of period one. 
Accordingly, the tree $\cR(f)$ of $g$-rescalings
has the following geometric realization. 
As above, each vertex is marked by a pair
$(\eps,\delta)$ where $\eps$ denotes the scales of the $g$-rescaling and $\delta$
the degree of its associated limit. 
The fundamental $g$-rescaling corresponds to the unique vertex of degree $d$.

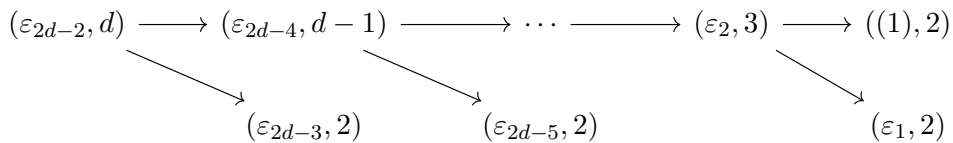
\[
\begin{tikzcd}
(\eps_{2d-2},d) \ar[rd] \ar[r]&  (\eps_{2d-4}, d-1) \ar[rd] \ar[r] & \cdots \ar[r]&(\eps_2,3) \ar[rd] \ar[r] & ((1),2)
\\
& (\eps_{2d-3},2) & (\eps_{2d-5},2) & & (\eps_1,2)
\end{tikzcd}
\]

\subsection{Two bad directions defining different baby $g$-rescalings}
\label{sec:diffbad}
Let 
\[
f_n(z)=2z+\frac{e^{-n}}{(z-\frac{1}{n})(z-1)}.
\]
Let  $\eps=(1/n)$, and $a,b \in \sH_{\eps}$ be represented by the sequences $(e^{-n})$ and $(1/n)$, so that  $|a|<1$ and $|b|=1$.
The associated map at scale $\eps$ is
\[
f_{\eps}(z)=2z+\frac{az}{(z-b)(z-1)},
\]
whose reduction is equal to  $
\tilde{f}_{\eps}(z)=2z.$ Note that $b$ and $1$ define two bad directions at $\xg$ with infinite forward orbit.
Therefore, $\xg$ belongs to the Julia set of $f_{\eps}$. In particular, $f_{\eps}$ does not have potential good reduction, and $\eps$ is the fundamental scale of $f$ by Theorem~\ref{thm:unique-rescaling}.

Set $M_n(z)=n z$. Then
\[
(M\cdot f \cdot M^{-1})_{(1)}(z)=2z,
\]
and $M(b)_{(1)}=1$ has an infinite forward orbit under $(M\cdot f \cdot M^{-1})_{(1)}$. It follows that $(M,(1))$ defines a $g$-rescaling adapted to $(f_n)$, induced by the bad direction $b$. Similarly, $(\mathrm{id},(1))$ defines a $g$-rescaling induced by the bad direction $1$. However, $(M,(1))$ and $(\mathrm{id},(1))$ are not equivalent.

\bibliographystyle{alpha}
\bibliography{ref}

\end{document}